\documentclass{article}

\usepackage[nonatbib,preprint]{neurips_2023}

\usepackage[utf8]{inputenc} 
\usepackage[T1]{fontenc}    
\usepackage{url}            
\usepackage{booktabs}       
\usepackage{amsfonts,amsthm,amsmath,amssymb}       
\usepackage{nicefrac}       
\usepackage{microtype}      
\usepackage{xcolor}         
\usepackage{optidef}
\usepackage{subcaption}
\usepackage{cite}
\DeclareMathOperator{\Clip}{Clip}
\DeclareMathOperator{\med}{med}
\newtheorem{thm}{Theorem}
\newtheorem{ass}{Assumption}
\newtheorem{defi}{Definition}
\newtheorem{lem}{Lemma}
\newtheorem{rmk}{Remark}

\title{Robust Nonparametric Regression under Poisoning Attack}

\author{%
  Puning Zhao, Zhiguo Wan\\
  Zhejiang Lab\\
  Hangzhou, Zhejiang, China\\
  \texttt{\{pnzhao,wanzhiguo\}@zhejianglab.com} \\
}

\begin{document}

\maketitle

\begin{abstract}
	This paper studies robust nonparametric regression, in which an adversarial attacker can modify the values of up to $q$ samples from a training dataset of size $N$. Our initial solution is an M-estimator based on Huber loss minimization. Compared with simple kernel regression, i.e. the Nadaraya-Watson estimator, this method can significantly weaken the impact of malicious samples on the regression performance. We provide the convergence rate as well as the corresponding minimax lower bound. The result shows that, with proper bandwidth selection, $\ell_\infty$ error is minimax optimal. The $\ell_2$ error is optimal with relatively small $q$, but is suboptimal with larger $q$. The reason is that this estimator is vulnerable if there are many attacked samples concentrating in a small region. To address this issue, we propose a correction method by projecting the initial estimate to the space of Lipschitz functions. The final estimate is nearly minimax optimal for arbitrary $q$, up to a $\ln N$ factor.
	
\end{abstract}

\section{Introduction}
In the era of big data, it is common for some samples to be corrupted due to various reasons, such as transmission errors, system malfunctions, malicious attacks, etc. The values of these samples may be altered in any way, rendering many traditional machine learning techniques less effective. Consequently, evaluating the effects of these corrupted samples, and making corresponding robust strategies, have become critical tasks in the research community \cite{natarajan2013learning,van2017theory,song2022learning}. 

Among all types of data contamination, adversarial attack is of particular interest in recent years \cite{biggio2012poisoning,xiao2015feature,jagielski2018manipulating,szegedy2014intriguing,goodfellow2015explaining,madrytowards,mao2019metric}, in which there exists a malicious adversary who aims at deteriorating our model performance. With this goal, the attacker alters the values of some samples using a carefully designed strategy. 
To cope with these attacks, robust statistics comes into being, which has been widely discussed in existing literatures \cite{huber1981robust,maronna2019robust}. Several commonly used methods are trimmed mean, median-of-means and $M$-estimators. In recent years, many new methods are proposed for high dimensional problems with optimal statistical rates. These methods are summarized in  \cite{steinhardt2018robust,diakonikolas2019recent,diakonikolas2023algorithmic}. For example,  \cite{diakonikolas2016robust,diakonikolas2017being,hopkins2018mixture,cheng2019faster} have solved some basic problems such as mean and covariance estimation. The idea of these research can then be used in machine learning problems with poisoning attack, which means that some training samples are modified by adversaries. \cite{bakshi2021robust,diakonikolas2019efficient} designed some robust methods for linear regression. \cite{diakonikolas2019sever,steinhardt2017certified} proposed a meta algorithm for robust learning with parametric models. There are also several other works that focus on general robust empirical risk minimization problems \cite{prasad2020robust,jambulapati2021robust}. 

Despite these previous works toward robust learning problems, most of them focus on parametric models. However, for nonparametric methods such as kernel \cite{nadaraya1964estimating} and k nearest neighbor estimator, defense strategies against poisoning attack still need further exploration \cite{salibian2022robust}.  Actually, designing robust techniques is indeed more challenging for nonparametric methods than parametric one. For parametric models, the parameters are estimated using full dataset, while nonparametric methods have to rely on local training data around the query point. Even if the ratio of attacked samples among the whole dataset is small, the local anomaly ratio in the neighborhood of the query point can be large. As a result, the estimated function value at such query point can be totally wrong. Despite such difficulty, in many real scenarios, due to problem complexity or lack of prior knowledge, parametric models are not always available. Therefore, we hope to explore effective schemes to overcome the robustness issue of nonparametric regression.

In this paper, we provide a theoretical study about robust nonparametric regression problem under poisoning attack. In particular, we hope to investigate the theoretical limit of this problem, and design a method to achieve this limit. Towards this goal, we make the following contributions: 

Firstly, we propose and analyze an estimator that minimizes a weighted Huber loss, which is quadratic with small input, and linear with large input. Such design achieves a tradeoff between consistency and adversarial robustness. It was originally proposed in \cite{hall1990adaptive}, but to the best of our knowledge, it was not analyzed under adversarial setting. We show the convergence rate of both $\ell_2$ and $\ell_\infty$ risk, under the assumption that the function to estimate is Lipschitz continuous, and the noise is sub-exponential. An interesting finding is that the maximum number of attacked samples (denoted as $q$) is not too large, then the convergence rate is not affected by adversarial samples, i.e. the influence of poisoning samples on the overall risk is only up to a constant factor.

Secondly, we provide an information theoretic minimax lower bound, which indicates the underlying limit one can achieve, with respect to $q$ and $N$. The minimax lower bound without adversarial samples can be derived using standard information theoretic methods \cite{tsybakov2009}. Under adversarial attack, the estimation problem is harder, thus the lower bound in \cite{tsybakov2009} may not be tight enough. We design some new techniques to derive a tighter one. The result shows that the initial estimator has optimal $\ell_\infty$ risk. With small $q$, the $\ell_2$ risk is also minimax optimal. Nevertheless, for larger $q$, the $\ell_2$ risk is not optimal, indicating that this estimator is still not perfect. We then provide an intuitive explanation of the suboptimality. Instead of attacking some randomly selected training samples, the best strategy for the attacker is to focus their attack within a small region. With this strategy, majority of training samples are altered here, resulting in wrong estimates. A simple remedy is to increase the kernel bandwidth to improve robustness, which can make $\ell_\infty$ risk optimal. However, this adjustment will introduce additional bias in other regions, thus the $\ell_2$ risk is still suboptimal. The drawback of the initial estimator is that it does not make full use of the continuity of regression function, and thus unable to correct the estimation.


Finally, motivated by the issues of the initial method mentioned above, we propose a corrected estimator. If the attack focuses on a small region, then the initial estimate fails here. However, the estimate elsewhere is still reliable. With the assumption that the underlying function is continuous, the value at the severely corrupted region can be inferred using the surrounding values. With such intuition, we propose a nonlinear filtering method, which projects the estimated function to the space of Lipschitz functions with minimal $\ell_1$ distance. The corrected estimate is then proved to be nearly minimax optimal up to only a $\ln N$ factor.

\section{Preliminaries}
In this section, we clarify notations and provide precise problem statements.
Suppose $\mathbf{X}_1,\ldots, \mathbf{X}_N \in \mathcal{X}\subset \mathbb{R}^d$ be $N$ independently and identically distributed training samples, generated from a common probability density function (pdf) $f$. For each sample $\mathbf{X}_i$, we can receive a corresponding label $Y_i$:
\begin{eqnarray}
	Y_i = \left\{
	\begin{array}{cc}
		\eta(\mathbf{X}_i)+ W_i & \text{if } i\notin \mathcal{B}\\
		\star & \text{otherwise,}
	\end{array}
	\right.
	\label{eq:Y}
\end{eqnarray}
in which $\eta:\mathbb{R}^d\rightarrow \mathbb{R}$ is the unknown underlying function that we would like to estimate. $W_i$ is the noise variable. For $i=1,\ldots, N$, $W_i$ are independent, with zero mean and finite variance. $\mathcal{B}$ is the set of indices of attacked samples. $\star$ means some value determined by the attacker. For each normal sample $\mathbf{X}_i$, the received label is $Y_i=\eta(\mathbf{X}_i)+W_i$. However, if a sample is attacked, then $Y_i$ can be arbitrary value determined by the attacker. The attacker can manipulate up to $q$ samples, thus $|\mathcal{B}|\leq q$. 

Our goal is opposite to the attacker. We hope to find an estimate $\hat{\eta}$ that is as close to $\eta$ as possible, while the attacker aims at reducing the estimation accuracy using a carefully designed attack strategy. We consider white-box setting here, in which the attacker has complete access to the ground truth $\eta$, $\mathbf{X}_i$ and $W_i$ for all $i\in \{1,\ldots,N \}$, as well as our estimation algorithm. Under this setting, we hope to design a robust regression method that resists to any attack strategies.

The quality of estimation is evaluated using $\ell_2$ and $\ell_\infty$ loss, which is defined as 
\begin{eqnarray}
	R_2[\hat{\eta}]&=& \mathbb{E}\left[\underset{\mathcal{A}}{\sup}(\hat{\eta}(\mathbf{X})-\eta(\mathbf{X}))^2\right],\label{eq:l2def}\\
	R_\infty[\hat{\eta}]&=&\mathbb{E}\left[\underset{\mathcal{A}}{\sup}\underset{\mathbf{x}}{\sup}|\hat{\eta}(\mathbf{x})-\eta(\mathbf{x})|\right],
	\label{eq:linftydef}
\end{eqnarray}
in which the expectation in \eqref{eq:l2def} and \eqref{eq:linftydef} are taken over all training samples $(\mathbf{X}_i, Y_i), \ldots, (\mathbf{X}_N, Y_N)$. $\mathcal{A}$ denotes the attack strategy. The supremum over $\mathcal{A}$ is taken here because the adversary is assumed to be smart enough and the attack strategy is optimal. In \eqref{eq:l2def},  $\mathbf{X}$ denotes a random test sample that follows a distribution with pdf $f$. Our analysis can be easily generated to $\ell_p$ loss with arbitrary $p$.

Without any adversarial samples, $\eta$ can be learned using kernel regression, also called the Nadaraya-Watson estimator \cite{nadaraya1964estimating,watson1964smooth}:
\begin{eqnarray}
	\hat{\eta}_{NW}(\mathbf{x}) = \frac{\sum_{i=1}^N K\left(\frac{\mathbf{x}-\mathbf{X}_i}{h}\right)Y_i}{\sum_{i=1}^N K\left(\frac{\mathbf{x} - \mathbf{X}_i}{h} \right)},
	\label{eq:simple}
\end{eqnarray}

in which $K$ is the Kernel function, $h$ is the bandwidth that will decrease with the increase of sample size $N$. $\hat{\eta}_{NW}(\mathbf{x})$ can be viewed as a weighted average of the labels around $\mathbf{x}$. Without adversarial attack, such estimator converges to $\eta$ \cite{devroye1978uniform}. However, \eqref{eq:simple} fails even if a tiny fraction of samples are attacked. The attacked labels can just set to be sufficiently large. As a result, $\hat{\eta}_{NW}(\mathbf{x})$ could be far away from its truth.

\section{The Initial Estimator}\label{sec:method}

Now we build the estimator based on Huber loss minimization. Similar method was proposed in \cite{hall1990adaptive}. However, \cite{hall1990adaptive} analyzed the case in which the distribution of label has heavy tails, instead of the case with corrupted samples. To the best of our knowledge, the performance under adversarial setting has not been analyzed. Now we use $\hat{\eta}_0$ to denote a slightly modified version of the estimator proposed in \cite{hall1990adaptive}:
\begin{eqnarray}
	\hat{\eta}_0(\mathbf{x}) = 
	\underset{|s|\leq M}{\arg\min}\sum_{i=1}^N K\left(\frac{\mathbf{x}-\mathbf{X}_i}{h}\right)\phi(Y_i - s),
	\label{eq:eta}
\end{eqnarray}
in which tie breaks arbitrarily if the minimum is not unique, and
\begin{eqnarray}
	\phi(u) = \left\{
	\begin{array}{ccc}
		u^2 &\text{if} & |u|\leq T\\
		2T|u|-T^2 &\text{if} & |u|>T
	\end{array}
	\right.
	\label{eq:phi}
\end{eqnarray}
is the Huber loss function.

Here we provide an intuitive understanding of this method. We hope the estimator to have two properties: robustness under attack, and consistency without attack. Robustness is guaranteed if we let $\phi$ be $\ell_1$ loss, but the solution is the local median instead of mean. As long as the noise distribution is not symmetric, $\ell_1$ minimizer is not consistent. On the contrary, letting $\phi$ be $\ell_2$ loss just yields the kernel regression \eqref{eq:simple}, which is not robust. Therefore, Huber cost \eqref{eq:phi} is designed to get a tradeoff between these two goals, which is quadratic with small input and linear with large input. The threshold parameter $T$ can be set flexibly. Moreover, consider that there exists nonzero probability that $|\hat{\eta}(\mathbf{x})-\eta(\mathbf{x})|$ is arbitrarily large, we project the result into $[-M, M]$. \footnote{Suppose that for some $\mathbf{x}$, there is only one training sample whose distance to $\mathbf{x}$ is smaller than $h$. This sample is controlled by adversary and the value is altered arbitrarily far away. This event can result in arbitrarily large estimation error, and happens with nonzero probability. Therefore, if we do not project the estimate output to $[-M, M]$, then both $\ell_2$ and $\ell_\infty$ loss will be infinite.}

There are several simple baselines for comparison. The first one is median-of-means (MoM) \cite{nemirovskij1983problem,ben2023robust}, which divides samples into groups and calculates the median of the estimates in each group. MoM is inefficient because it fails even when there is only one attacked sample in each group. Another solution is trimmed mean \cite{bickel1965some,welsh1987trimmed,dhar2022trimmed}, which removes a fraction of samples with largest and smallest label values. The trim fraction parameter depends on the ratio of attacked samples. Unfortunately, such ratio is highly likely to be uneven over the support, while the trim fraction is set uniformly. This dilemma makes trimmed mean method not efficient. Robust regression with spline smoothing \cite{eubank1999nonparametric} is another alternative but is restricted to one dimensional problems. Finally, robust regression trees \cite{chaudhuri2002nonparametric} works practically but theoretical guarantee is not provided.

Finally, we comment on the computation of the estimator \eqref{eq:eta}. Note that $\phi$ is convex, therefore the minimization problem in \eqref{eq:eta} can be solved by gradient descent. The derivative of $\phi$ is
\begin{eqnarray}
	\phi'(u)= \left\{
	\begin{array}{ccc}
		2u &\text{if} & |u|\leq T\\
		2T & \text{if} & u>T\\
		-2T &\text{if} & u<-T.
	\end{array}
	\right.
	\label{eq:derivative}
\end{eqnarray} 
Based on \eqref{eq:eta} and \eqref{eq:derivative}, $s$ can be updated using binary search. Denote $\epsilon$ as the required precision, then the number of iterations for binary search should be $O(\ln (M/\epsilon))$. Therefore, the computational complexity is higher than kernel regression up to a $\ln (M/\epsilon)$ factor.

\section{Theoretical Analysis}\label{sec:theory}
This section proposes the theoretical analysis of the initial estimator \eqref{eq:eta} under adversarial setting. To begin with, we make some assumptions about the problem.
\begin{ass}\label{ass:problem}
	(Problem Assumption) there exists a compact set $\mathcal{X}$ and several constants $L$, $\gamma$, $f_m$, $f_M$, $D$, $\alpha$, $\sigma$, such that the pdf $f$ is supported at $\mathcal{X}$, and
	
	(a) (Lipschitz continuity) For any $\mathbf{x}_1, \mathbf{x}_2 \in \mathcal{X}$, 
	$|\eta(\mathbf{x}_1)-\eta(\mathbf{x}_2)|\leq L||\mathbf{x}_1-\mathbf{x}_2||$;
	
	(b) (Bounded $f$ and $\eta$) For all $\mathbf{x} \in \mathcal{X}$, $f_m\leq f(\mathbf{x}) \leq f_M$ and $|\eta(\mathbf{x})| \leq M$, in which $M$ is the parameter used in \eqref{eq:eta};
	
	(c) (Corner shape restriction) For all $r<D$ and $\mathbf{x}\in \mathcal{X}$,
	$V(B(\mathbf{x}, r)\cap \mathcal{X})\geq \alpha v_dr^d$,
	in which $B(\mathbf{x}, r)$ is the ball centering at $\mathbf{x}$ with radius $r$, $v_d$ is the volume of $d$ dimensional unit ball, which depends on the norm we use;
	
	(d) (Sub-exponential noise) The noise $W_i$ is subexponential with parameter $\sigma$, 
	\begin{eqnarray}
		\mathbb{E}[e^{\lambda W_i}]\leq e^{\frac{1}{2}\sigma^2 \lambda^2}, \forall |\lambda|\leq \frac{1}{\sigma},
		\label{eq:subexp}
	\end{eqnarray}
	for $i=1,\ldots, N$.
\end{ass}

In Assumption \ref{ass:problem}, (a) is a common assumption for smoothness. (b) is also commonly made and usually called "strong density assumption" in existing literatures on nonparametric statistics \cite{audibert2007fast,doring2017rate}. This assumption requires the pdf to be both upper and lower bounded in its support. Although somewhat restrictive, this assumption facilitates theoretical analysis. Relaxing this assumption is possible. For example, we may use the adaptive strategies in \cite{gadat2016classification,zhao2021minimax,Zhao:2022:TIT}. We refer to section \ref{sec:beyond} in the appendix for some further analysis. (c) prevents the shape of the corner of the support from being too sharp. Without assumption (c), the samples around the corner may not be enough, and the attacker can just attack the samples at the corner of the support, which can result in large errors. (d) requires that the noise is sub-exponential. If the noise assumption is weaker, e.g. only requiring the bounded moments of $W_i$ up to some order, then the noise can be disperse. In this case, it will be harder to distinguish adversarial samples from clean samples. 

We then make some restrictions on the kernel function $K$.
\begin{ass}\label{ass:kernel}
	(Kernel Assumption) the kernel need to satisfy:
	
	(a) $\int K(\mathbf{u}) du = 1$;
	
	(b)$K(\mathbf{u}) = 0, \forall ||\mathbf{u}||>1$;
	
	(c) $c_K \leq K(\mathbf{u)}\leq C_K$ for two constants $c_K$ and $C_K$.
\end{ass}
In Assumption \ref{ass:kernel}, (a) is actually not necessary, since from \eqref{eq:eta}, the estimated value will not change if the kernel function is multiplied by a constant factor. This assumption is only for convenience of proof. (b) and (c) require that the kernel need to be somewhat close to the uniform function in the unit ball. Intuitively, if the attacker wants to modify the estimate at some $\mathbf{x}$, the best way is to change the response of sample $i$ with large $K((\mathbf{X}_i-\mathbf{x})/h)$, in order to make strong impact on $\hat{\eta}(\mathbf{x})$. To defend against such attack, the upper bound of $K$ should not be too large. Besides, to ensure that clean samples dominate corrupted samples everywhere, the effect of each clean sample on the estimation should not be too small, thus $K$ also need to be bounded from below in its support.

Furthermore, recall that \eqref{eq:eta} has three parameters, i.e. $h$, $T$ and $M$. We assume that these three parameters satisfy the following conditions.

\begin{ass}\label{ass:parameter}
	(Parameter Assumption) $h$, $T$, $M$ need to satisfy
	
	(a)$h>\ln^2 N/N$;	
	
	(b)$T\geq 4Lh+16\sigma\ln N$;
	
	(c)$M>\underset{\mathbf{x}\in \mathcal{X}}{\sup} |\eta(\mathbf{x})|$.
\end{ass}
In Assumption \ref{ass:parameter}, (a) ensures that the number of samples whose distance to $\mathbf{x}$ less than $h$ is not too small. It is necessary for consistency, but is not enough for a good tradeoff between bias, variance and robustness. The optimal dependence of $h$ over $N$ will be discussed later. (b) requires that $T\sim \ln N$. This rule is based on the sub-exponential noise condition in Assumption \ref{ass:problem}(d). If we use sub-Gaussian assumption instead, then it is enough for $T\sim \sqrt{\ln N}$. If the noise is further assumed to be bounded, then $T$ can just be set to constant. On the contrary, if the noise has heavier tail, then $T$ needs to grow with $N$ faster. (b) is mainly designed for rigorous theoretical analysis. Practically, one may choose smaller $T$. (c) prevents the estimate from being truncated too much.

The upper bound of $\ell_2$ error is derived under these assumptions. Denote $a\lesssim b$ if $a\leq Cb$ for some constant $C$ that depends only on $L, M, \gamma, f_m, f_M, D, \alpha, \sigma,c_K, C_K$.

\begin{thm}\label{thm:main}
	Under Assumption \ref{ass:problem}, \ref{ass:kernel} and \ref{ass:parameter},
	\begin{eqnarray}
		\mathbb{E}\left[\underset{\mathcal{A}}{\sup}\left(\hat{\eta}_0(\mathbf{X})-\eta(\mathbf{X})\right)^2\right]\lesssim \frac{T^2q}{N}\min\left\{\frac{q}{Nh^d}, 1 \right\} + h^{2} + \frac{1}{Nh^d}.
		\label{eq:upper}
	\end{eqnarray}
\end{thm}

The detailed proof of Theorem \ref{thm:main} is shown in section \ref{sec:l2} in the appendix. Here we provide an intuitive explanation of \eqref{eq:upper}. The first term in \eqref{eq:upper} is caused by adversarial attack, while the remaining two terms are just the standard nonparametric regression error \cite{tsybakov2009} for clean samples. Therefore we only discuss the first term here. The best strategy for the adversary is to concentrate its attack on a small region. Denote $B_h(\mathbf{x})$ as the ball centering at $\mathbf{x}$ with radius $h$, in which $h$ is the bandwidth parameter in \eqref{eq:eta}. Since the pdf $f$ is both upper and lower bounded, the number of samples within $B_h(x)$ roughly scales as $Nh^d$. Now we discuss two cases. Firstly, if $q\lesssim Nh^d$, with $q$ attacked samples around $\mathbf{x}$, the additional estimation error caused by these adversarial samples roughly scales as $Tq/(Nh^d)$. These attacked samples can affect $\hat{\eta}_0(\mathbf{x})$ for a region with radius roughly $h$, thus the overall additional $\ell_2$ error is $(Tq/(Nh^d))^2 h^d = T^2q^2/(N^2 h^d)$. Secondly, if $q\gtrsim Nh^d$, then the adversary can attack most of samples in a much broader region, whose volume scales as $q/N$. The additional estimation error in this region is proportional to $T$, thus the additional $\ell_2$ error is $T^2q/N$. Combining these two cases yields \eqref{eq:upper}, in which there is a phase transition between $q\lesssim Nh^d$ and $q\gtrsim Nh^d$.

\begin{rmk}
	Theorem \ref{thm:main} is based on the assumption that the pdf $f$ is bounded from below. For the case such that $f$ has bounded support but can approach zero arbitrarily, we have provided an analysis in section \ref{sec:beyond} in the appendix. The result is that
	\begin{eqnarray}
		\mathbb{E}\left[\underset{\mathcal{A}}{\sup}\left(\hat{\eta}_0(\mathbf{X})-\eta(\mathbf{X})\right)^2\right]\lesssim \frac{T^2 q}{N}+h^2 +\frac{1}{Nh^d}.
		\label{eq:upper2}
	\end{eqnarray}
	From \eqref{eq:upper2}, the bound is worse than the case with densities bounded from below if $q\lesssim Nh^d$. Intuitively, in this case, the best strategy for the adversary would be to attack samples in the region with low pdf values.
\end{rmk}

The next theorem shows the bound of $\ell_\infty$ error:
\begin{thm}\label{thm:sup}
	Under Assumption \ref{ass:problem}, \ref{ass:kernel}, \ref{ass:parameter}, if $K(\mathbf{u})$ is monotonic decreasing with respect to $\|u\|$, then
	\begin{eqnarray}
		\mathbb{E}\left[\underset{\mathcal{A}}{\sup}\underset{\mathbf{x}}{\sup}|\hat{\eta}_0(\mathbf{x})-\eta(\mathbf{x})|\right] \lesssim \frac{Tq}{Nh^d} + h + \frac{\ln N}{\sqrt{Nh^d}}.	
		\label{eq:sup}
	\end{eqnarray}
\end{thm}
The detailed proof is in section \ref{sec:linfty} in the appendix. Unlike $\ell_2$ loss, the assumption that $f$ is bounded from below can not be relaxed under $\ell_\infty$ loss. If $f$ can approach zero, then the adversary can just attack the region with low density. As a result, we can only get a trivial bound $R_\infty \lesssim 1$.

We then show the minimax lower bound, which indicates the information theoretic limit of the adversarial nonparametric regression problem. In general, it is impossible to design an estimator with convergence rate faster than the following bound.

\begin{thm}\label{thm:minimax}
	Let $\mathcal{F}$ be the collection of $f,\eta,\mathbb{P_N}$ that satisfy Assumption \ref{ass:problem}, in which $\mathbb{P}_N$ is the distribution of the noise $W_1,\ldots, W_N$. Then 
	\begin{eqnarray}
		\underset{\hat{\eta}}{\inf}\underset{(f,\eta,\mathbb{P}_N)\in \mathcal{F}}{\sup}\mathbb{E}\left[\underset{\mathcal{A}}{\sup}\left(\hat{\eta}(\mathbf{X}) - \eta(\mathbf{X})\right)^2 \right]
		 \gtrsim \left(\frac{q}{N}\right)^\frac{d+2}{d+1} + N^{-\frac{2}{d+2}},
		\label{eq:mmx}
	\end{eqnarray}
	and
	\begin{eqnarray}
		\underset{\hat{\eta}}{\inf}\underset{(f,\eta,\mathbb{P}_N)\in \mathcal{F}}{\sup}\mathbb{E}\left[\underset{\mathcal{A}}{\sup}\underset{\mathbf{x}}{\sup}|\hat{\eta}(\mathbf{x})-\eta(\mathbf{x})|\right]
		\gtrsim \left(\frac{q}{N}\right)^\frac{1}{d+1}+N^{-\frac{1}{d+2}}.
		\label{eq:mmxsup}
	\end{eqnarray}
\end{thm}

The proof is shown in section \ref{sec:mmx} in the appendix. In the right hand side of \eqref{eq:mmx} and \eqref{eq:mmxsup}, $N^{-2/(d+2)}$ is the standard minimax lower bound for nonparametric estimation \cite{tsybakov2009}, which holds even if there are no adversarial samples. In the appendix, we only prove the lower bound with the first term in the right hand side of \eqref{eq:mmx}. The basic idea is to construct two hypotheses on the regression function $\eta$. The total variation distance between these two hypotheses is not too large, thus the adversary can transform one of them to the other. As a result, after adversarial contamination with a carefully designed strategy, we are no longer able to distinguish between these hypotheses. The lower bounds in \eqref{eq:mmx} and \eqref{eq:mmxsup} can then be constructed accordingly.

Compare Theorem \ref{thm:main}, \ref{thm:sup} and Theorem \ref{thm:minimax}, we have the following findings. We claim that the upper and lower bound nearly match, if these two bounds match up to a polynomial of $\ln N$:
\begin{itemize}
	\item The $\ell_\infty$ error is rate optimal. From \eqref{eq:sup} and \eqref{eq:mmxsup}, with $h\sim \max\{(q/N)^{1/(d+1)}, N^{-1/(d+2)}\}$ and $T\sim \ln N$, the upper and minimax lower bound of $\ell_\infty$ error nearly match.
	
	\item The $\ell_2$ error is rate optimal if $q\lesssim \max\left\{\sqrt{N/\ln^2 N}, N^{d/(d+2)}/\ln^2 N\right\}$. From \eqref{eq:upper} and \eqref{eq:mmx}, let $h\sim N^{-\frac{1}{d+2}}$, the upper and minimax lower bound of $\ell_2$ error match. In fact, in this case, the convergence rate of \eqref{eq:eta} is the same as ordinary kernel regression without adversarial samples, i.e. $h^2+1/(Nh^d)$. With optimal selection of $h$, $\ell_2$ error scales as $N^{-2/(d+2)}$, which is just the standard rate for nonparametric statistics \cite{krzyzak1986rates,tsybakov2009}.
	
	\item The $\ell_2$ error is not rate optimal if $q\gtrsim \max\left\{\sqrt{N/\ln^2 N}, N^{d/(d+2)}/\ln^2 N\right\}$. In this case, if $d\leq 2$, the optimal $h$ in \eqref{eq:upper} is $h\sim (q\ln N/N)^{2/(d+2)}$, and resulting $\ell_2$ error is $R_2\lesssim (q\ln N/N)^{4/(d+2)}$. If $d>2$, then optimal $h$ is $h\sim N^{-1/(d+2)}$, and the $\ell_2$ error is $q\ln^2 N/N+N^{-2/(d+2)}$. For either $d\leq 2$ or $d>2$, these two bounds are worse than the lower bound in \eqref{eq:mmx}.
\end{itemize}
This result indicates that the initial estimator \eqref{eq:eta} is optimal under $\ell_\infty$, or under $\ell_2$ with small $q$. However, under large number of adversarial samples, the $\ell_2$ error becomes suboptimal. 

Now we provide an intuitive understanding of the suboptimality of $\ell_2$ risk with large $q$ using a simple one dimensional example shown in Figure \ref{fig:example}, in which $N=10000$, $h=0.05$, $M=3$, $f(x)=1$ for $x\in (0,1)$, $\eta(x)=\sin(2\pi x)$, and the noise follows standard normal distribution $\mathcal{N}(0,1)$. For each $x$, denote $q_h(x)$, $n_h(x)$ as the number of attacked samples and total samples within $(x-h, x+h)$, respectively. For robust mean estimation problems, the breakdown point is $1/2$ \cite{andrews2015robust}, which also holds locally for nonparametric regression problem. Hence, if $q_h(x)/n_h(x)>1/2$, the estimator will collapse and return erroneous values even if we use Huber cost. In Fig \ref{fig:example}(a), $q=500$, among which $250$ attacked samples are around $x=0.25$, while others are around $x=0.75$. In this case, $q_h(x)/n_h(x)<1/2$ over the whole support. The curve of estimated function is shown in Fig \ref{fig:example}(b). The estimate with \eqref{eq:eta} is significantly better than kernel regression. Then we increase $q$ to $2000$. In this case, $q_h(x)/n_h(x)>1/2$ around $0.25$ and $0.75$ (Fig \ref{fig:example}(c)), thus the estimate fails. The estimated function curve shows an undesirable spike (Fig \ref{fig:example}(d)). 

\begin{figure}[h!]
	\centering	
	\begin{subfigure}{0.24\linewidth}
		\includegraphics[width=\textwidth,height=0.76\textwidth]{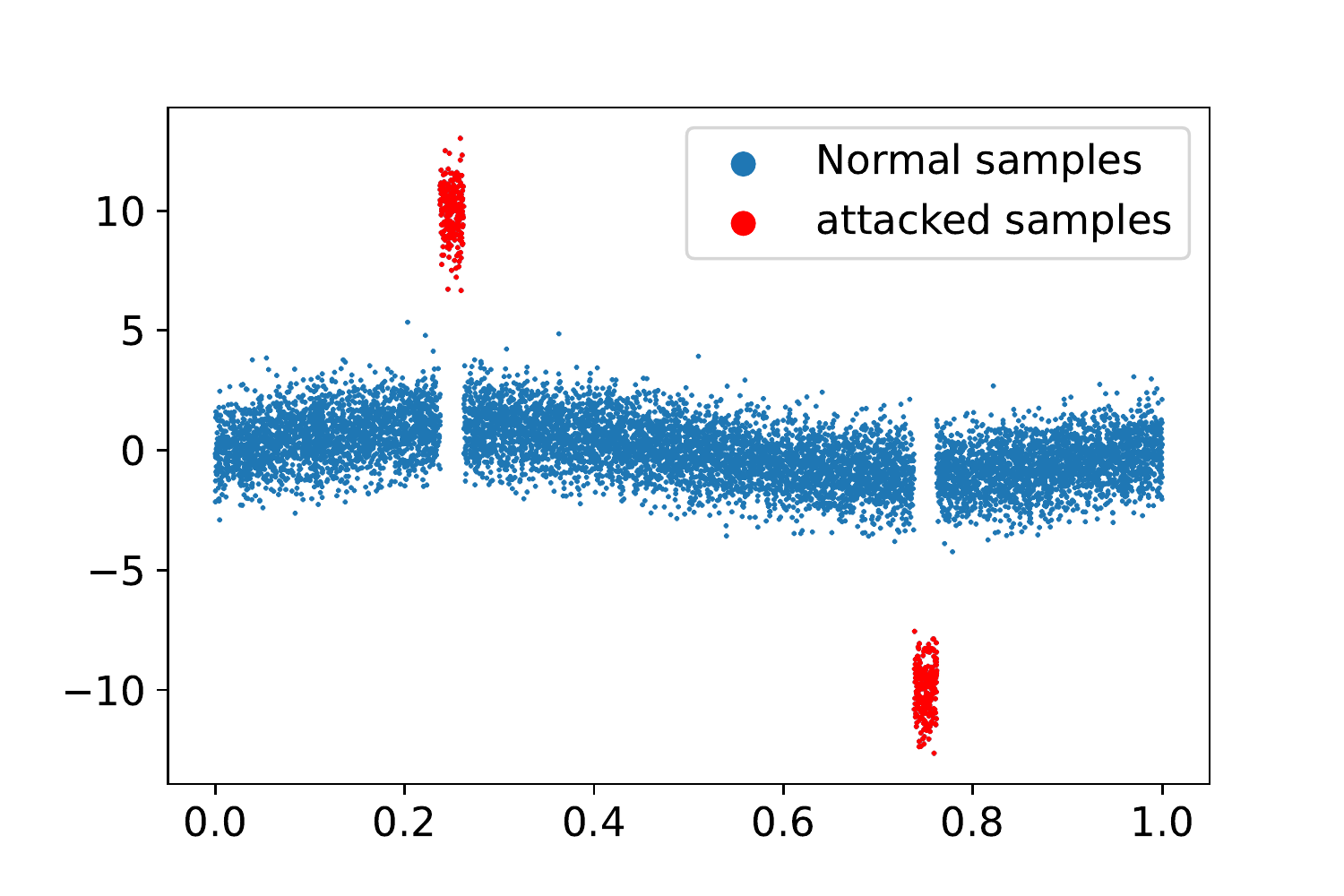}
		\caption{Scatter plots with $q=500$.}
	\end{subfigure}
	\hfill
	\begin{subfigure}{0.24\linewidth}
		\includegraphics[width=\textwidth,height=0.76\textwidth]{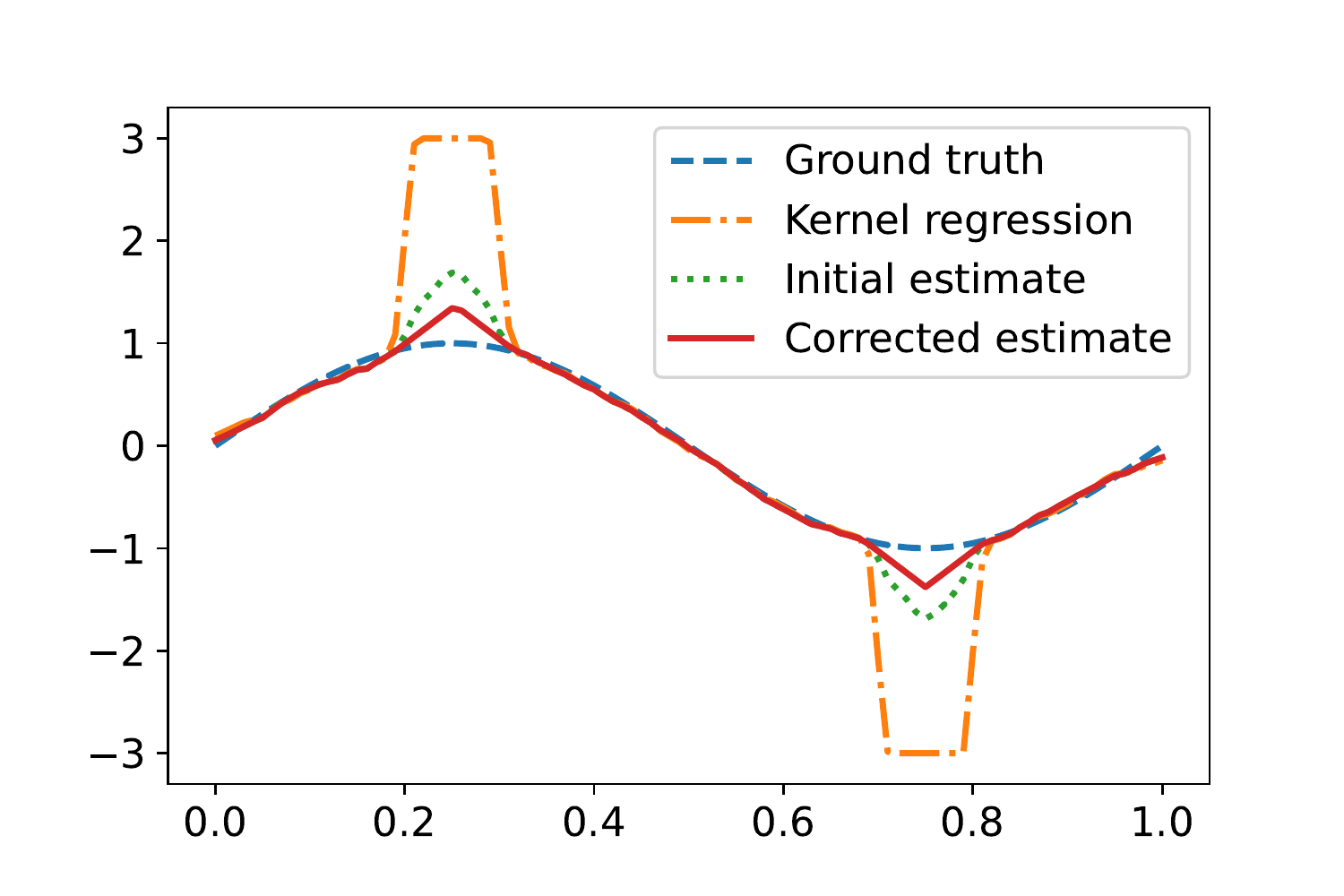}
		\caption{Estimated curves, $q=500$.}
	\end{subfigure}
	\hfill
	\begin{subfigure}{0.24\linewidth}
		\includegraphics[width=\textwidth,height=0.76\textwidth]{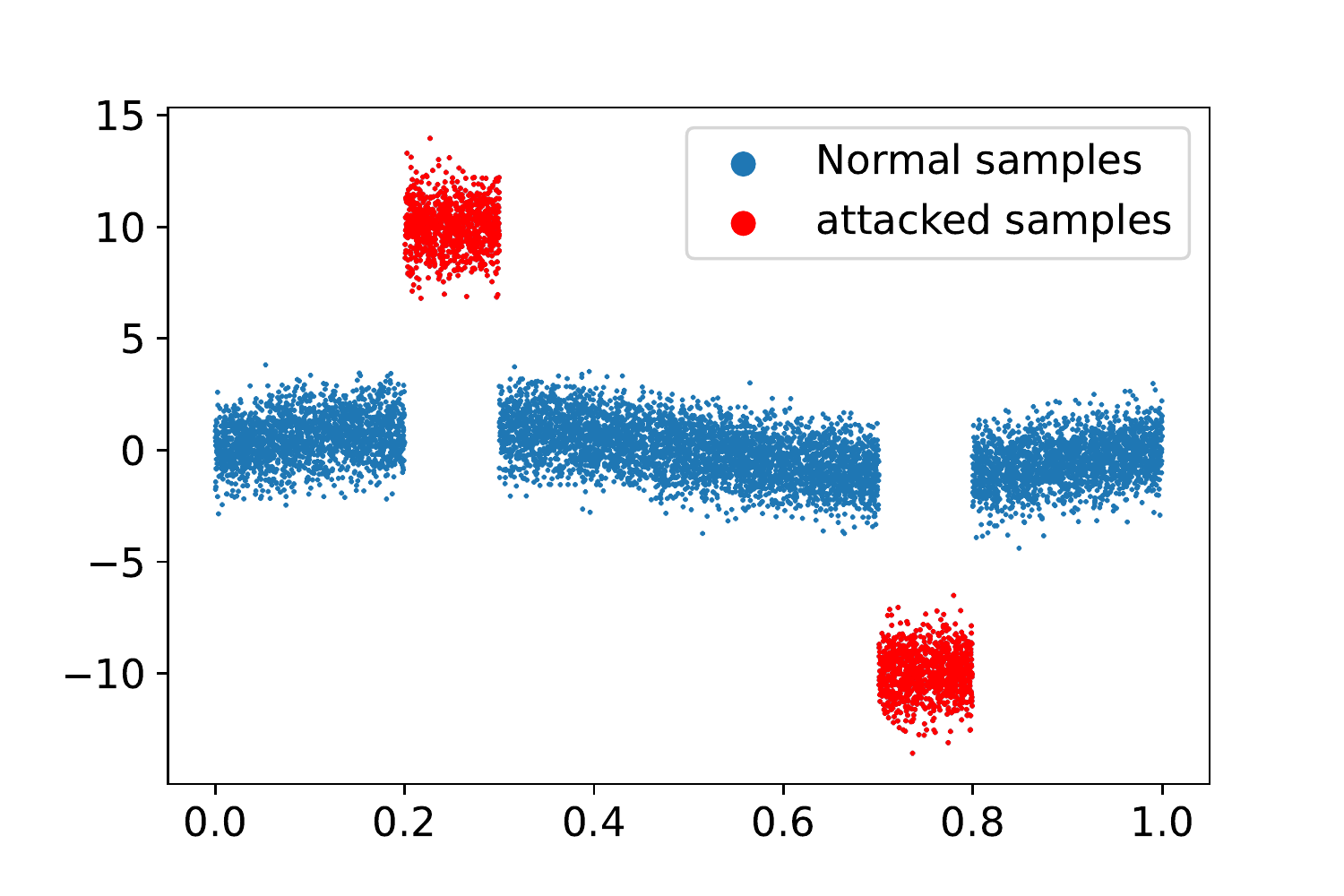}
		\caption{Scatter plots with $q=2000$.}
	\end{subfigure}
	\hfill
	\begin{subfigure}{0.24\linewidth}
		\includegraphics[width=\textwidth,height=0.76\textwidth]{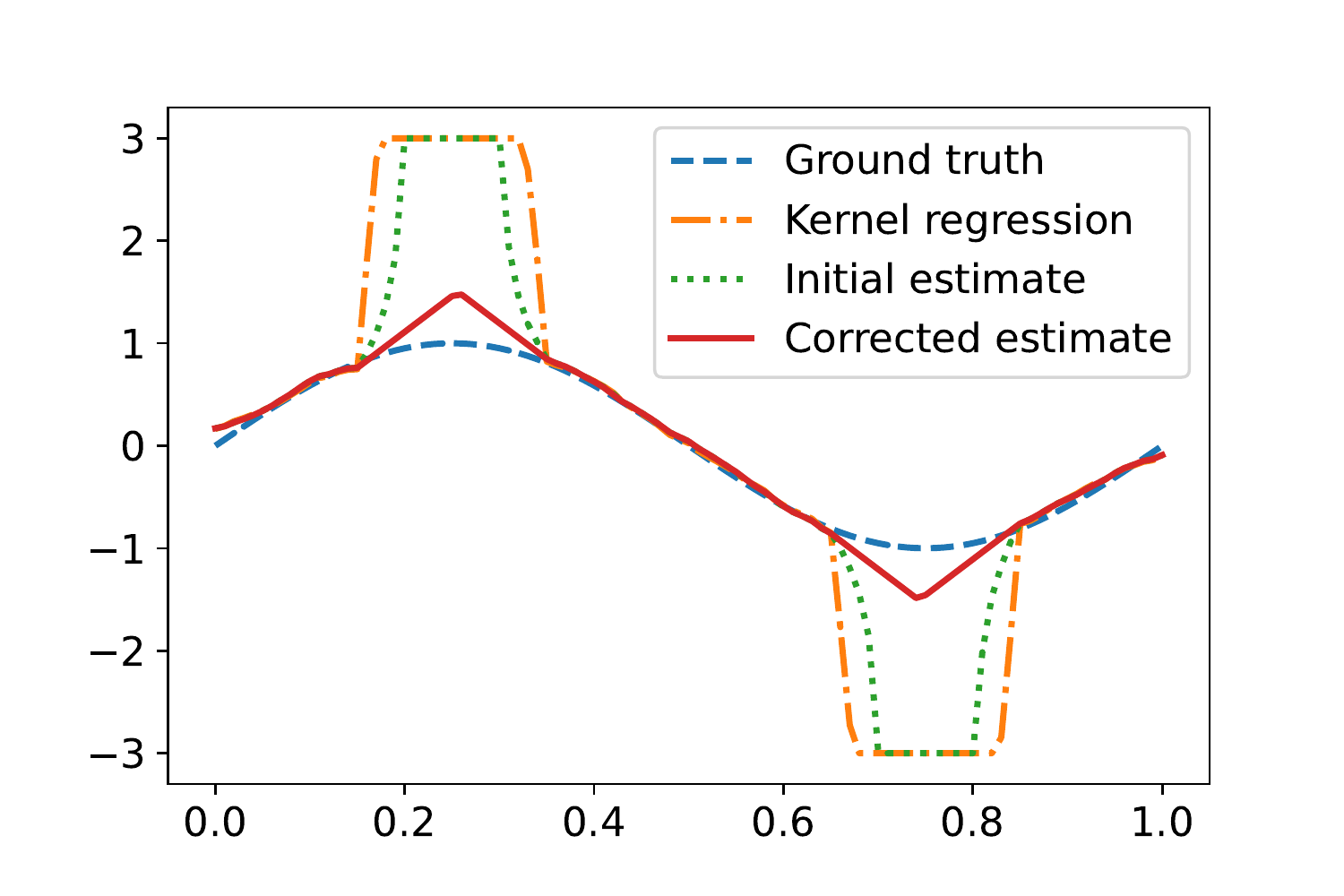}
		\caption{Estimated curves, $q=2000$.}
	\end{subfigure}	
	\caption{A simple example with $q=500$ and $q=2000$. In (a) and (c), red dots are attacked samples, while blue dots are normal samples. In (b) and (d), four curves correspond to ground truth $\eta$, the result of kernel regression, initial estimate and corrected estimate, respectively. With $q=500$, the initial estimate \eqref{eq:eta} works well. However, with $q=2000$, the initial estimate fails, while the corrected regression works well. }\label{fig:example}
\end{figure}

The above example shows that the initial estimator \eqref{eq:eta} fails if the adversary focus its attack at a small region. In this case, the local ratio of attacked samples surpasses the breakdown point, resulting in spikes here. With such strategy and sufficiently large $q$, the initial estimator \eqref{eq:eta} fails to be optimal. Actually, getting a robust estimate of $\eta(\mathbf{x})$ using local training samples around $\mathbf{x}$ only is not enough. To improve the estimator, we exploit the continuity property of $\eta$ (Assumption \ref{ass:problem}(a)), and use the estimate in neighboring regions to correct $\hat{\eta}_0(\mathbf{x})$. Based on such intuition, we propose a projection technique, which will close the gap between the upper and minimax lower bound. The details are shown in the next section. 
\section{Corrected Regression}\label{sec:corrected}

As has been discussed in the previous section, while the initial estimator is already efficient in its own right with small $q$, it does not tolerate larger $q$. In particular, concentrated attack will generate undesirable spikes in $\hat{\eta}_0$. We hope to remove these spikes without introducing two much additional estimation error. Linear filters\footnote{Here linear filter means that the output is linear in the input, i.e. an operator $F$ is linear if for any function $f_1$, $f_2$ and any scalars $\lambda_1$ and $\lambda_2$, $F[\lambda_1f_1+\lambda_2f_2]=\lambda_1 F[f_1]+ \lambda_2F[f_2]$. Alternatively, $F[f]$ is a convolution of $f$ with another function $K_F$. Such convolution can blur the regression estimate.} do not work here since the profile of the regression estimate will be blurred. Therefore, we propose a nonlinear filter as following. It conducts minimum correction (in $\ell_1$ sense) to the initial result $\hat{\eta}_0$, while ensuring that the corrected estimate is Lipschitz. Formally, given the initial estimate $\hat{\eta}_0(\mathbf{x})$, our method solves the following optimization problem
\begin{eqnarray}
	\hat{\eta}_c = \underset{\|\nabla g\|_\infty \leq L}{\arg\min} \|\hat{\eta}_0-g\|_1,
	\label{eq:optim}
\end{eqnarray}
in which
\begin{eqnarray}
	\|\nabla g\|_\infty = \max\left\{\left|\frac{\partial g}{\partial x_1}\right|, \ldots, \left|\frac{\partial g}{\partial x_d}\right| \right\}.
\end{eqnarray}

In section \ref{sec:unique} in the appendix, we prove that the solution to the optimization problem \eqref{eq:optim} is unique. 

From \eqref{eq:optim}, $\hat{\eta}_c$ can be viewed as the projection of the output of initial estimator $\hat{\eta}_0$ into the space of Lipschitz functions. Here we would like to explain intuitively why we use $\ell_1$ distance instead of other metrics in \eqref{eq:optim}. Using the example in Fig.\ref{fig:example}(d) again, it can be observed that at the position of such spikes, $|\eta(\mathbf{x}) - g(\mathbf{x})|$ can be large. In order to ensure successful removal of spikes, we hope that the derivative of such cost should not be too large, otherwise the corrected estimate will tend to be closer to the original one to minimize the cost, thus spikes may not be fully removed. Based on such intuition, $\ell_1$ cost is preferred here, since it has bounded derivatives, while other costs such as $\ell_2$ distance have growing derivatives.

The estimation error of the corrected regression can be bounded by the following theorem.
\begin{thm}\label{thm:corrected}
	(1) Under the same conditions as Theorem \ref{thm:main},
	\begin{eqnarray}
		\mathbb{E}\left[\underset{\mathcal{A}}{\sup}\left(\hat{\eta}_c(\mathbf{X})-\eta(\mathbf{X})\right)^2\right]\lesssim \left(\frac{q\ln N}{N}\right)^\frac{d+2}{d+1}+h^2 + \frac{\ln N}{Nh^d}.
		\label{eq:upperc}
	\end{eqnarray}
	
	(2) Under the same conditions as Theorem \ref{thm:sup},
	\begin{eqnarray}
		\mathbb{E}\left[\underset{\mathcal{A}}{\sup}\underset{\mathbf{x}}{\sup}|\hat{\eta}_c(\mathbf{x})-\eta(\mathbf{x})|\right] \lesssim \frac{Tq}{Nh^d} + h + \frac{\ln N}{\sqrt{Nh^d}}.	
	\end{eqnarray}
\end{thm}
The proof is shown in section \ref{sec:converge} in the appendix. Here we provide a brief idea of the proof. For the error of $\hat{\eta}_0$ in \eqref{eq:upper}, the first term is caused by adversarial samples, while the second and third term are just usual regression error. The latter one nearly remains the same after filtering, while the impact of the former error is significantly reduced. In particular, the $\ell_1$ additional estimation error can be bounded first. This bound can then be used to infer $\ell_2$ and $\ell_\infty$ error caused by adversarial samples, using the property that $\hat{\eta}_c$ is Lipschitz. From \eqref{eq:upperc}, compared with Theorem \ref{thm:minimax}, with $T\sim \ln N$ and a proper $h$, the upper and lower bound nearly match. 

Now we discuss the practical implementation. \eqref{eq:optim} can not be calculated directly for a continuous function. Therefore, we find an approximate numerical solution instead. The detail of practical implementation is shown in section \ref{sec:implement} in the appendix. 

Despite the optimal sample complexity, the computation of the corrected estimator is expensive for high dimensional distributions. It would be an interesting future direction to improve the computational complexity on dimensionality. Currently, our method is designed mainly for low dimensional problems.

\section{Numerical Examples}\label{sec:numerical}
In this section we show some numerical experiments. In particular, we show the curve of the growth of mean square error over the attacked sample size $q$. More numerical results are shown in the appendix.

For each case, we generate $N=10000$ training samples, with each sample follows uniform distribution in $[0,1]^d$. The kernel function is
\begin{eqnarray}
	K(u) = 2-|u|, \forall |u|\leq 1.
\end{eqnarray}

We compare the performance of kernel regression, the median-of-means method, trimmed mean, initial estimate, and the corrected estimation under multiple attack strategies. For kernel regression, the output is $\max(\min(\hat{\eta}_{NW}, M), -M)$, in which $\hat{\eta}_{NW}$ is the simple kernel regression defined in \eqref{eq:simple}. We truncate the result into $[-M, M]$ for a fair comparison with robust estimators.  For the median-of-means method, we divide the training samples into $b=20$ groups randomly, and then conduct kernel regression for each group and then find the median, i.e.
\begin{eqnarray}
	\hat{\eta}_{MoM} = \Clip(\med(\{\hat{\eta}_{NW}^{(1)}, \ldots, \hat{\eta}_{NW}^{(b)}\}), [-M, M]),
\end{eqnarray}
in which $\Clip(u, [-M, M])=\max(\min(x, M), -M)$ projects the value onto $[-M, M]$, and $\med$ denotes the median. For trimmed mean regression, the trim fraction is $0.2$.

For the initial estimator \eqref{eq:eta}, the parameters are $T=1$ and $M=3$. The corrected estimate uses (3) in the appendix. For $d=1$, the grid count is $m=50$. For $d=2$, $m_1=m_2=20$. Consider that the optimal bandwidth ($h$ in \eqref{eq:eta}) need to increase with the dimension, in \eqref{eq:simple}, the bandwidths of all these four methods are set to be $h=0.03$ for one dimensional distribution, and $h=0.1$ for two dimensional case. Here $M$ and $h$ satisfy Assumption \ref{ass:parameter}(a) and (c), while $T$ is smaller than the requirement in Assumption \ref{ass:parameter}(b). As was already discussed earlier, the parameter selection rule in Assumption \ref{ass:parameter} is designed mainly for theoretical analysis, and does not need to be exactly satisfied in practice.

The attack strategies are designed as following. Let $q=500k$ for $k=0,1,\ldots, 10$.
\begin{defi}
	There are three strategies, namely, random attack, one directional attack, and concentrated attack, which are defined as following:
	
	(1) Random Attack. The attacker randomly select $q$ samples among the training data to attack. The value of each attacked sample is $-10$ or $10$ with equal probability;
	
	(2) One directional Attack. The attacker randomly select $q$ samples among the training data to attack. The value of all attacked samples are $10$;
	
	(3) Concentrated Attack. The attacker pick two random locations $\mathbf{c}_1$, $\mathbf{c}_2$ that are uniformly distributed in $[0,1]^d$. For $\lfloor q/2\rfloor$ samples that are closest to $\mathbf{c}_1$, modify their values to $10$. For $\lfloor q/2\rfloor$ samples that are closest to $\mathbf{c}_2$, modify their values to $-10$.
\end{defi}

\begin{figure}[h!]
	\centering
	\begin{subfigure}{0.32\linewidth}
		\includegraphics[width=\textwidth,height=0.76\textwidth]{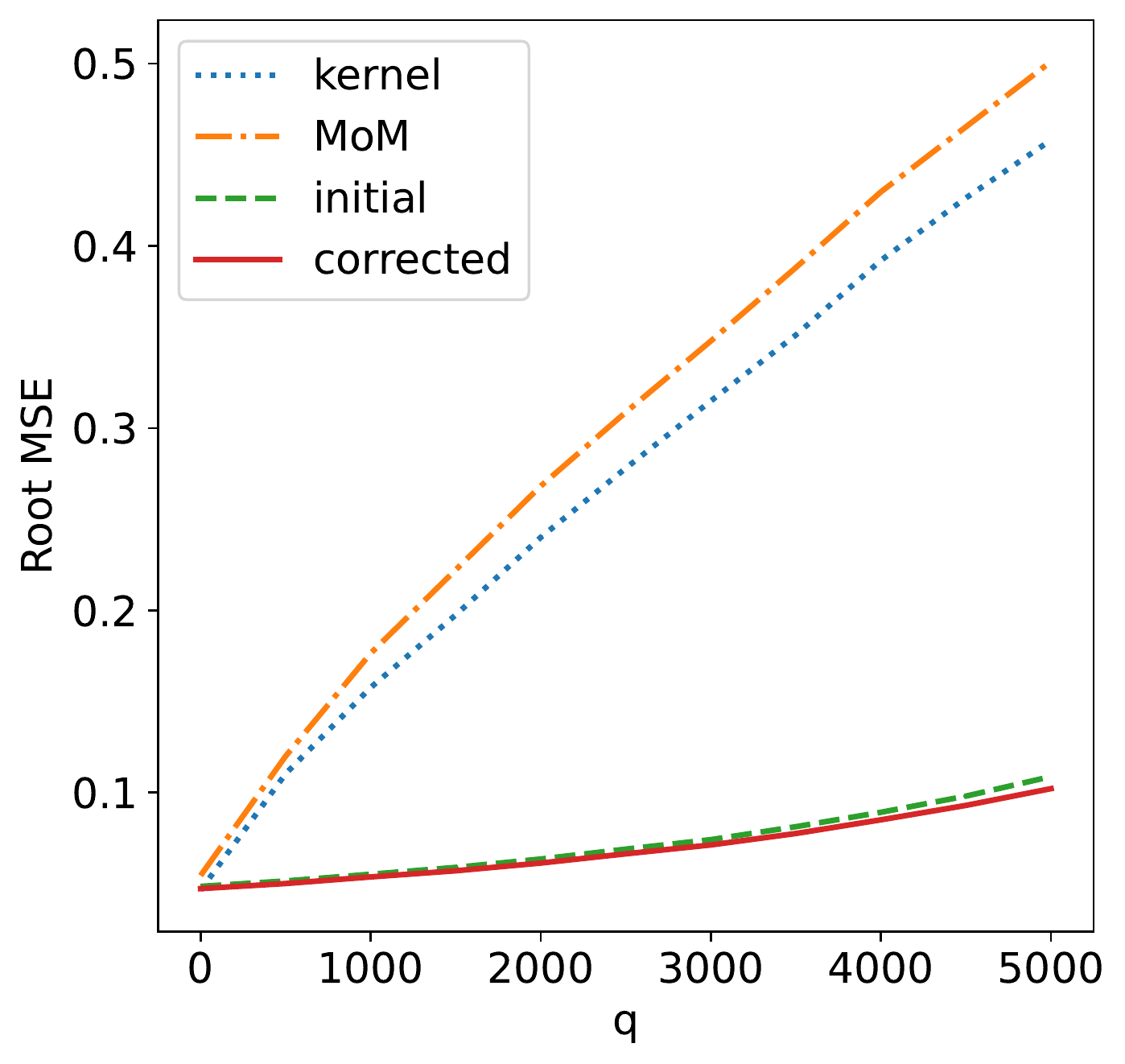}
		\caption{Squared root of $\ell_2$ error, random attack.}
	\end{subfigure}
	\begin{subfigure}{0.32\linewidth}
		\includegraphics[width=\textwidth,height=0.76\textwidth]{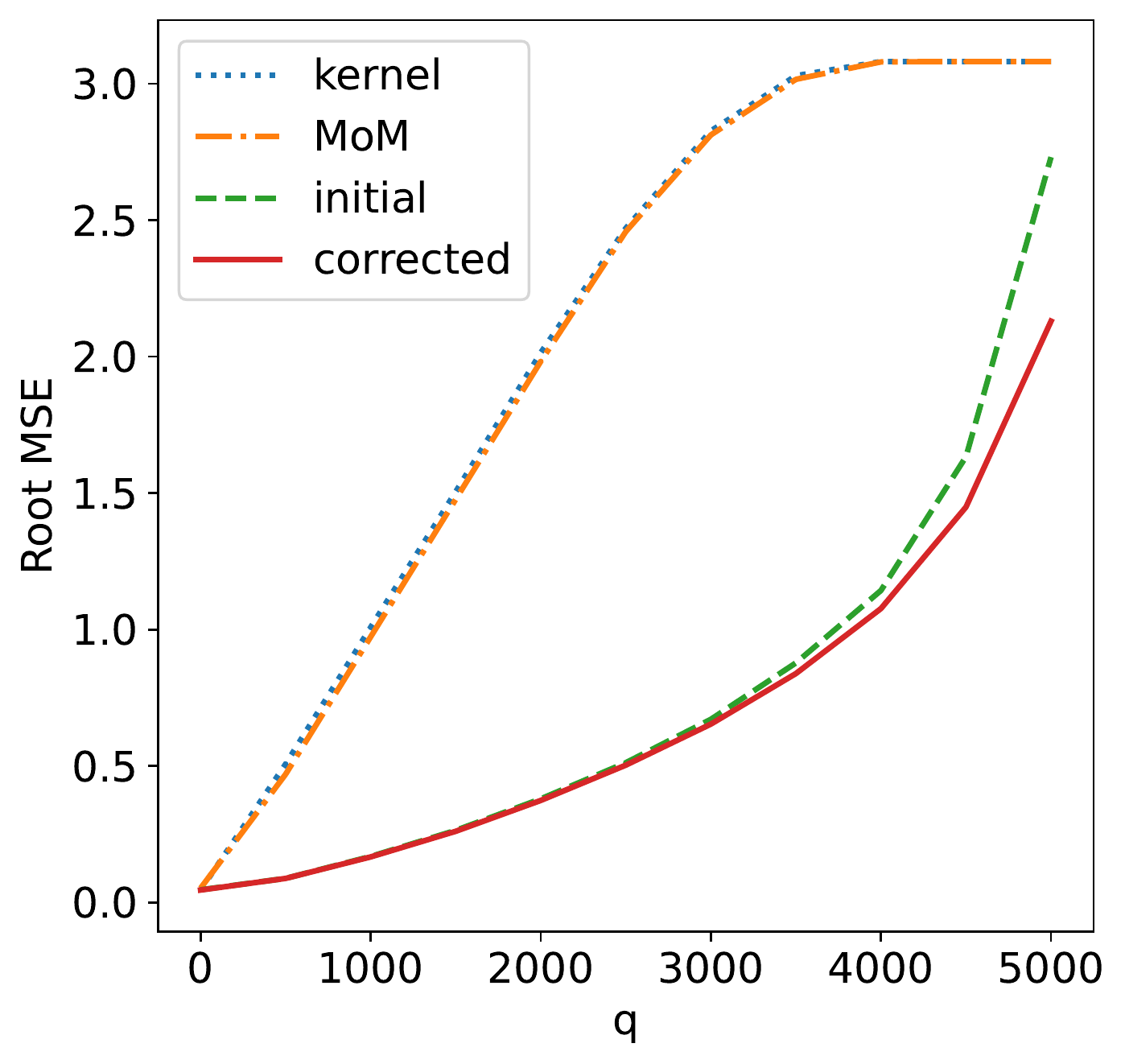}
		\caption{Squared root of $\ell_2$ error, one directional attack.}
	\end{subfigure}	
	\begin{subfigure}{0.32\linewidth}
		\includegraphics[width=\textwidth,height=0.76\textwidth]{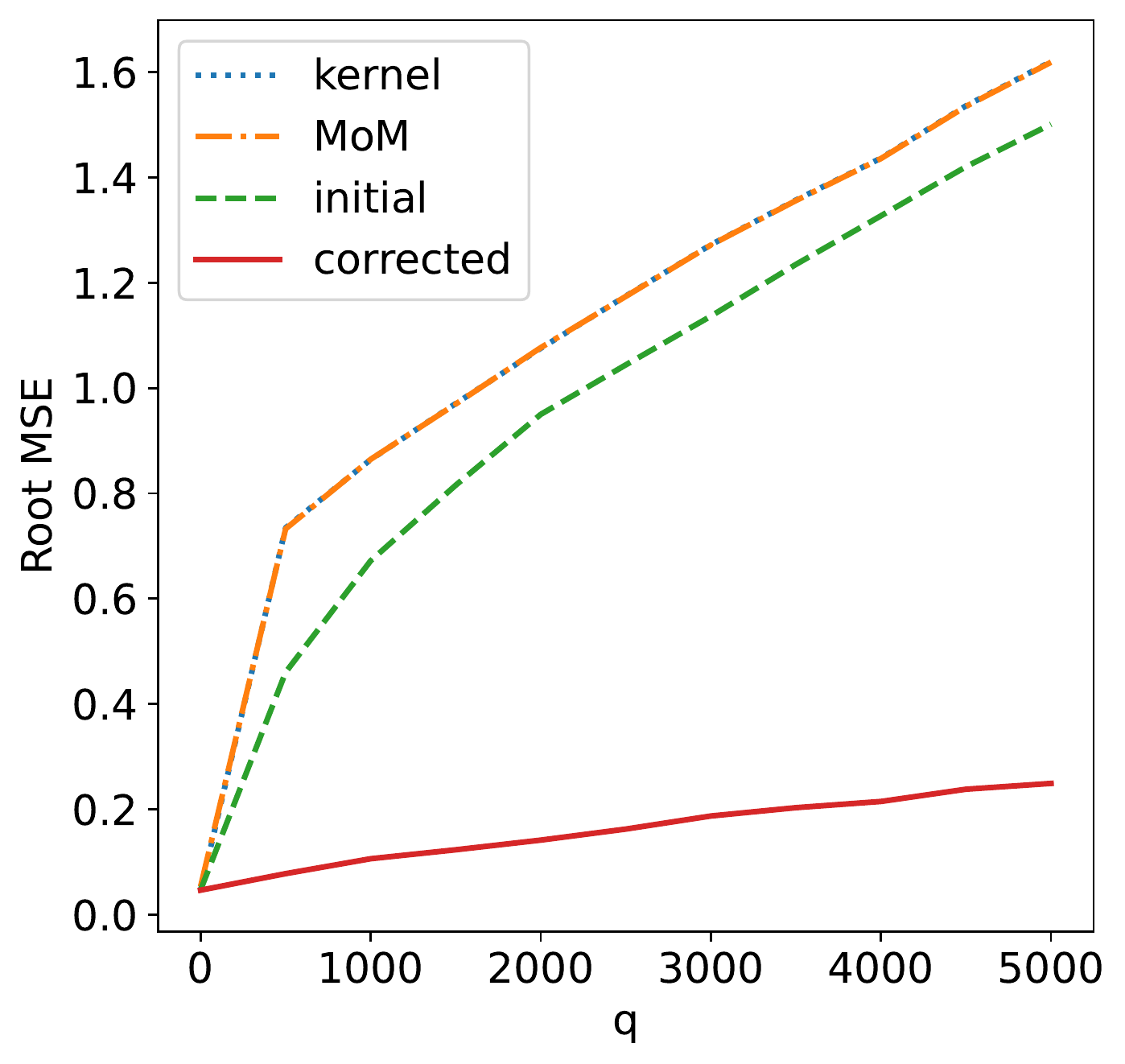}
		\caption{Squared root of $\ell_2$ error, concentrated attack.}
	\end{subfigure}	
	\hfill
	\begin{subfigure}{0.32\linewidth}
		\includegraphics[width=\textwidth,height=0.76\textwidth]{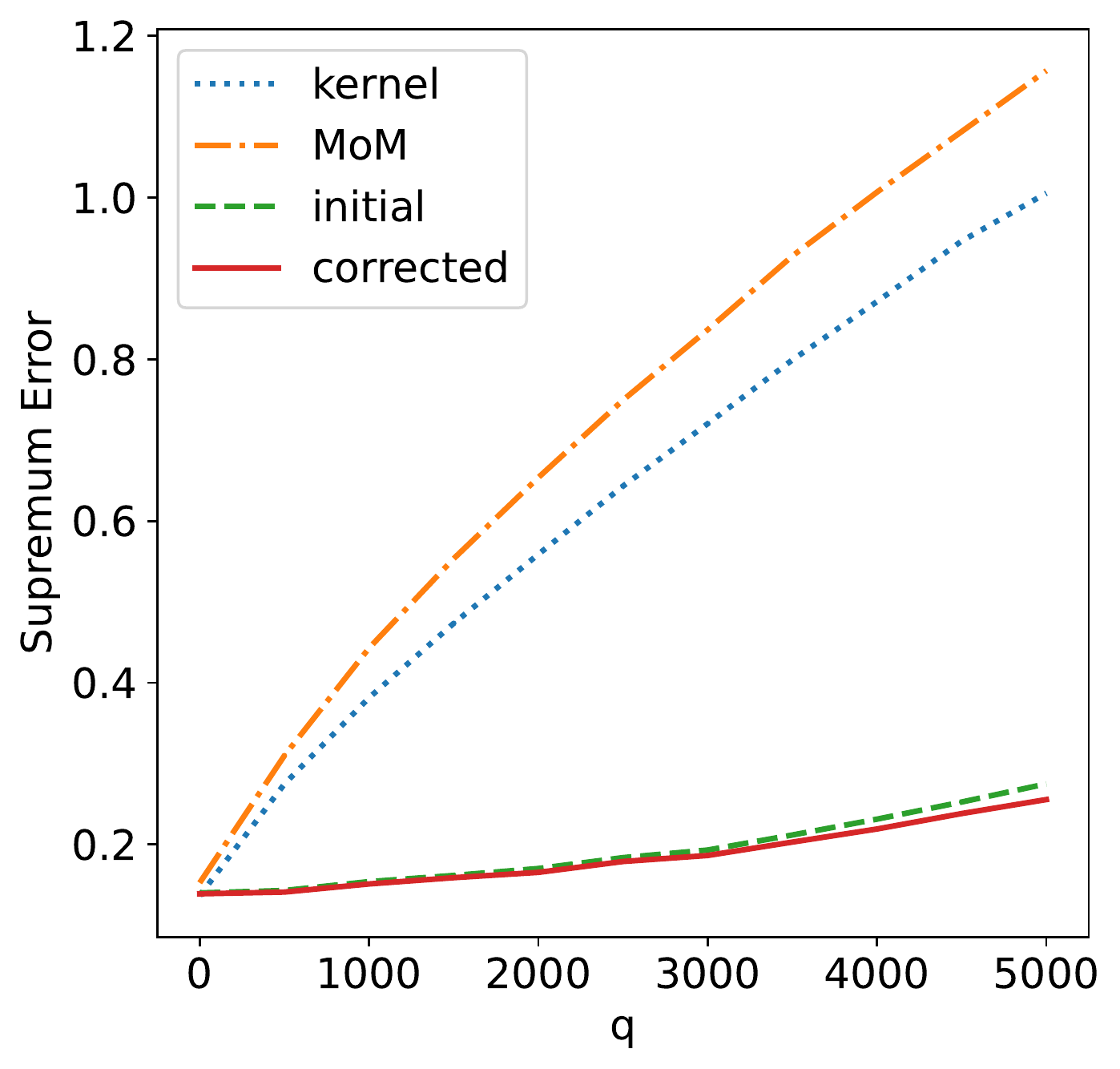}
		\caption{$\ell_\infty$ error, random attack.}
	\end{subfigure}
	\begin{subfigure}{0.32\linewidth}
		\includegraphics[width=\textwidth,height=0.76\textwidth]{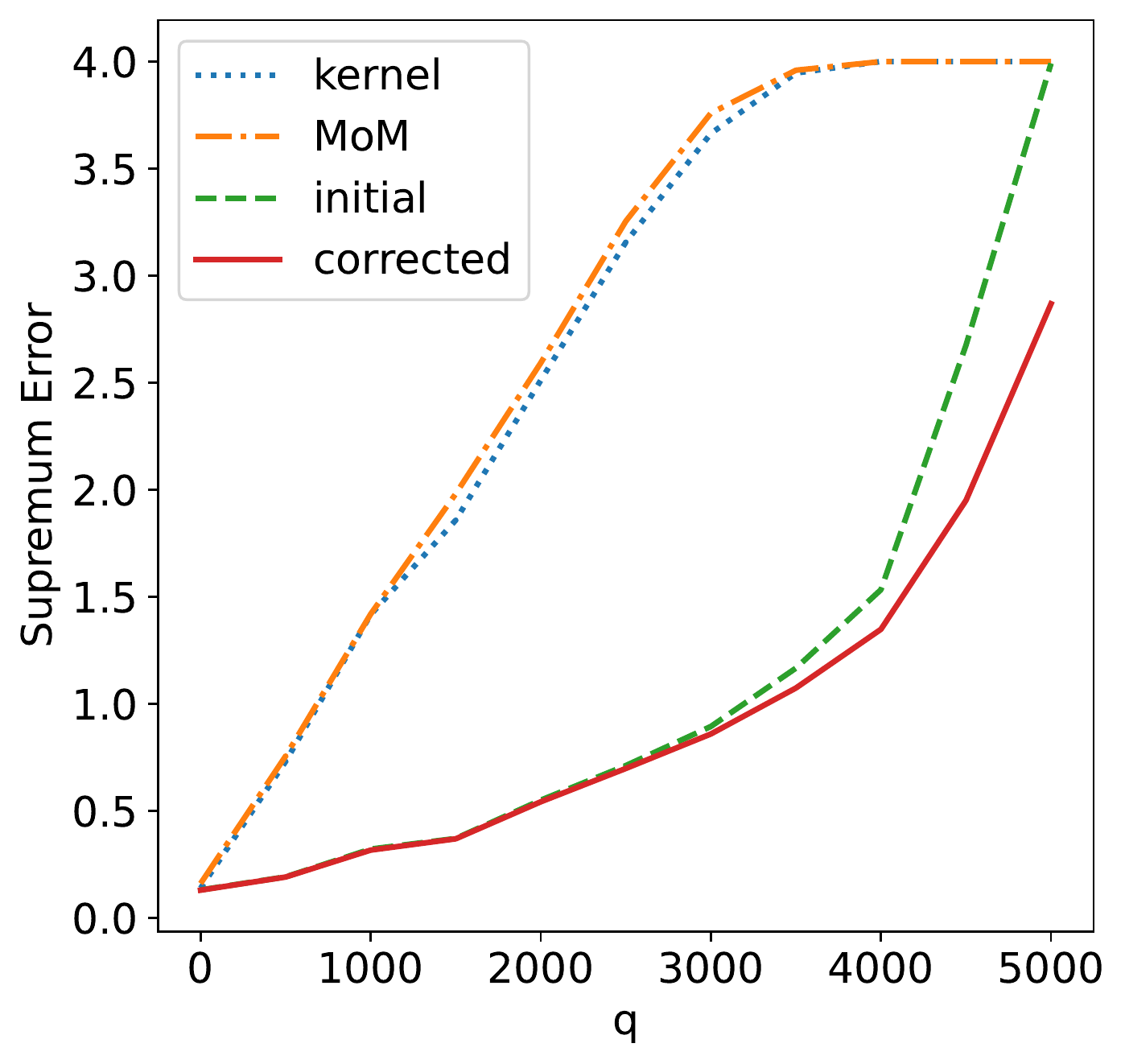}
		\caption{$\ell_\infty$ error, one directional attack.}
	\end{subfigure}
	\begin{subfigure}{0.32\linewidth}
		\includegraphics[width=\textwidth,height=0.76\textwidth]{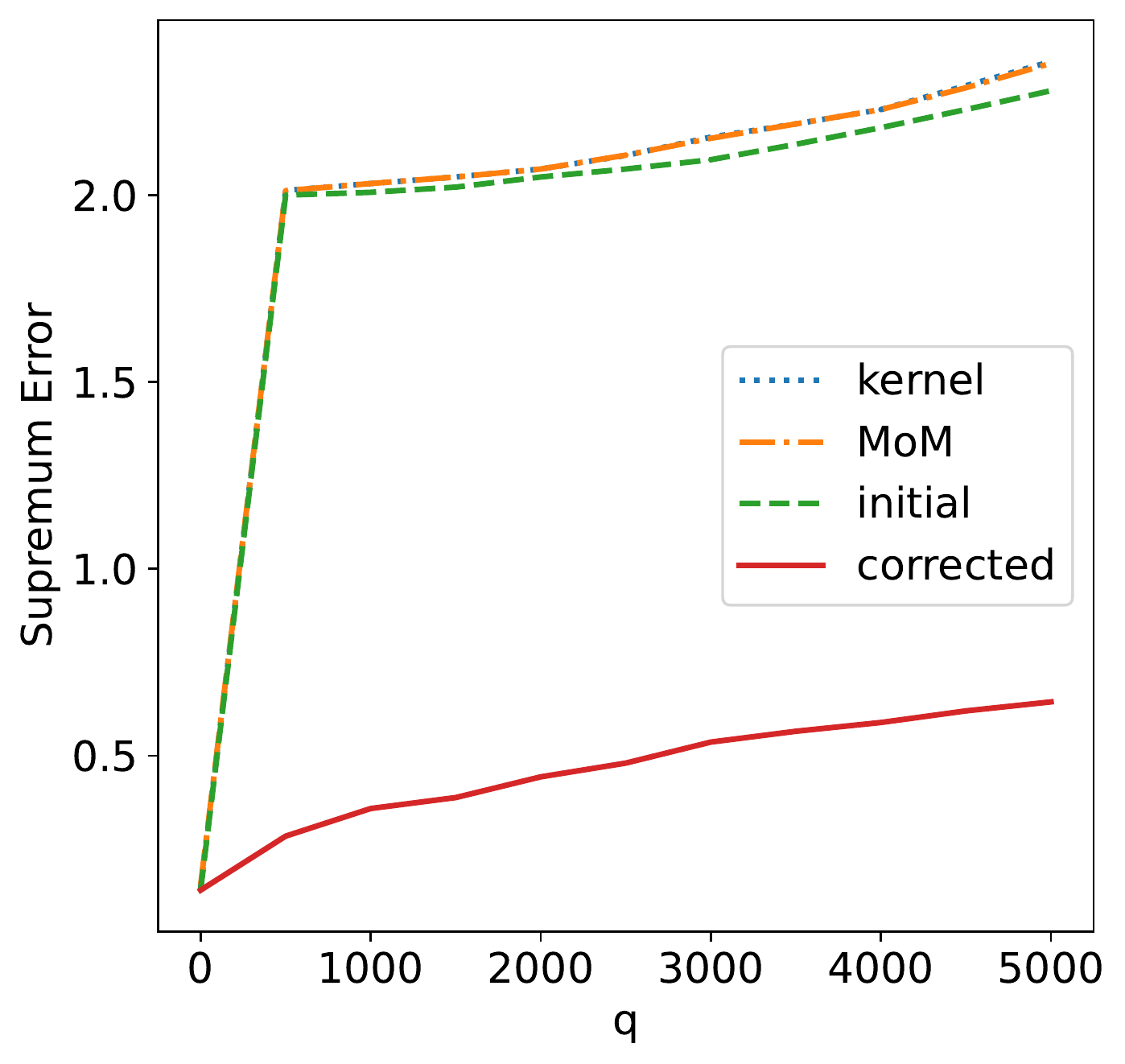}
		\caption{$\ell_\infty$ error, concentrated attack.}
	\end{subfigure}
	\caption{Comparison of $\ell_2$ and $\ell_\infty$ error between various methods for $d=1$.}
	\label{fig:err_1d}
	\vspace{-5mm}
\end{figure}
For one dimensional distribution, let the ground truth be
\begin{eqnarray}
	\eta_1(x) = \sin(2\pi x).
\end{eqnarray}
For two dimensional distribution,
\begin{eqnarray}
	\eta(\mathbf{x}) = \sin(2\pi x_1) + \cos(2\pi x_2).
\end{eqnarray}

The noise follows standard Gaussian distribution $\mathcal{N}(0,1)$. The performances are evaluated using square root of $\ell_2$ error, as well as $\ell_\infty$ error. The results are shown in Figure \ref{fig:err_1d} and \ref{fig:err_2d} for one and two dimensional distributions, respectively. In these figures, each point is the average over $1000$ independent trials. 

\begin{figure}[h!]
	\centering
	\begin{subfigure}{0.32\linewidth}
		\includegraphics[width=\textwidth,height=0.76\textwidth]{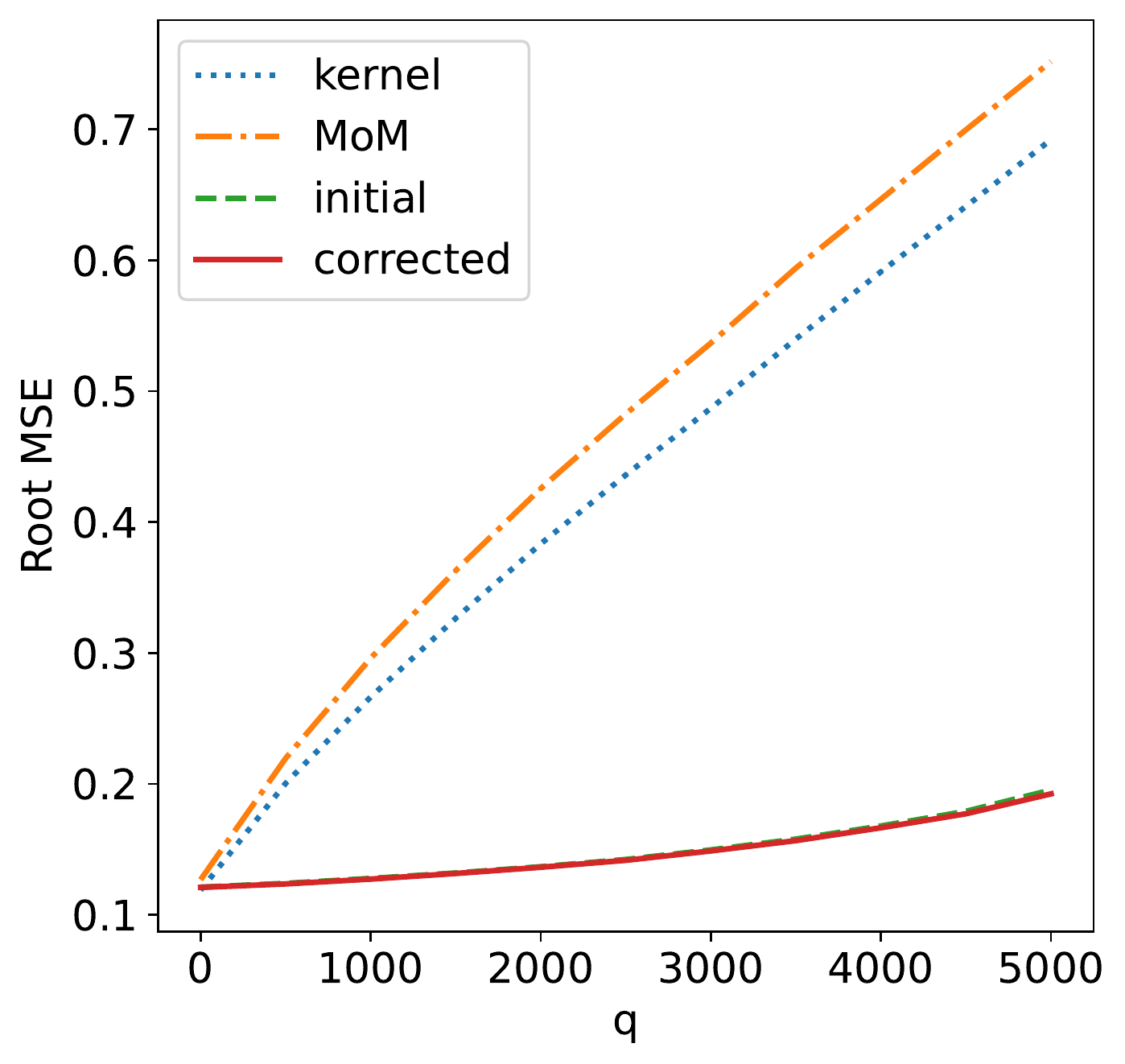}
		\caption{Squared root of $\ell_2$ error, random attack.}
	\end{subfigure}
	\begin{subfigure}{0.32\linewidth}
		\includegraphics[width=\textwidth,height=0.76\textwidth]{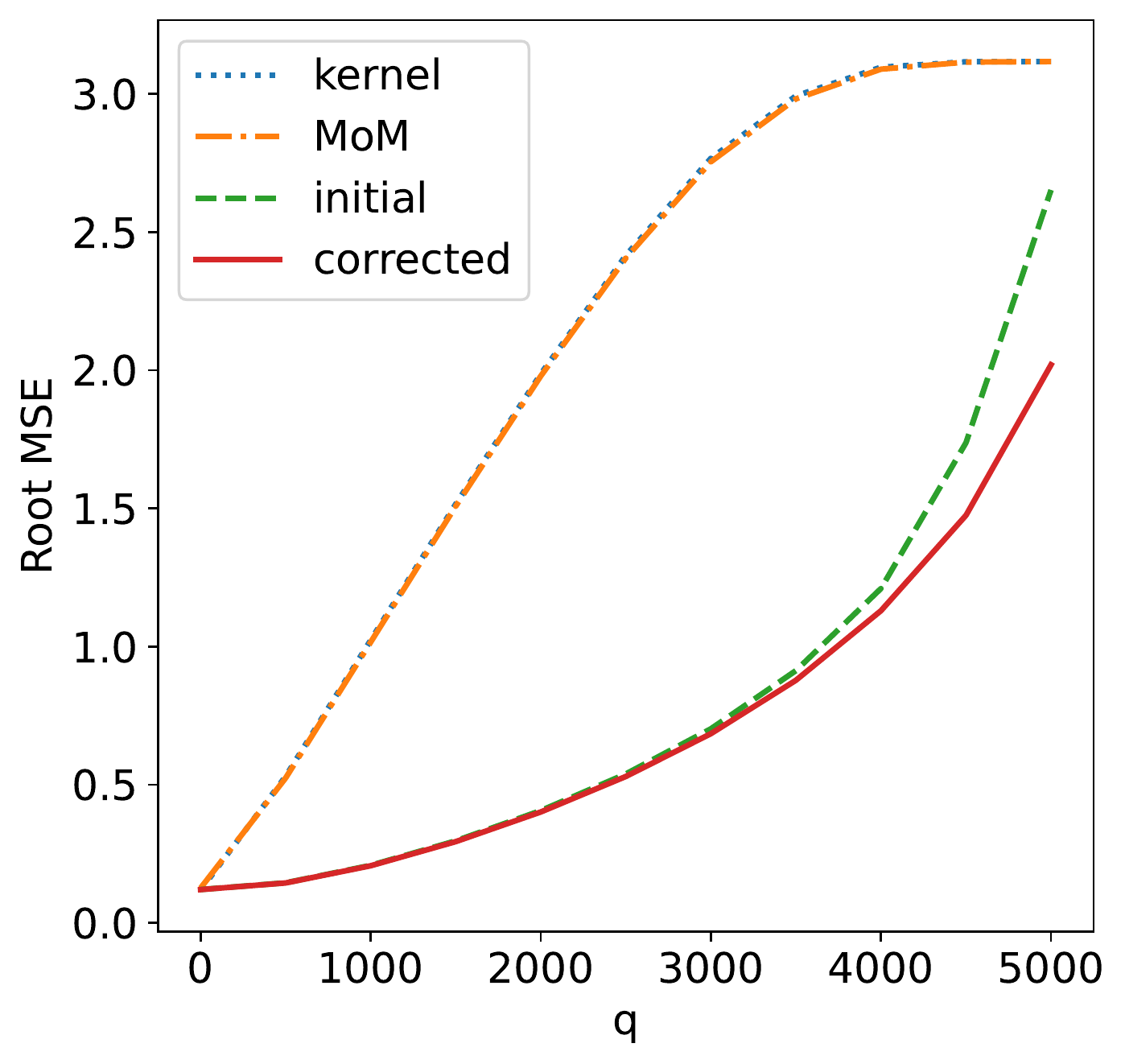}
		\caption{Squared root of $\ell_2$ error, one directional attack.}
	\end{subfigure}	
	\begin{subfigure}{0.32\linewidth}
		\includegraphics[width=\textwidth,height=0.76\textwidth]{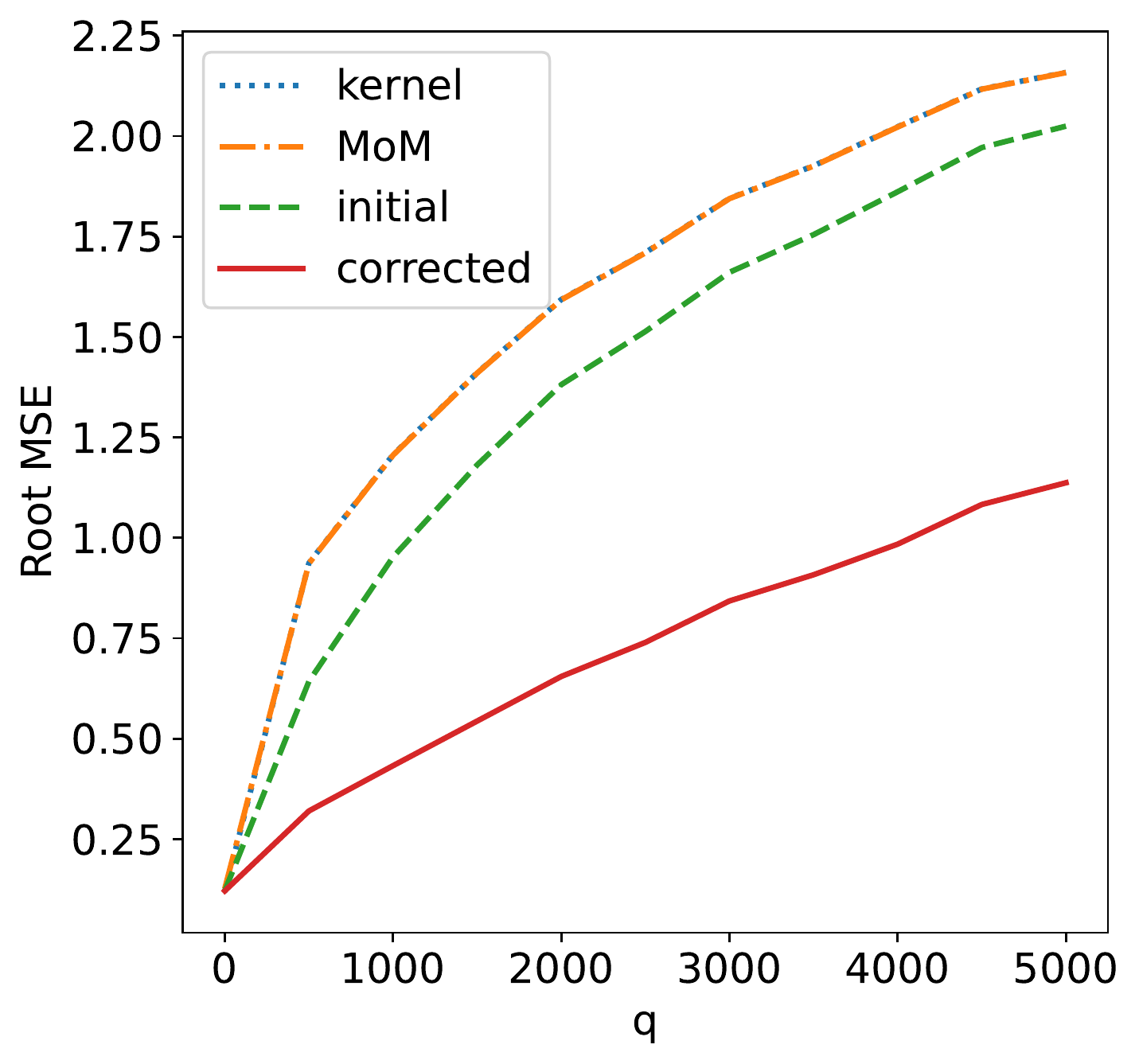}
		\caption{Squared root of $\ell_2$ error, concentrated attack.}
	\end{subfigure}	
	\hfill
	\begin{subfigure}{0.32\linewidth}
		\includegraphics[width=\textwidth,height=0.76\textwidth]{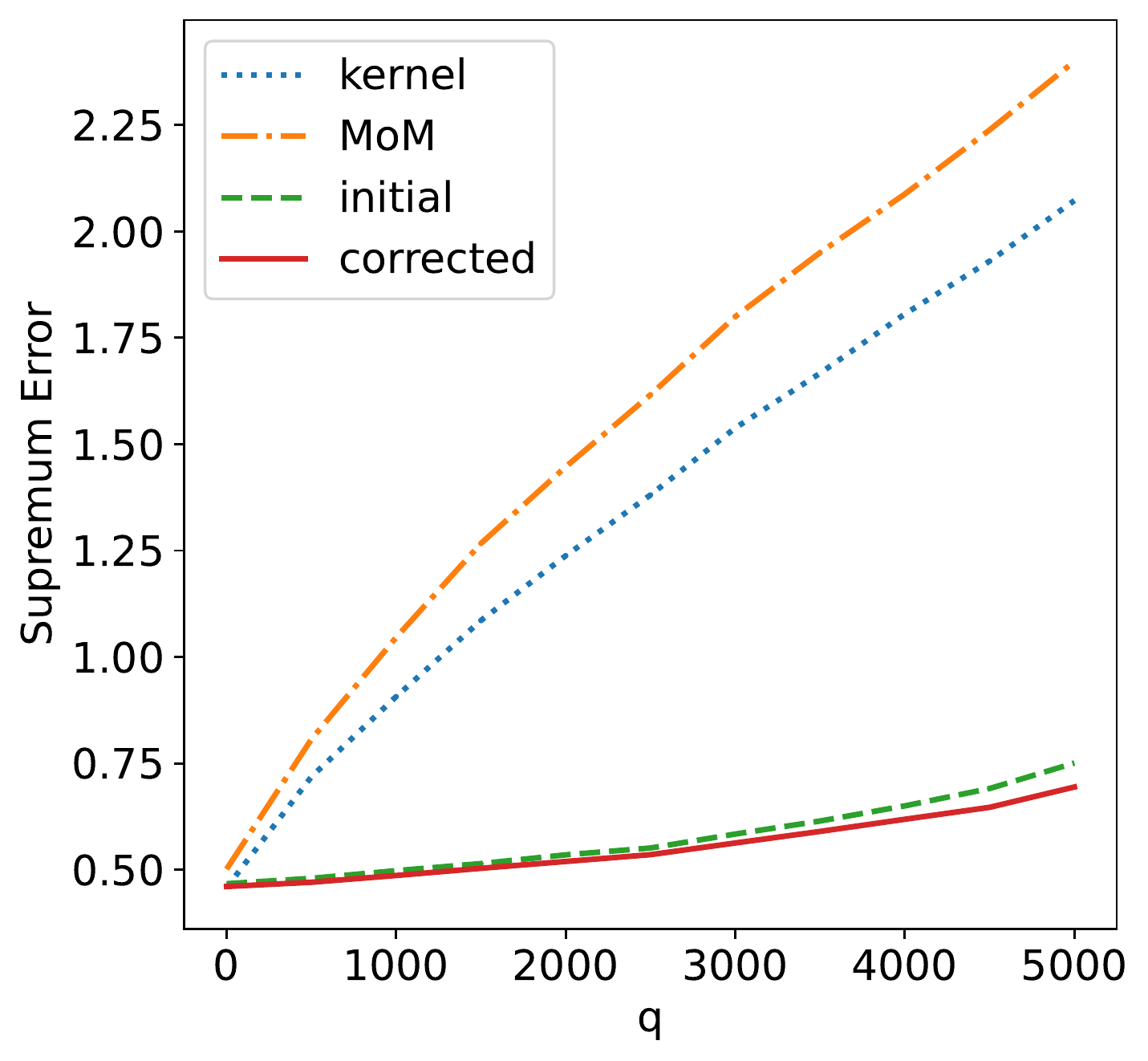}
		\caption{$\ell_\infty$ error, random attack.}
	\end{subfigure}
	\begin{subfigure}{0.32\linewidth}
		\includegraphics[width=\textwidth,height=0.76\textwidth]{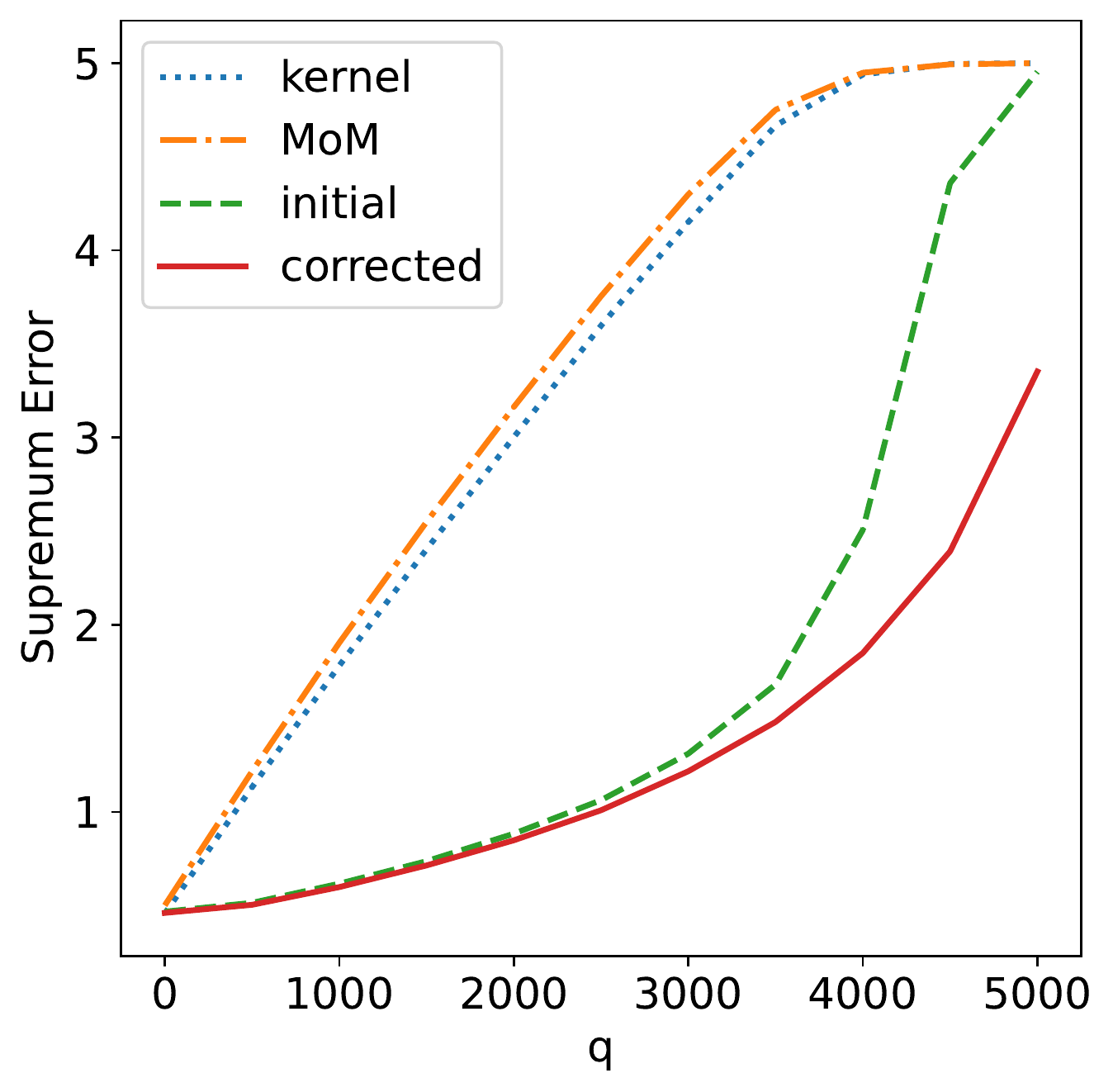}
		\caption{$\ell_\infty$ error, one directional attack.}
	\end{subfigure}
	\begin{subfigure}{0.32\linewidth}
		\includegraphics[width=\textwidth,height=0.76\textwidth]{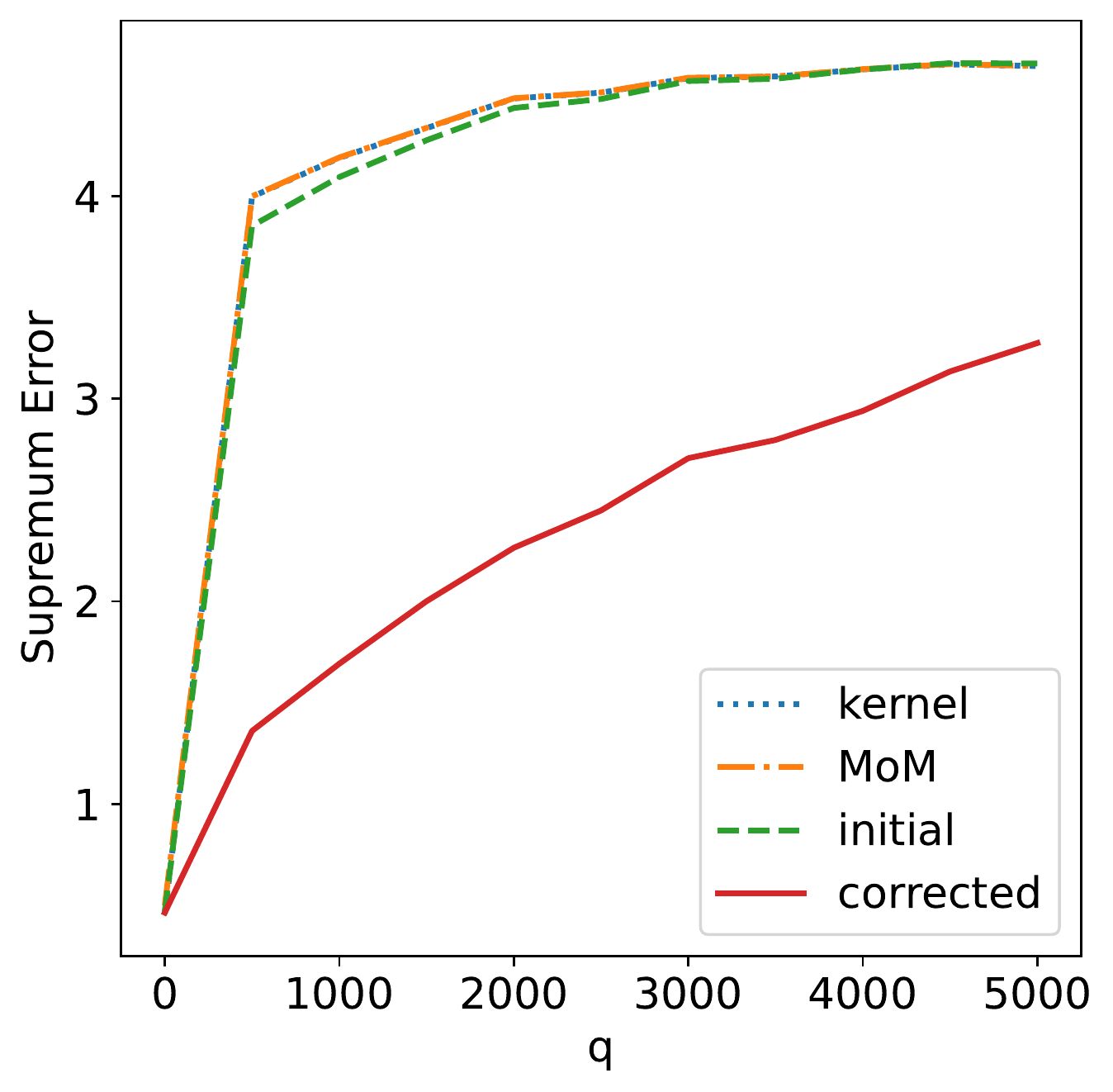}
		\caption{$\ell_\infty$ error, concentrated attack.}
	\end{subfigure}
	\caption{Comparison of $\ell_2$ and $\ell_\infty$ error between various methods for $d=2$.}
	\label{fig:err_2d}
	\vspace{-6mm}
\end{figure}
Figure \ref{fig:err_1d} and \ref{fig:err_2d} show that the simple kernel regression (blue dotted line) fails under poisoning attack. The $\ell_2$ and $\ell_\infty$ error grows fast with the increase of $q$. Besides, traditional median-of-means (orange dash-dot line) does not improve over kernel regression. Trimmed mean estimator works well under random or one directional attack with small $q$, but fails otherwise. Moreover, the initial estimator \eqref{eq:eta} (purple dashed line) shows significantly better performance than kernel estimator under random and one directional attack, as are shown in Fig.\ref{fig:err_1d} and \ref{fig:err_2d}, (a), (b), (d), (e). However, if the attacked samples concentrate around some centers, then the initial estimator fails. Compared with kernel regression, there is some but limited improvement for \eqref{eq:eta}. Finally, the corrected estimator (red solid line) performs well under all attack strategies. Under random attack, the corrected estimator performs nearly the same as initial one. For one directional attack, the corrected estimator performs better than the initial one with large $q$. Under concentrated attack, the correction shows significant improvement. These results are consistent with our theoretical analysis.

We have also conducted numerical experiments using real data. In particular, we obtain and compare the root MSE score of the median-of-means, trimmed mean, our initial estimator and the corrected estimator under all three types of attacks. All experiments show that our methods have desirable performance. The initial estimator significantly improves over median-of-means and trimmed mean estimator. The performance is further improved using our correction technique. The detailed implementation and results are shown in section \ref{sec:numadd} in the appendix.



\section{Conclusion}\label{sec:conc}

In this paper, we have provided a theoretical analysis of robust nonparametric regression problem under adversarial attack. In particular, we have derived the convergence rate of an M-estimator based on Huber loss minimization. We have also derived the information theoretic minimax lower bound, which is the underlying limit of robust nonparametric regression. The result shows that the initial estimator has minimax optimal $\ell_\infty$ risk. With small $q$, which is the number of adversarial samples, $\ell_2$ risk is also optimal. However, for large $q$, the initial estimator becomes suboptimal. Finally, we have proposed a correction technique, which is a nonlinear filter that projects the estimated function into the space of Lipschitz functions. Our theoretical analysis shows that the corrected estimator is minimax optimal even for large $q$. Numerical experiments on both synthesized and real data validate our theoretical analysis.

\bibliographystyle{nips}
\bibliography{robust_regression}
\newpage
\appendix

In this appendix, section \ref{sec:implement} shows our implementation of the corrected estimator. Other sections are proofs of the theoretical results.  Throughout this appendix, capital letters are random variables, while lowercase letters are their values.

\section{Implementation of Corrected Estimator}\label{sec:implement}
The corrected estimation is
\begin{eqnarray}
	\hat{\eta}_c = \underset{\|\nabla g\|_\infty\leq L}{\arg\min} \|\hat{\eta}_0-g\|_1.
	\label{eq:optim_app}
\end{eqnarray}

In this section, we find a approximate numerical solution instead. In particular, we only optimize $g$ at grid points that are dense enough, and the values elsewhere can be simply calculated via interpolation. The grid points are set to be $\mathbf{x}_{j_1,\ldots, j_d} = \mathbf{x}_0 + (j_1,\ldots, j_d)a$, with indices $j_k\in \{1,\ldots, m_k \}$, $m_k$ is the grid count along $k$-th dimension. $\mathbf{x}_0$ and $m_k$ need to satisfy
\begin{eqnarray}
	x_{0k}\leq \underset{\mathbf{x}\in \mathcal{X}}{\inf}x_k,\\
	x_{0k}+m_ka\geq \underset{\mathbf{x}\in \mathcal{X}}{\sup}x_k,
\end{eqnarray}
so that these grid points cover the whole support. $a$ is the grid size.  Denote $\mathbf{j}=(j_1,\ldots, j_d)$, $g_\mathbf{j} = g(\mathbf{x}_\mathbf{j})$, and $r_\mathbf{j}=\hat{\eta}_0(\mathbf{x}_\mathbf{j})$. Then the discretized optimization problem can be formulated as following:
\begin{mini}
	{g}{\sum_\mathbf{j}|g_\mathbf{j}-r_\mathbf{j}|}{}{}
	\addConstraint{|g_\mathbf{j}-g_{\mathbf{j}'}| \leq La,}{}{\forall |\mathbf{j}'-\mathbf{j}|=1,}
	\label{eq:discrete}
\end{mini}
in which $|\mathbf{j}'-\mathbf{j}|=\sum_{k=1}^d |j_k'-j_k|$. With sufficiently small $a$, the discretized problem approximates \eqref{eq:optim} well. \eqref{eq:discrete} can be solved simply by optimizing each $g_\mathbf{j}$ iteratively.

\section{Proof of Theorem 1: $\ell_2$ Convergence of Initial Estimator}\label{sec:l2}
This section proves the convergence rate of the initial estimator
\begin{eqnarray}
	\hat{\eta}_0(\mathbf{x}) = 
	\underset{|s|\leq M}{\arg\min}\sum_{i=1}^N K\left(\frac{\mathbf{x}-\mathbf{X}_i}{h}\right)\phi(Y_i - s).
	\label{eq:eta_app}
\end{eqnarray}
To begin with, we use the following notations.
\begin{defi}
	Define
	\begin{eqnarray}
		B_h(\mathbf{x})=\{\mathbf{u}| \|\mathbf{u}-\mathbf{x}\|<h \}
	\end{eqnarray}
	as the ball centering at $\mathbf{x}$ with radius $h$,
	\begin{eqnarray}
		q_h(\mathbf{x}) = |\{i|i\in \mathcal{B}, \mathbf{X}_i\in B_h(\mathbf{x}) \}|
	\end{eqnarray}
	as the number of attacked samples within $B_h(\mathbf{x})$,
	\begin{eqnarray}
		I_h(\mathbf{x}) = \{i|\mathbf{X}_i\in B_h(\mathbf{x}) \}
		\label{eq:Ih}
	\end{eqnarray}
	as the set of the indices of all samples within $B_h(\mathbf{x})$, and
	\begin{eqnarray}
		n_h(\mathbf{x}) = |I_h(\mathbf{x})|
		\label{eq:nh}
	\end{eqnarray}
	as the total number of samples within $B_h(\mathbf{x})$.
\end{defi}
\begin{defi}\label{def:ab}
	Define $a(\mathbf{x})$, $b(\mathbf{x})$ as
	\begin{eqnarray}
		a(\mathbf{x}) = \underset{|s|\leq M}{\arg\min} \sum_{i\in I_h(\mathbf{x})}K\left(\frac{\mathbf{x}-\mathbf{X}_i}{h}\right)\phi(\eta(\mathbf{X}_i)+W_i-s),
		\label{eq:adef}
	\end{eqnarray}
	\begin{eqnarray}
		b(\mathbf{x})= \underset{|s|\leq M}{\arg\min} \sum_{i\in I_h(\mathbf{x})}K\left(\frac{\mathbf{x}-\mathbf{X}_i}{h}\right)(\eta(\mathbf{X}_i)+W_i-s)^2,
		\label{eq:bdef}
	\end{eqnarray}
\end{defi}
$a(\mathbf{x})$ is the estimated value with no adversarial attacks. $b(\mathbf{x})$ is just the ordinary kernel regression estimates clipped into $[-M, M]$. Then
\begin{eqnarray}
	|\hat{\eta}_0(\mathbf{x}) - \eta(\mathbf{x})|\leq |\hat{\eta}_0(\mathbf{x}) - a(\mathbf{x})| + |a(\mathbf{x}) - b(\mathbf{x})|+|b(\mathbf{x})-\eta(\mathbf{x})|.
\end{eqnarray}
Note that $a(\mathbf{x})$ and $b(\mathbf{x})$ is not affected by the behavior of the attacker. Hence
\begin{eqnarray}
	R &=& \mathbb{E}\left[\underset{\mathcal{A}}{\sup} (\hat{\eta}_0(\mathbf{X}) - \eta(\mathbf{X}))^2\right]\nonumber\\
	&\leq & 3\mathbb{E}\left[\underset{\mathcal{A}}{\sup} (\hat{\eta}_0(\mathbf{X}) - a(\mathbf{X}))^2\right] + 3\mathbb{E}\left[(a(\mathbf{X})-b(\mathbf{X})^2)\right] + 3\mathbb{E}\left[(b(\mathbf{X}) - \eta(\mathbf{X}))^2\right]\nonumber\\
	&:= & 3(I_1+I_2+I_3),
	\label{eq:R}
\end{eqnarray}
now we bound these three terms separately.

\textbf{Bound of $I_1$.} Define a new random variable
\begin{eqnarray}
	Z = 2Lh + \underset{i\in [N]}{\max}W_i - \underset{i \in [N]}{\min} W_i,
	\label{eq:Z}
\end{eqnarray}
in which $[N] = \{1,\ldots, N\}$. Then $Z$ can be bounded using the following lemma.

\begin{lem}\label{lem:Z}
	If $t>2Lh$, then
	\begin{eqnarray}
		\text{P}(Z>t)\leq 2\exp\left[-\min\left\{\frac{(t-2Lh)^2}{8\sigma^2}, \frac{t-2Lh}{4\sigma}\right\}+\ln N\right],
		\label{eq:Zbound}
	\end{eqnarray}
	and for $t>2Lh + 4\sigma \ln N$, 
	\begin{eqnarray}
		\mathbb{E}[Z^2 \mathbf{1}(Z>t)]\leq 2N(t^2+32\sigma^2 + 8t\sigma)e^{\frac{t-2Lh}{4\sigma}}.
		\label{eq:zsq}
	\end{eqnarray}
\end{lem}
Given $Z$, $|\hat{\eta}_0(\mathbf{x}) - a(\mathbf{x})|$ can be bounded. Define
\begin{eqnarray}
	r(\mathbf{x}) = \frac{C_Kq_h(\mathbf{x})}{c_K(n_h(\mathbf{x})-q_h(\mathbf{x}))},
	\label{eq:rx}
\end{eqnarray}
in which $n_h$ is defined in \eqref{eq:nh}, and
\begin{eqnarray}
	n_0=\frac{1}{2}\alpha f_m v_dh^dN,
	\label{eq:n0}
\end{eqnarray} 
in which $v_d$ is the volume of $d$ dimensional unit ball.

Then the following lemmas hold:
\begin{lem}\label{lem:e1}	
	If $r(\mathbf{x})\leq (T-Z)/(T+Z)$, then $|\hat{\eta}_0(\mathbf{x})-a(\mathbf{x})|\leq (T+Z)r(\mathbf{x})$.
\end{lem}
\begin{lem}\label{lem:e2}	
	For any $\mathbf{x}\in \mathcal{X}$, under the following three conditions:
	
	(a) $n_h(\mathbf{x})\geq n_0$;
	
	(b) $Z\leq 2Lh+8\sigma \ln N$;
	
	(c) $q_h(\mathbf{x})\leq c_Kn_0/(3C_K+c_K)$,
	
	then
	\begin{eqnarray}
		|\hat{\eta}_0(\mathbf{x}) - a(\mathbf{x})|\leq \frac{2TC_Kq_h(\mathbf{x})}{c_Kn_0}.
	\end{eqnarray}
	
\end{lem}

Moreover, since $|\hat{\eta}_0(\mathbf{x})|\leq M$, and according to Assumption 1(b), $|\eta(\mathbf{x})|\leq M$, $|\hat{\eta}_0(\mathbf{x})|\leq 2M$ always hold, regardless of whether the conditions (a)-(c) in Lemma \ref{lem:e2} are satisfied. Therefore
\begin{eqnarray}
	I_1 &=& \mathbb{E}\left[\underset{\mathcal{A}}{\sup}(\hat{\eta}_0(\mathbf{X}) - a(\mathbf{X}))^2\right]\nonumber\\
	&=&\mathbb{E}\left[\underset{\mathcal{A}}{\sup}\int (\hat{\eta}_0(\mathbf{x})- a(\mathbf{x}))^2 f(\mathbf{x})d\mathbf{x}\right]\nonumber\\
	&\leq & \frac{4C_K^2 T^2}{c_K^2 n_0^2}f_M\underset{\mathcal{A}}{\sup}\int q_h^2(\mathbf{x})\mathbf{1}\left(q_h(\mathbf{x})\leq \frac{c_kn_0}{3C_K+c_K} \right)d\mathbf{x} + 4M^2\left[\text{P}(n_h(\mathbf{x})<n_0)\right.\nonumber\\
	&&\left.+\text{P}(Z>2Lh+8\sigma \ln N)+\underset{\mathcal{A}}{\sup}\text{P}\left(q_h(\mathbf{X})>\frac{c_K}{3C_K+c_K}n_0\right)\right].
	\label{eq:I1}	
\end{eqnarray}
Now we bound each term separately. For the first term in \eqref{eq:I1},
\begin{eqnarray}
	\underset{\mathcal{A}}{\sup}\int q_h^2(\mathbf{x})d\mathbf{x} &\leq& \underset{\mathcal{A}}{\sup}\int \left(\sum_{i\in \mathcal{B}}\mathbf{1}(\| \mathbf{x}-\mathbf{X}_i\|<h)\right)^2 d\mathbf{x}\nonumber\\
	&\leq & |\mathcal{B}|\underset{\mathcal{A}}{\sup}\int \sum_{i\in \mathcal{B}}\mathbf{1}(\|\mathbf{x}-\mathbf{X}_i\|<h)d\mathbf{x}\nonumber\\
	&=& |\mathcal{B}|^2 v_dh^d=q^2v_dh^d.
	\label{eq:qb1}
\end{eqnarray}
Moreover, 
\begin{eqnarray}
	\underset{\mathcal{A}}{\sup}\int q_h^2(\mathbf{x}) \mathbf{1}\left(q_h(\mathbf{x})\leq \frac{c_Kn_0}{3C_K+c_K}\right) d\mathbf{x}&\leq& \frac{c_Kn_0}{3C_K+c_K}\underset{\mathcal{A}}{\sup}\int q_h(\mathbf{x})d\mathbf{x}\nonumber\\
	&=&\frac{c_Kn_0}{3C_K+c_K}qv_dh^d.
	\label{eq:qb2}
\end{eqnarray}
Therefore, combine \eqref{eq:qb1} and \eqref{eq:qb2}, the first term in \eqref{eq:I1} can be bounded by
\begin{eqnarray}
	\frac{4C_K^2 T^2}{c_K^2n_0^2}f_M \underset{\mathcal{A}}{\sup}\int q_h^2(\mathbf{x})\mathbf{1}\left(q_h(\mathbf{x})\leq \frac{c_Kn_0}{3C_K+c_K} \right)d\mathbf{x}&\lesssim& \frac{T^2}{n_0^2} \min\left\{q^2 h^d, n_0qh^d \right\}\nonumber\\
	&\sim& \frac{T^2q}{N}\min\left\{\frac{q}{Nh^d}, 1 \right\}.
	\label{eq:term1}
\end{eqnarray}

For the second term in \eqref{eq:I1}, we need to bound $\text{P}(n_h(\mathbf{x})<n_0)$ and $\text{P}(Z>2Lh+8\sigma\ln N)$. Note that with sufficiently large $N$, we have $h<D$. Hence
\begin{eqnarray}
	\mathbb{E}[n_h(\mathbf{x})]&\overset{(a)}{=}&\sum_{i=1}^N \text{P}(\mathbf{X}_i\in B_h(\mathbf{x}))\nonumber\\
	&\overset{(b)}{\geq}& Nf_mV(B_h(\mathbf{x})\cap \mathcal{X})\nonumber\\
	&\overset{(c)}{\geq}& Nf_m\alpha v_dh^d =2n_0,	
\end{eqnarray}
in which (a) comes from definitions \eqref{eq:nh} and \eqref{eq:Ih}, (b) comes from Assumption 1(b), (c) comes from Assumption 1(c). From Chernoff inequality\footnote{Here we use this version of Chernoff inequality: For i.i.d binary random variables $U_1,\ldots, U_N$ with $\text{P}(U_i=1) = p$, $\text{P}(U_i=0)=1-p$, $i=1,\ldots, N$, if $k<Np$, then $\text{P}(\sum_{i=1}^N U_i<k)\leq e^{-Np}(eNp/k)^k$.}, denote
\begin{eqnarray}
	p_h(\mathbf{x})=\int_{B_h(\mathbf{x})} f(\mathbf{u})d\mathbf{u},
\end{eqnarray}
then
\begin{eqnarray}
	p_h(\mathbf{x})\geq f_m\alpha v_dh^d=\frac{2n_0}{N}.
\end{eqnarray}
Therefore
\begin{eqnarray}
	\text{P}(n_h(\mathbf{x})<n_0)&\leq& e^{-Np_h(\mathbf{x})}\left(\frac{eNp_h(\mathbf{x})}{n_0}\right)^{n_0}\nonumber\\
	&\leq& e^{-2n_0}(2e)^{n_0}\nonumber\\
	&\leq & e^{-(1-\ln 2)n_0}.
	\label{eq:smalln}
\end{eqnarray}
Assumption 3 requires $h>\ln^2 N/N$. Recall \eqref{eq:n0}, $n_0\sim Nh^d\gtrsim \ln^2 N$, thus \eqref{eq:smalln} decays faster than any polynomial of $n_0$. For the third term in \eqref{eq:I1}, from Lemma \ref{lem:Z}, 
\begin{eqnarray}
	\text{P}(Z>2Lh+8\sigma \ln N)\leq \frac{2}{N};
	\label{eq:term3}
\end{eqnarray}
Finally, for the last term in \eqref{eq:I1},
\begin{eqnarray}
	\text{P}\left(q_h(\mathbf{X})>\frac{c_K}{3C_K+c_K}n_0\right)&\leq& \frac{\mathbb{E}[q_h^2(\mathbf{X})]}{\left(\frac{c_K}{3C_K+c_K}n_0\right)^2}\nonumber\\
	&\leq& \frac{(3C_K+c_K)^2 f_M}{c_K^2n_0^2}\int q_h^2(\mathbf{x})d\mathbf{x}\nonumber\\
	&\lesssim& \frac{q^2h^d}{n_0^2}\sim \frac{q^2}{N^2 h^d}.
	\label{eq:largeq1}
\end{eqnarray}
We can get the following alternative bound:
\begin{eqnarray}
	\text{P}\left(q_h(\mathbf{X})>\frac{c_K}{3C_K+c_K}n_0\right)\leq \frac{\mathbb{E}[q_h(\mathbf{X})]}{\frac{c_K}{3C_K+c_K}n_0}\leq \frac{3C_K+c_K}{c_K}\frac{qh^d}{n_0}\lesssim \frac{q}{N}.
	\label{eq:largeq2}
\end{eqnarray}
Therefore, combine \eqref{eq:largeq1} and \eqref{eq:largeq2},
\begin{eqnarray}
	\text{P}\left(q_h(\mathbf{X})>\frac{c_K}{3C_K+c_K}n_0\right)\lesssim \frac{q}{N}\min\left\{\frac{q}{Nh^d}, 1 \right\}.
	\label{eq:term4}
\end{eqnarray}
Now it remains to bound \eqref{eq:I1} using \eqref{eq:term1}, \eqref{eq:smalln}, \eqref{eq:term3} and \eqref{eq:term4}. This yields
\begin{eqnarray}
	I_1&\lesssim& \frac{T^2q}{N}\min\left\{\frac{q}{Nh^d}, 1 \right\}+e^{-(1-\ln 2)n_0}+\frac{2}{N}+\frac{q}{N}\min\left\{\frac{q}{Nh^d}, 1 \right\}\nonumber\\
	&\lesssim& \frac{T^2q}{N}\min\left\{\frac{q}{Nh^d}, 1 \right\}.
	\label{eq:result1}
\end{eqnarray}

\textbf{Bound of $I_2$.} 
\begin{lem}\label{lem:nodiff}
	If $Z\leq T$,  then $a(\mathbf{x})-b(\mathbf{x}) = 0$.
\end{lem}
Lemma \ref{lem:nodiff} will also be used later in other theorems.
\begin{proof}
	\begin{eqnarray}
		\underset{i\in I_h(\mathbf{x})}{\max}(\eta(\mathbf{X}_i) + W_i) - 	\underset{i\in I_h(\mathbf{x})}{\min}(\eta(\mathbf{X}_i) + W_i) \leq T.
	\end{eqnarray}
	Therefore $\phi(\eta(\mathbf{X})_i+W_i-s) = (\eta(\mathbf{X}_i)+W_i-s)^2$ for $\underset{i\in I_h(\mathbf{x})}{\min}(\eta(\mathbf{X}_i) + W_i) \leq s\leq \underset{i\in I_h(\mathbf{x})}{\max}(\eta(\mathbf{X}_i) + W_i)$. From \eqref{eq:adef} and \ref{eq:bdef}, $a(\mathbf{x})-b(\mathbf{x}) = 0$.
\end{proof}

If $Z>T$, \eqref{eq:adef} and \eqref{eq:bdef} gives $|a(\mathbf{x})-b(\mathbf{x})|\leq 2M$. Therefore, from Lemma \ref{lem:Z},
\begin{eqnarray}
	I_2 = \mathbb{E}[(a(\mathbf{X})-b(\mathbf{X}))^2] \leq 4M^2\text{P}(Z>T)\leq \frac{8M^2}{N^3}.
	\label{eq:result2}
\end{eqnarray}

\textbf{Bound of $I_3$.}
Since $W_i$ is sub-exponential, it is straightforward to show that the variance is bounded by $\sigma^2$:
\begin{eqnarray}
	\mathbb{E}[W_i^2]=\mathbb{E}\left[\underset{\lambda \rightarrow 0}{\lim} \frac{2}{\lambda^2}(e^{\lambda W_i} - 1 - \lambda W_i)\right]\leq \underset{\lambda \rightarrow 0}{\liminf}\frac{2}{\lambda^2}(e^{\frac{1}{2}\sigma^2\lambda^2} - 1)=\sigma^2,
\end{eqnarray}
in which we used Fatou's lemma in the second step. Then $I_3$ can simply be bounded by standard analysis of kernel regression \cite{krzyzak1986rates,mack1982weak,devroye1978uniform}. For the completeness of the paper, we provide a brief proof here. 

If $n_h(\mathbf{x})\geq n_0$, in which $n_0$ is defined in \eqref{eq:n0}, then with the Lipschitz assumption (Assumption 1(a)),
\begin{eqnarray}
	|b(\mathbf{x})-\eta(\mathbf{x})| &\leq& \left| \frac{\sum_{i\in I_h(\mathbf{x})}K\left(\frac{\mathbf{x}-\mathbf{X}_i}{h}\right) (\eta(\mathbf{X}_i)-\eta(\mathbf{x}))}{\sum_{i\in I_h(\mathbf{x})}K\left(\frac{\mathbf{x}-\mathbf{X}_i}{h}\right)}\right|+\left|\frac{\sum_{i\in I_h(\mathbf{x})}K\left(\frac{\mathbf{x}-\mathbf{X}_i}{h}\right) W_i}{\sum_{i\in I_h(\mathbf{x})}K\left(\frac{\mathbf{x}-\mathbf{X}_i}{h}\right)}\right|\nonumber\\
	&\leq &Lh +\frac{1}{n_0c_K}\left|\sum_{i\in I_h(\mathbf{x})}K\left(\frac{\mathbf{x}-\mathbf{X}_i}{h}\right)W_i\right|.
	\label{eq:ekernel}
\end{eqnarray}

If $n_h(\mathbf{x})<n_0$, then $|b(\mathbf{x})-\eta(\mathbf{x})|\leq 2M$. Since $\mathbb{E}[W_i] = 0$, $\mathbb{E}[W_i^2] \leq \sigma^2$, 
\begin{eqnarray}
	\mathbb{E}[(b(\mathbf{x})-\eta(\mathbf{x}))^2] \leq L^2 h^2+\frac{\sigma^2 C_k^2}{n_0c_k^2} + 4M^2 \text{P}(n_h(\mathbf{x})<n_0).
	\label{eq:I3}
\end{eqnarray}
Using \eqref{eq:smalln}, integrate \eqref{eq:I3} over the whole support, 
\begin{eqnarray}
	I_3\lesssim h^2 + \frac{1}{n_0}\sim h^2+\frac{1}{Nh^d}.
	\label{eq:result3}
\end{eqnarray}
Combine \eqref{eq:R}, \eqref{eq:result1}, \eqref{eq:result2} and \eqref{eq:result3},
\begin{eqnarray}
	R\lesssim \frac{T^2q}{N}\min\left\{\frac{q}{Nh^d}, 1 \right\} + h^2 + \frac{1}{Nh^d}.
\end{eqnarray}

\section{Beyond the Strong Density Assumption}\label{sec:beyond}
In this section, we discuss the bound without assuming that the pdf $f$ is bounded from below. However, it is still required that the support is bounded. The precise assumption is stated as following:
\begin{ass}\label{ass:new}
	Here we make the following assumptions.
	
	(a) $\mathcal{X}$ has bounded volume, i.e. $Vol(\mathcal{X}) = V<\infty$;
	
	(b) For any $\mathbf{x}\in \mathcal{X}$, $p_h(\mathbf{x})\geq \alpha f(\mathbf{x})v_dh^d$, in which $p_h(\mathbf{x}) = \int_{B_h(\mathbf{x}} f(\mathbf{u})d\mathbf{u}$.
\end{ass}

Now we analyze the $\ell_2$ error under Assumption \ref{ass:new}. Define 
\begin{eqnarray}
	n_0(\mathbf{x})=\frac{1}{2}\alpha f(\mathbf{x})v_dh^dN.
\end{eqnarray}
Then under the new assumption, Lemma \ref{lem:e2} is replaced by the following one:
\begin{lem}
	For any $\mathbf{x}\in \mathcal{X}$, under the following three conditions:
	
	(a) $n_h(\mathbf{x})\geq n_0(\mathbf{x})$;
	
	(b) $Z\leq 2Lh + 8\sigma \ln N$;
	
	(c) $q_h(\mathbf{x})\leq c_Kn_0(\mathbf{x}) / (3C_K+c_K)$.
	
	Then
	\begin{eqnarray}
		|\hat{\eta}_0(\mathbf{x}) - a(\mathbf{x})|\leq \frac{2TC_Kq_h(\mathbf{x})}{c_Kn_0}.
	\end{eqnarray}
\end{lem}
\begin{proof}
	The proof just follows the proof of Lemma \ref{lem:e2} in section \ref{sec:e2}. The detailed steps are omitted here.
\end{proof}
Follow \eqref{eq:I1}, we have
\begin{eqnarray}
	I_1 &=&\mathbb{E}\left[\underset{\mathcal{A}}{\sup}\left(\hat{\eta}_0(\mathbf{X}) - a(\mathbf{X})\right)^2\right]\nonumber\\
	&=&\mathbb{E}\left[\underset{\mathcal{A}}{\sup}\int (\hat{\eta}_0(\mathbf{x}) - a(\mathbf{x}))^2 f(\mathbf{x})d \mathbf{x}\right]\nonumber\\
	&\leq & \frac{4T^2C_K^2}{c_K^2}\underset{\mathcal{A}}{\sup}\int \frac{q_h^2(\mathbf{x})}{n_0^2(\mathbf{x})}\mathbf{1}\left(q_h(\mathbf{x})\leq \frac{c_Kn_0(\mathbf{x})}{3C_K+c_K}\right) f(\mathbf{x})d\mathbf{x}+4M^2 \left[\text{P}(n_h(\mathbf{X})<n_0(\mathbf{X}))\right.\nonumber\\
	&&\left.+\text{P}(Z>2Lh+8\sigma\ln N)+\underset{\mathcal{A}}{\sup}\text{P}\left(q_h(\mathbf{X})>\frac{c_K}{3C_K+c_K}n_0(\mathbf{X})\right)\right].
	\label{eq:I1new}
\end{eqnarray}
Now we bound each terms in \eqref{eq:I1new}.

For the first term in \eqref{eq:I1new},
\begin{eqnarray}
	\underset{\mathcal{A}}{\sup}\int \frac{q_h^2(\mathbf{x})}{n_0^2(\mathbf{x})}\mathbf{1}\left(q_h(\mathbf{x})\leq \frac{c_Kn_0(\mathbf{x})}{3C_K+c_K}\right) f(\mathbf{x}) d\mathbf{x}&\leq& \frac{c_K}{3C_K+c_K}\underset{\mathcal{A}}{\sup}\int \frac{q_h(\mathbf{x})}{n_0(\mathbf{x})} f(\mathbf{x})d\mathbf{x}\nonumber\\
	&\leq & \frac{2c_K}{(3C_K+c_K)\alpha v_d}\frac{1}{Nh^d}\underset{\mathcal{A}}{\sup}\int q_h(\mathbf{x})d\mathbf{x}\nonumber\\
	&\lesssim & \frac{q}{N}
	\label{eq:term1new}
\end{eqnarray}
For the second term, similar to \eqref{eq:smalln}, 
\begin{eqnarray}
	\text{P}(n_h(\mathbf{x}) < n_0(\mathbf{x}))\leq e^{-(1-\ln 2)n_0(\mathbf{x})}=\exp\left[-\frac{1}{2}(1-\ln 2)\alpha f(\mathbf{x})v_dh^dN\right].
\end{eqnarray}
Therefore
\begin{eqnarray}
	\text{P}(n_h(\mathbf{X}) < n_0(\mathbf{X}))&\leq& \int f(\mathbf{x})\exp\left[-\frac{1}{2}(1-\ln 2)\alpha v_d h^d Nf(\mathbf{x})\right]d\mathbf{x}\nonumber\\
	&\leq &\int_\mathcal{X} \frac{2}{(1-\ln 2) e\alpha v_dNh^d}d\mathbf{x}\nonumber\\
	&=&\frac{2V}{(1-\ln 2)e\alpha v_dh^d N}\sim \frac{1}{Nh^d}.
	\label{eq:smallnnew}
\end{eqnarray}
The bound of the third term is the same as \eqref{eq:term3}. For the fourth term, from \eqref{eq:term1new},
\begin{eqnarray}
	\int \frac{q_h(\mathbf{x})}{n_0(\mathbf{x})}d\mathbf{x}\lesssim \frac{q}{N}.
\end{eqnarray}
Using Markov inequality, 
\begin{eqnarray}
	\text{P}\left(\frac{q_h(\mathbf{X})}{n_h(\mathbf{X})}>\frac{c_K}{3C_K+c_K}\right)\lesssim \frac{q}{N}.
\end{eqnarray}
Hence
\begin{eqnarray}
	I_1\lesssim \frac{T^2q}{N}+\frac{1}{Nh^d}.
\end{eqnarray}
The bound of $I_2$ is the same as \eqref{eq:result2}. Now it remains to bound $I_3$. \eqref{eq:I3} becomes
\begin{eqnarray}
	\mathbb{E}[(b(\mathbf{x}) - \eta(\mathbf{x}))^2]\leq L^2 h^2 +\frac{\sigma^2 C_K^2}{n_0(\mathbf{x})c_K^2}+4M^2 \text{P}(n_h(\mathbf{X})<n_0(\mathbf{X})).
\end{eqnarray}
Therefore
\begin{eqnarray}
	I_3\lesssim h^2+\int_\mathcal{X} \frac{1}{n_0(\mathbf{x})} f(\mathbf{x}) d\mathbf{x} + \text{P}(n_h(\mathbf{X}) < n_0(\mathbf{X})).
\end{eqnarray}
Note that
\begin{eqnarray}
	\int_\mathcal{X} \frac{1}{n_0(\mathbf{x})} f(\mathbf{x})d\mathbf{x} = \frac{2}{\alpha v_d}\frac{1}{Nh^d}\int_\mathcal{X} 1d\mathbf{x} = \frac{2V}{\alpha v_d Nh^d}\sim \frac{1}{Nh^d}.
	\label{eq:I32}
\end{eqnarray}
From \eqref{eq:I32} and \eqref{eq:smallnnew},
\begin{eqnarray}
	I_3\lesssim h^2 +\frac{1}{Nh^d}.
\end{eqnarray}
Combine $I_1$, $I_2$ and $I_3$,
\begin{eqnarray}
	R\lesssim \frac{T^2 q}{N}+h^2 +\frac{1}{Nh^d}.
\end{eqnarray}

\section{Proof of Theorem 2: $\ell_\infty$ Convergence of Initial Estimator}\label{sec:linfty}
In the following proof, we assume that
\begin{eqnarray}
	h>\left(\frac{2(3C_K+c_K)q}{c_K\alpha f_mv_dN}\right)^\frac{1}{d}.
	\label{eq:assumeh}
\end{eqnarray}
If \eqref{eq:assumeh} does not hold, then $q/(Nh^d)\gtrsim 1$. Since $|\hat{\eta}_0(\mathbf{x})|\leq M$ always hold, we have
\begin{eqnarray}
	|\hat{\eta}_0(\mathbf{x}) - \eta(\mathbf{x})|\leq 2M \lesssim \frac{q}{Nh^d}\lesssim \frac{Tq}{Nh^d} + h + \frac{\ln N}{\sqrt{Nh^d}},
\end{eqnarray}
thus Theorem 2 is proved trivially. From now on, assume \eqref{eq:assumeh} holds.

To begin with, define event $E$, which is true if all of the three conditions hold:

(1) $\max_i W_i\leq 4\sigma \ln N$, $\min_i W_i\geq -4\sigma \ln N$;

(2) $n_0\leq n_h(\mathbf{x})\leq n_M, \forall \mathbf{x}\in \mathcal{X}$, in which $n_0$ is defined in \eqref{eq:n0}, $n_h(\mathbf{x})$ is defined in \eqref{eq:nh}, and
\begin{eqnarray}
	n_M = \frac{3}{2}Nf_Mv_dh^d;
\end{eqnarray}

(3) For all $\mathbf{x}\in \mathcal{X}$ and any $k\in \{1,\ldots, N\}$,
\begin{eqnarray}
	\left|\sum_{i\in \mathcal{N}_k(\mathbf{x})}W_i\right|\leq \sigma \max\{\sqrt{k}\ln N, \ln^2 N \},
\end{eqnarray}
in which $\mathcal{N}_k(\mathbf{x})$ is the set of all samples among $k$ nearest neighbors of $\mathbf{x}$. We remark that according to \eqref{eq:Z}, (1) implies that $Z\leq 2Lh+8\sigma \ln N$.

Denote the complement of $E$ as $E^c$. Now we bound $P(E^c)$. From \eqref{eq:wmax}, $\text{P}(\max_iW_i>4\sigma\ln N)\leq 1/N$. Similar bound holds for $\min_i W_i$. This bounds the probability of violating (1). (2) and (3) can be bounded using the following two lemmas:

\begin{lem}\label{lem:n}
	For sufficiently large $N$, with probability at least $1-1/N$, 
	\begin{eqnarray}
		\underset{\mathbf{x}\in \mathcal{X}}{\inf}n_h(\mathbf{x})&\geq& n_0,\\
		\underset{\mathbf{x}\in \mathcal{X}}{\sup}n_h(\mathbf{x}) &\leq & n_M.
	\end{eqnarray}
\end{lem}
\begin{lem}\label{lem:localw}
	Let $[N] = \{1,\ldots, N \}$. Then
	\begin{eqnarray}
		\text{P}\left(\exists \mathbf{x}\in \mathbf{X}, \exists k\in [N],\left|\sum_{i\in \mathcal{N}_k(\mathbf{x})}W_i\right|>\sigma\max\{\sqrt{k}\ln N, \ln^2 N \}\right)\leq 2dN^{2d+1} e^{-\frac{1}{2}\ln^2 N}.
	\end{eqnarray}
\end{lem}
Therefore
\begin{eqnarray}
	\text{P}(E^c) \leq \frac{3}{N} + 2dN^{2d+1} e^{-\frac{1}{2}\ln^2 N}.
\end{eqnarray}

Now we bound $\ell_\infty$ error with the condition that $E$ is true. 
\begin{eqnarray}
	|\hat{\eta}_0(\mathbf{x}) - \eta(\mathbf{x})|\leq |\hat{\eta}_0(\mathbf{x}) - a(\mathbf{x})| + |a(\mathbf{x})-b(\mathbf{x})|+|b(\mathbf{x}) - \eta(\mathbf{x})|.
	\label{eq:decomp}
\end{eqnarray}
\textbf{Bound of the first term in \eqref{eq:decomp}}. 
Under $E$, condition (a), (b) in Lemma \ref{lem:e2} are satisfied. Moreover, from \eqref{eq:assumeh} and \eqref{eq:n0}, condition (c) also hold. According to Lemma \ref{lem:e2},
\begin{eqnarray}
	|\hat{\eta}_0(\mathbf{x}) - a(\mathbf{x})|\leq \frac{2TC_Kq}{c_Kn_0}.
\end{eqnarray}
\textbf{Bound of the second term in \eqref{eq:decomp}}. Recall that $Z\leq 2Lh+8\sigma \ln N< T$, from Lemma \ref{lem:nodiff}, $a(\mathbf{x})-b(\mathbf{x}) = 0$.

\textbf{Bound of the third term in \eqref{eq:decomp}}. We use the following additional lemma:
\begin{lem}\label{lem:K}
	If $E$ is true, then
	\begin{eqnarray}
		\left|\sum_{i=1}^N K\left(\frac{\mathbf{x}-\mathbf{X}_i}{h}\right) W_i\right|\leq C_K\sqrt{n_M}\ln N.
	\end{eqnarray}
\end{lem}

From \eqref{eq:ekernel} and Lemma \ref{lem:K}, 
\begin{eqnarray}
	|b(\mathbf{x}) - \eta(\mathbf{x})|\leq Lh+\frac{C_K}{n_0c_K}\sqrt{n_M}\ln N.
\end{eqnarray}
Substitute these results into \eqref{eq:decomp}, and recall that $n_0\sim Nh^d$, $n_M\sim Nh^d$, under $E$,
\begin{eqnarray}
	|\hat{\eta}_0(\mathbf{x}) - \eta(\mathbf{x})|\lesssim \frac{Tq}{Nh^d} + h +\frac{\ln N}{\sqrt{Nh^d}},
\end{eqnarray}
Under $E^c$, $|\hat{\eta}_0(\mathbf{x}) - \eta(\mathbf{x})|$ can be bounded by $2M$, hence
\begin{eqnarray}
	\mathbb{E}\left[\|\hat{\eta}_0-\eta\|\right] &\leq& \mathbb{E}\left[\|\hat{\eta}_0-\eta\|\mathbf{1}(E)\right]+2MP(E^c)\nonumber\\
	&\lesssim & \frac{Tq}{Nh^d} + h +\frac{\ln N}{\sqrt{Nh^d}}.
\end{eqnarray}
The proof is complete.

\section{Proof of Theorem 3: Minimax Convergence Rate}\label{sec:mmx}
The proof begins with the following lemma.
\begin{lem}\label{lem:attack}
	Let $p_1$, $p_2$ be the pdf of $\mathcal{N}(\mu_1,\sigma^2 )$ and $\mathcal{N}(\mathbf{\mu}_2, \sigma^2)$ respectively, with $\mathcal{N}$ being the normal distribution. Then there exists $\alpha \in [0, |\mathbf{\mu}_1-\mathbf{\mu}_2|/(2\sigma)]$ and two other pdfs $q_1$, $q_2$, such that
	\begin{eqnarray}
		(1-\alpha)p_1 + \alpha q_1 = (1-\alpha) p_2+\alpha q_2.
		\label{eq:attack}
	\end{eqnarray}
\end{lem}
\begin{proof}
	Let
	\begin{eqnarray}
		q_1&=&\frac{1-\alpha}{\alpha}(p_2-p_1)\mathbf{1}(p_2\geq p_1),\nonumber\\
		q_2 &=& \frac{1-\alpha}{\alpha}(p_1-p_2)\mathbf{1}(p_2<p_1,)
	\end{eqnarray}
	then \eqref{eq:attack} holds. Note that $q_1$ and $q_2$ need to be normalized:
	\begin{eqnarray}
		\int q_1(u)du &=& \frac{1-\alpha}{\alpha}\int (p_2(u)-p_1(u)) \mathbf{1}(p_2(u)\geq p_1(u))du\nonumber\\
		&=&\frac{1-\alpha}{\alpha}\mathbb{TV}(p_1,p_2),
	\end{eqnarray}
	in which $\mathbb{TV}$ is the total variation distance. Hence \eqref{eq:attack} can be achieved with
	\begin{eqnarray}
		\alpha = \frac{\mathbb{TV}(p_1,p_2)}{1+\mathbb{TV}(p_1,p_2)}.
		\label{eq:alpha}
	\end{eqnarray}
	From Pinsker's inequality,
	\begin{eqnarray}
		\mathbb{TV}(p_1,p_2)\leq \sqrt{\frac{1}{2}D_{KL}(p_1||p_2)}=\frac{|\mu_1-\mu_2|}{2\sigma}.
		\label{eq:pinsker}
	\end{eqnarray}
	With \eqref{eq:alpha} and \eqref{eq:pinsker}, the proof is complete.
\end{proof}
Let 
\begin{eqnarray}
	\eta_1(\mathbf{x}) &= & 0,\\
	\eta_2(\mathbf{x}) &=& L\max\{r-\|\mathbf{x}\|, 0\},
\end{eqnarray}
in which
\begin{eqnarray}
	r = \left(\frac{\sigma (d+1)q}{f_mLv_dN}\right)^\frac{1}{d+1}.
\end{eqnarray}
Let $\eta\in \{\eta_1,\eta_2\}$. Furthermore, assume that the noise variables are Gaussian, i.e. $W_i\sim \mathcal{N}(0, \sigma^2)$. Assume $f=f_m$ for $\mathbf{x}\in \mathcal{X}$. Design the attack strategy as following: go through all samples from $i=1$ to $N$, and initialize $|\mathcal{B}_0|=\emptyset$, then

(1) If $\|\mathbf{X}_i\|>r$, do not attack;

(2) If $\|\mathbf{X}_i\|\leq r$, then find $\alpha_i$, $q_{i1}$, $q_{i2}$ such that $(1-\alpha_i)p_{i1}+\alpha q_{i1}=(1-\alpha)p_{i2}+\alpha q_{i2}$, in which $p_{i1}$, $p_{i2}$ is the pdf of $\mathcal{N}(\eta_1(\mathbf{X}_i), \sigma^2)$ and $\mathcal{N}(\eta_2(\mathbf{X}_i), \sigma^2)$, respectively. With probability $\alpha_i$, incorporate sample $i$ into $\mathcal{B}_0$;

(3) Repeat (1), (2) for $i=1,\ldots, N$; 

(4) If $|\mathcal{B}_0|\leq q$, then attack all samples in $\mathcal{B}_0$. Otherwise, pick $q$ samples randomly from $\mathcal{B}_0$ to attack. For each attacked sample $i\in \mathcal{B}_0$, let it follow distribution $q_1$ if $\eta=\eta_1$ and $q_2$ if $\eta=\eta_2$.

Use Lemma \ref{lem:attack}, and denote $s_d$ as the $d-1$ dimensional surface area of a $d$ dimensional unit ball. $s_d=dv_d$. Then
\begin{eqnarray}
	\text{P}(i\in \mathcal{B}_0)&\leq& \int f(\mathbf{x})\frac{|\eta_2(\mathbf{x})-\eta_1(\mathbf{x})|}{2\sigma}d\mathbf{x}\nonumber\\
	&=&\frac{f_mL}{2\sigma}\int_{\|\mathbf{x}\|\leq r} (r-\|\mathbf{x}\|)d\mathbf{x}\nonumber\\
	&=&\frac{f_mL}{2\sigma}\int_0^r (r-u)s_du^{d-1}du\nonumber\\
	&=& \frac{f_mLs_d}{2\sigma d(d+1)}r^{d+1}\nonumber\\
	&=& \frac{q}{2N}.
\end{eqnarray}
From Chernoff inequality, it can be easily shown that $\text{P}(|\mathcal{B}_0|>q)\leq e^{-(\ln 2-1/2)q}$. Therefore with high probability, all samples in $\mathcal{B}_0$ will be attacked, and the distribution of $Y_i$ conditional on the value of $\mathbf{X}_i$ has no difference between $\eta_1$ and $\eta_2$. This indicates that $\eta_1$ and $\eta_2$ are indistinguishable. Therefore
\begin{eqnarray}
	\underset{\hat{\eta}_0}{\inf}\underset{(f,\eta,\mathbb{P}_N)\in \mathcal{F}}{\sup}\mathbb{E}\left[\underset{\mathcal{A}}{\sup}\left(\hat{\eta}_0(\mathbf{X})-\eta(\mathbf{X})\right)^2\right]&\geq & \underset{\hat{\eta}_0}{\inf}\underset{\eta\in \{\eta_1,\eta_2\}}{\sup}\mathbb{E}\left[\underset{\mathcal{A}}{\sup}\left(\hat{\eta}_0(\mathbf{X})-\eta(\mathbf{X})\right)^2\right]\nonumber\\
	&\geq & (1-\text{P}(|\mathcal{B}_0|>q))\frac{1}{4}\int (\eta_1(\mathbf{x})-\eta_2(\mathbf{x}))^2d\mathbf{x}\nonumber\\
	&\geq & \frac{1}{4}\left(1-e^{-\left(\ln 2-\frac{1}{2}\right) q}\right) L^2\int_0^r (r-u)^2s_d u^{d-1} du\nonumber\\
	&\gtrsim & r^{d+2}\nonumber\\
	&\gtrsim & \left(\frac{q}{N}\right)^\frac{d+2}{d+1},
	\label{eq:lb1}
\end{eqnarray}
and similarly, the $\ell_\infty$ error can be lower bounded by
\begin{eqnarray}
	\underset{\hat{\eta}}{\inf}\underset{(f,\eta,\mathbb{P}_N)\in \mathcal{F}}{\sup}\mathbb{E}\left[\underset{\mathcal{A}}{\sup}\underset{\mathbf{x}}{\sup}|\hat{\eta}(\mathbf{x})-\eta(\mathbf{x})|\right]&\geq& \frac{1}{4}\left(1-e^{-\left(\ln 2-\frac{1}{2}\right) q}\right) Lr\nonumber\\
	&\gtrsim & r\nonumber\\
	&\gtrsim & \left(\frac{q}{N}\right)^\frac{1}{d+1}.
	\label{eq:lb2}
\end{eqnarray}

Moreover, even if there are no adversarial samples, from standard minimax analysis \cite{tsybakov2009}, it can be easily shown that
\begin{eqnarray}
	\underset{\hat{\eta}_0}{\inf}\underset{(f,\eta,\mathbb{P}_N)\in \mathcal{F}}{\sup}\mathbb{E}\left[\left(\hat{\eta}_0(\mathbf{X})-\eta(\mathbf{X})\right)^2\right]\gtrsim N^{-\frac{2}{d+2}},
	\label{eq:lb3}
\end{eqnarray}
and
\begin{eqnarray}
	\underset{\hat{\eta}}{\inf}\underset{(f,\eta,\mathbb{P}_N)\in \mathcal{F}}{\sup}\mathbb{E}\left[\underset{\mathcal{A}}{\sup}\underset{\mathbf{x}}{\sup}|\hat{\eta}(\mathbf{x})-\eta(\mathbf{x})|\right]	\gtrsim N^{-\frac{1}{d+2}}.
	\label{eq:lb4}
\end{eqnarray}
Combine \eqref{eq:lb1}, \eqref{eq:lb2}, \eqref{eq:lb3} and \eqref{eq:lb4}, the proof is complete.
\section{Proof of Uniqueness of Corrected Estimator}\label{sec:unique}

Suppose that there are two solutions, $g_1^*$, $g_2^*$, such that $g_1^*(\mathbf{x})\neq g_2^*(\mathbf{x})$ for some $\mathbf{x}$, and
\begin{eqnarray}
	\|\hat{\eta}_0-g_1^*\|_1=\|\hat{\eta}_0-g_2^*\|\leq \|\hat{\eta}_0-g\|,
	\label{eq:solution}
\end{eqnarray}
for any $L$-Lipschitz function $g$. Since $g_1^*$ and $g_2^*$ are Lipschitz continuous, there must be a compact region around $\mathbf{x}$ such that $g_1^*\neq g_2^*$ everywhere in this region.

We first show that for all $\mathbf{x}$ with $g_1^*(\mathbf{x})\neq g_2^*(\mathbf{x})$,
\begin{eqnarray}
	(\hat{\eta}_0(\mathbf{x})-g_1^*(\mathbf{x}))(\hat{\eta}_0(\mathbf{x})-g_2^*(\mathbf{x}))\geq 0.	
	\label{eq:samesign}
\end{eqnarray}
Let $g_a= (g_1^*+g_2^*)/2$. If $\hat{\eta}_0(\mathbf{x})-g_1^*(\mathbf{x})$ and $\hat{\eta}_0(\mathbf{x})-g_2^*(\mathbf{x})$ have opposite sign, then
\begin{eqnarray}
	|g_a(\mathbf{x})-\hat{\eta}_0(\mathbf{x})|< \frac{1}{2}(|g_1^*(\mathbf{x})-\hat{\eta}_0(\mathbf{x})|+|g_2^*(\mathbf{x})-\hat{\eta}_0(\mathbf{x})|),
\end{eqnarray}
thus $\|g_a-\hat{\eta}_0\|_1< (\|g_1^*-\hat{\eta}_0\|_1+\|g_1^*-\hat{\eta}_0\|_1)$, contradicts \eqref{eq:solution}. Therefore \eqref{eq:samesign} holds. This indicates that the support $\mathcal{X}$ can be divided into $S_1$ and $S_2$, such that in $\max\{g_1^*, g_2^*\}\leq\hat{\eta}_0$ within $S_1$, $\min\{g_1^*, g_2^*\}\geq \hat{\eta}_0$ within $S_2$. 

Then let
\begin{eqnarray}
	g_b(\mathbf{x})=\left\{
	\begin{array}{ccc}
		\max\{g_1^*(\mathbf{x}), g_2^*(\mathbf{x})\} &\text{if} & x\in S_1\\
		\min\{g_1^*(\mathbf{x}), g_2^*(\mathbf{x})\} &\text{if} &x\in S_2,
	\end{array}
	\right.
\end{eqnarray}
then it can be easily shown that $g_b$ is Lipschitz and $\|\hat{\eta}_0-g\|<\max\{\|\hat{\eta}_0-g_1^*\|, \|\hat{\eta}_0-g_2^*\|\}$, contradicts \eqref{eq:solution}. The proof is complete.

\section{Proof of Theorem 4: Convergence Rate of Corrected Estimator}\label{sec:converge}
Denote $F[\eta]$ as the solution of the optimization problem
\begin{mini}
	{g}{\|\eta-g\|_1}{}{}
	\addConstraint{\|\nabla g\|_\infty \leq L.}{}{}
	\label{eq:opt}
\end{mini}
Then the corrected estimate is $\hat{\eta}_c=F[\hat{\eta}_0]$, with $\hat{\eta}_0$ being the initial estimate. The following lemma holds:
\begin{lem}\label{lem:monotone}
	For some $\eta_1$, $\eta_2$, If $\eta_1(\mathbf{x})\leq \eta_2(\mathbf{x})$, $\forall \mathbf{x}\in \mathcal{X}$, then $F[\eta_1](\mathbf{x})\leq F[\eta_2](\mathbf{x}), \forall \mathbf{x}\in \mathcal{X}$.
\end{lem}

Denote $E$ as the event that $\underset{x}{\inf}n_h(\mathbf{x})\geq n_0$, $\underset{x}{\sup} n_h(\mathbf{x})\leq 3NfMv_dh^d/2$, and $Z\leq 2Lh+8\sigma\ln N$. From Lemma \ref{lem:Z}, $\text{P}(Z>2Lh+8\sigma \ln N)\leq 2/N$. Combine with Lemma \ref{lem:n}, the probability of violating $E$ can be bounded by
\begin{eqnarray}
	\text{P}(E^c)\leq \frac{3}{N},
\end{eqnarray}
in which $E^c$ is the complement of $E$.

In the following analysis, we bound the estimation error under the condition that $E$ is true. We show the following additional lemma:
\begin{lem}\label{lem:largek}
	\begin{eqnarray}
		\text{P}\left(\left|\sum_{i\in I_h(\mathbf{x})}K\left(\frac{\mathbf{x}-\mathbf{X}_i}{h}\right) W_i\right|> 3C_k\sigma\sqrt{Nf_Mv_dh^d\ln N}|E\right)\leq \frac{2}{N^2}.
	\end{eqnarray}
\end{lem}

With these lemmas, we analyze the corrected estimator $\hat{\eta}_c$ under $E$. To begin with, $\hat{\eta}_0$ satisfies
\begin{eqnarray}
	|\hat{\eta}_0(\mathbf{x})-\eta(\mathbf{x})|\leq |\hat{\eta}_0(\mathbf{x}) - a(\mathbf{x})|+|a(\mathbf{x}) - b(\mathbf{x})|+|b(\mathbf{x}) - \eta(\mathbf{x})|.
\end{eqnarray}
From Lemma \ref{lem:e2}, Under $E$,
\begin{eqnarray}
	|\hat{\eta}_0(\mathbf{x}) - a(\mathbf{x})|\leq \frac{2TC_Kq_h(\mathbf{x})}{c_Kn_0}+2M\mathbf{1}\left(q_h(\mathbf{x}) > \frac{c_kn_0}{3C_K+c_K}\right);
\end{eqnarray}
From Lemma \ref{lem:nodiff}, under $E$, and Assumption 3, $Z\leq T/2$, thus $a(\mathbf{x})-b(\mathbf{x}) = 0$;

From \eqref{eq:ekernel},
\begin{eqnarray}
	|b(\mathbf{x}) - \eta(\mathbf{x})|&\leq& Lh+\frac{3C_K\sigma}{c_Kn_0}\sqrt{Nf_Mv_dh^d\ln N} \nonumber\\
	&&+ 2M \mathbf{1}\left(\left|\sum_{i\in I_h(\mathbf{x})}K\left(\frac{\mathbf{x}-\mathbf{X}_i}{h}\right) W_i\right|> 3C_K\sigma\sqrt{Nf_Mv_dh^d\ln N}\right).
\end{eqnarray}
Define
\begin{eqnarray}
	S:=\left\{\mathbf{x}|\left|\sum_{i\in I_h(\mathbf{x})}K\left(\frac{\mathbf{x}-\mathbf{X}_i}{h}\right) W_i\right|> 3C_K\sigma\sqrt{Nf_Mv_dh^d\ln N} \right\},
\end{eqnarray}
\begin{eqnarray}
	\delta(\mathbf{x}):= \frac{2TC_Kq_h(\mathbf{x})}{c_Kn_0}+Lh + 3C_K\sigma \sqrt{Nf_Mv_dh^d\ln N}+2M\mathbf{1}\left(q_h(\mathbf{x})>\frac{c_kn_0}{3C_K+c_K}\right)+2M\mathbf{1}(\mathbf{x}\in S).
	\label{eq:deltax}
\end{eqnarray}
Therefore, under $E$, $|\hat{\eta}_0(\mathbf{x}) - \eta(\mathbf{x})|\leq \delta(\mathbf{x})$. Furthermore, define
\begin{eqnarray}
	\eta_1(\mathbf{x})&:=& \eta(\mathbf{x}) - \delta(\mathbf{x}),\\
	\eta_2(\mathbf{x})&:= & \eta(\mathbf{x})+\delta(\mathbf{x}),
\end{eqnarray}
then $\eta_1\leq \hat{\eta}_0\leq \eta_2$ under $E$. From Lemma \ref{lem:monotone}, $F[\eta_1]\leq F[\eta_0]\leq F[\eta_2]$. The error of $F[\eta_0]$ can be bounded by the error of $F[\eta_1]$ and $F[\eta_2]$. Define
\begin{eqnarray}
	\Delta := Lh+3C_K\sigma \sqrt{NfMv_dh^d\ln N}.
\end{eqnarray}
The next lemma bounds $\|F[\eta_1]-\eta\|_2^2$ and $\|F[\eta_2] - \eta_2\|_2^2$:
\begin{lem}\label{lem:l2}
	Under $E$,
	\begin{eqnarray}
		\max\left\{\|F[\eta_1]-\eta\|_2^2,\|F[\eta_2]-\eta\|_2^2 \right\}\lesssim \left(\frac{Tq}{N}\right)^\frac{d+2}{d+1} + h^2+\frac{\ln N}{Nh^d}.
	\end{eqnarray}
\end{lem}
Therefore
\begin{eqnarray}
	\|F[\hat{\eta}_0]-\eta\|_2^2\lesssim \left(\frac{Tq}{N}\right)^\frac{d+2}{d+1} + h^2+\frac{\ln N}{Nh^d}.
\end{eqnarray}
The overall risk can be bounded by
\begin{eqnarray}
	\mathbb{E}\left[\underset{\mathcal{A}}{\sup}\left(\hat{\eta}_c(\mathbf{X})-\eta(\mathbf{X})\right)^2\right]&\leq& f_M\mathbb{E}\left[\underset{\mathcal{A}}{\sup}\|\hat{\eta}_c-\eta\|_2^2\mathbf{1}(E)\right]+4M^2\text{P}(E^c)\nonumber\\
	&\lesssim & \left(\frac{Tq}{N}\right)^\frac{d+2}{d+1} + h^2+\frac{\ln N}{Nh^d}.
\end{eqnarray}
The proof of $\ell_2$ bound is complete.

For the $\ell_\infty$ bound, note that
\begin{eqnarray}
	\hat{\eta}_0\leq \eta + \|\hat{\eta}_0-\eta\|_\infty,
\end{eqnarray}
thus from Lemma \ref{lem:monotone},
\begin{eqnarray}
	\hat{\eta}_c = F[\hat{\eta}_0]\leq F[\eta + \|\hat{\eta}_0-\eta\|_\infty] = \eta + \|\hat{\eta}_0-\eta\|_\infty.
\end{eqnarray}
Similarly,
\begin{eqnarray}
	\hat{\eta}_c\geq \eta - \|\hat{\eta}_0-\eta\|_\infty.
\end{eqnarray}
Hence
\begin{eqnarray}
	\|\hat{\eta}_c-\eta\|_\infty\leq \|\hat{\eta}_0-\eta\|_\infty.
\end{eqnarray}
This indicates that the $\ell_\infty$ error of the corrected estimator does not exceed the initial estimator. Therefore, Theorem 2 can be directly used here.

\section{Proof of Lemmas}\label{sec:Lemmas}
\subsection{Proof of Lemma \ref{lem:Z}}
\textbf{Proof of \eqref{eq:Zbound}.} Recall Assumption 1(d). Let $W_{max}=\max_i W_i$, $W_{min} = \min_i W_i$ for $i=1,\ldots, N$. For $|\lambda|\leq 1/\sigma$,
\begin{eqnarray}
	\mathbb{E}[e^{\lambda W_{max}}] \leq \sum_{i=1}^N \mathbb{E}[e^{\lambda W_i}]\leq Ne^{\frac{1}{2}\lambda^2 \sigma^2}.
\end{eqnarray}
Then
\begin{eqnarray}
	\text{P}(W_{max} > t) &\leq& N \underset{|\lambda|\leq 1/\sigma}{\inf}e^{-\lambda t}e^{\frac{1}{2}\lambda^2 \sigma^2}\nonumber\\
	&\leq & \exp\left[-\min\left\{\frac{t^2}{2\sigma^2}, \frac{t}{2\sigma}\right\} + \ln N\right].
	\label{eq:wmax}
\end{eqnarray}
Similar bound holds for $W_{min}$. Then
\begin{eqnarray}
	\text{P}(W_{max} - W_{min} > t)&\leq &\text{P} \left(W_{max} > \frac{t}{2}\right) + \text{P}\left(W_{min} < -\frac{t}{2}\right)\nonumber\\
	&\leq & 2\exp\left[-\left\{\frac{t^2}{8\sigma^2}, \frac{t}{4\sigma}\right\} +\ln N\right].
\end{eqnarray}
From \eqref{eq:Z}, $Z=W_{max} - W_{min} +2Lh^\gamma$, hence \eqref{eq:Zbound} holds.

\textbf{Proof of \eqref{eq:zsq}.}
\begin{eqnarray}
	\mathbb{E}[Z^2\mathbf{1}(Z>t)] &\leq & \int_0^{t^2} \text{P}(Z>t) du + \int_{t^2}^\infty \text{P}(Z>\sqrt{u}) du\nonumber\\
	&=& t^2\text{P}(Z>t) + 2\int_{t^2}^\infty \exp\left[-\frac{\sqrt{u}-2Lh^\gamma}{4\sigma}+\ln N\right] du\nonumber\\
	&\leq & 2N(s^2+32\sigma^2+8s\sigma)e^\frac{t-2Lh^\gamma}{4\sigma}.
\end{eqnarray}
\subsection{Proof of Lemma \ref{lem:e1}}
We first discuss the case when
\begin{eqnarray}
	\sum_{i\in I_h(\mathbf{x})}K\left(\frac{\mathbf{x}-\mathbf{X}_i}{h}\right) \phi'(\eta(\mathbf{X}_i)+W_i-a(\mathbf{x})) = 0.
	\label{eq:cond}
\end{eqnarray}
According to \eqref{eq:adef}, this happens if the minimum value in \eqref{eq:adef} is reached within $[-M, M]$.
\begin{eqnarray}
	\left|\sum_{i\in I_h(\mathbf{x})}K\left(\frac{\mathbf{x}-\mathbf{X}_i}{h}\right)\phi'(Y_i-a(\mathbf{x}))\right|&=& \left|\sum_{i\in I_h(\mathbf{x})}K\left(\frac{\mathbf{x}-\mathbf{X}_i}{h}\right)\left[\phi'(Y_i-a(\mathbf{x}))-\phi'(\eta(\mathbf{X}_i) + W_i - a(\mathbf{x}))\right]\right|\nonumber\\
	&\leq & \sum_{i\in I_h(\mathbf{x})}K\left(\frac{\mathbf{x}-\mathbf{X}_i}{h}\right)|\phi'(Y_i-a(\mathbf{x})) - \phi'(\eta(\mathbf{X}_i) + W_i - a(\mathbf{x}))|\nonumber\\
	&\overset{(a)}{\leq} & \underset{i\in I_h(\mathbf{x})\cap \mathcal{B}}K\left(\frac{\mathbf{x}-\mathbf{X}_i}{h}\right)|\phi'(Y_i-a(\mathbf{x})) - \phi'(\eta(\mathbf{X}_i) + W_i - a(\mathbf{x}))|\nonumber\\
	&\overset{(b)}{\leq}& 2\sum_{i\in I_h(\mathbf{x})\cap \mathcal{B}}K\left(\frac{\mathbf{x}-\mathbf{X}_i}{h}\right) (T+Z)\nonumber\\
	&\overset{(c)}{\leq} & 2(T+Z)C_Kq_h(\mathbf{x}).
	\label{eq:zeropoint1}	
\end{eqnarray}
For (a), recall that $Y_i\neq \eta(\mathbf{X}_i) + W_i$ only for attacked sample. For (b), recall \eqref{eq:adef}, we have
\begin{eqnarray}
	\underset{j\in I_h(\mathbf{x})}{\min} \eta(\mathbf{X}_j) + W_j\leq a(\mathbf{x})\leq \underset{j\in I_h(\mathbf{x})}{\max} \eta(\mathbf{X}) + W_j,
	\label{eq:abound}
\end{eqnarray}
therefore $|\eta(\mathbf{x}) + W_i-a|\leq Z$, $\forall i\in I_h(\mathbf{x})$. Recall that
\begin{eqnarray}
	\phi'(u)= \left\{
	\begin{array}{ccc}
		2u &\text{if} & |u|\leq T\\
		2T & \text{if} & u>T\\
		-2T &\text{if} & u<-T,
	\end{array}
	\right.
	\label{eq:derivative}
\end{eqnarray} 
thus $|\phi'(\eta(\mathbf{X}_i)+W_i-a(\mathbf{x}))|\leq 2Z$. (c) just uses Assumption 2.

Moreover, if $\delta \leq T-Z$,
\begin{eqnarray}
	&&\sum_{i\in I_h(\mathbf{x})}K\left(\frac{\mathbf{x}-\mathbf{X}_i}{h}\right)[\phi'(Y_i - a(\mathbf{x})) - \phi'(Y_i-(a(\mathbf{x})+\delta))]\nonumber\\
	&\geq & 2\delta \sum_{i\in I_h(\mathbf{x})}K\left(\frac{\mathbf{x}-\mathbf{X}_i}{h}\right) \mathbf{1}(|Y_i-a(\mathbf{x})|<T, |Y_i-(a(\mathbf{x})+\delta)|<T)\nonumber\\
	&\geq & 2\delta c_K |I_h(\mathbf{x})\setminus B|\nonumber\\
	&\geq & 2\delta c_K (n_h(\mathbf{x})-q_h(\mathbf{x})).
	\label{eq:zeropoint2}
\end{eqnarray}
Similarly,
\begin{eqnarray}
	\sum_{i\in I_h(\mathbf{x})}K\left(\frac{\mathbf{x}-\mathbf{X}_i}{h}\right)[\phi'(Y_i - (a(\mathbf{x})-\delta)) - \phi'(Y_i-a(\mathbf{x}))]\geq 2\delta c_K(n_h(\mathbf{x})-q_h(\mathbf{x})).
	\label{eq:zeropoint3}
\end{eqnarray}
Let $\delta = (T+Z) r(\mathbf{x})$, with condition $r(\mathbf{x})\leq (T-Z)/(T+Z)$, $\delta\leq T-Z$, thus \eqref{eq:zeropoint2} and \eqref{eq:zeropoint3} hold. From \eqref{eq:zeropoint1}, \eqref{eq:zeropoint2} and \eqref{eq:zeropoint3},
\begin{eqnarray}
	\sum_{i\in I_h(\mathbf{x})}K\left(\frac{\mathbf{x}-\mathbf{X}_i}{h}\right) \phi'(Y_i-(a(\mathbf{x})+\delta))\leq 0\leq 	\sum_{i\in I_h(\mathbf{x})}K\left(\frac{\mathbf{x}-\mathbf{X}_i}{h}\right) \phi'(Y_i-(a(\mathbf{x})-\delta)).
\end{eqnarray}
Therefore $\hat{\eta}_0(\mathbf{x})\in [a-\delta, a+\delta])$.

Now it remains to discuss the case when \eqref{eq:cond} is violated, which indicates that the minimum in \eqref{eq:adef} is not reached in $[-M, M]$. Then $a(\mathbf{x})=M$ or $a(\mathbf{x}) = -M$. If $a(\mathbf{x}) = M$,
\begin{eqnarray}
	\sum_{i\in I_h(\mathbf{x})}K\left(\frac{\mathbf{x}-\mathbf{X}_i}{h}\right) \phi'(\eta(\mathbf{X}_i)+W_i-a(\mathbf{x})) > 0,
\end{eqnarray}
then go through \eqref{eq:zeropoint1}, 
\begin{eqnarray}
	\left|\sum_{i\in I_h(\mathbf{x})}K\left(\frac{\mathbf{x}-\mathbf{X}_i}{h}\right)\phi'(Y_i-a(\mathbf{x}))\right|> -2(T+Z)C_Kq_h(\mathbf{x}).
\end{eqnarray}
With \eqref{eq:zeropoint3} and $\delta=(T+Z)r(\mathbf{x})$,
\begin{eqnarray}
	\left|\sum_{i\in I_h(\mathbf{x})}K\left(\frac{\mathbf{x}-\mathbf{X}_i}{h}\right)\phi'(Y_i-(a(\mathbf{x})-\delta))\right|> 0.
\end{eqnarray}
Therefore $\hat{\eta}_0(\mathbf{x})\geq a(\mathbf{x})-\delta$, and from \eqref{eq:eta_app}, $\hat{\eta}_0(\mathbf{x})\leq M$. Therefore $|\hat{\eta}_0(\mathbf{x})-a(\mathbf{x})|\leq \delta$ still holds. Similar argument holds for $a(\mathbf{x}) = -M$. The proof is complete.

\subsection{Proof of Lemma \ref{lem:e2}}\label{sec:e2}
From Assumption 3, $T\geq 4Lh + 16\sigma \ln N$, thus by condition (b) in Lemma \ref{lem:e2}, $T\geq 2Z$, $(T-Z)/(T+Z)\geq 1/3$, and
\begin{eqnarray}
	r(\mathbf{x}) = \frac{C_Kq_h(\mathbf{x})}{c_K(n_h(\mathbf{x}) - q_h(\mathbf{x}))}\leq \frac{C_K\frac{c_K}{3C_K+c_K}n_0}{c_K\left(1-\frac{c_K}{3C_K+c_K}\right)n_0} = \frac{1}{3}\leq \frac{T-Z}{T+Z}.
\end{eqnarray}
Therefore from Lemma \ref{lem:e1},
\begin{eqnarray}
	|\hat{\eta}_0(\mathbf{x})-a(\mathbf{x})|&\leq & (T+Z)r(\mathbf{x})\nonumber\\
	&\leq & T\left(1+\frac{1-r(\mathbf{x})}{1+r(\mathbf{x})}\right) r(\mathbf{x}) \nonumber\\
	&=&\frac{2T}{\frac{1}{r(\mathbf{x})} + 1}\nonumber\\
	&\leq & \frac{2TC_Kq_h(\mathbf{x})}{c_Kn_h(\mathbf{x})}\nonumber\\
	&\leq & \frac{2TC_Kq_h(\mathbf{x})}{c_Kn_0}.
\end{eqnarray}

\subsection{Proof of Lemma \ref{lem:n}}
Define 
\begin{eqnarray}
	p_h(\mathbf{x}) = \int_{\|\mathbf{u}-\mathbf{x}\|\leq h}f(\mathbf{u})d\mathbf{u}
\end{eqnarray}
as the probability mass of ball centering at $\mathbf{x}$ with radius $h$.
\begin{lem}
	(\cite{jiang2017uniform}, Lemma 3) with probability at least $1-1/N$, for all $\mathbf{x}\in \mathcal{X}$,
	\begin{eqnarray}
		\left|p_h(\mathbf{x}) - \frac{n_h(\mathbf{x})}{N}\right|\leq \beta_N\sqrt{p_h(\mathbf{x})} + \beta_N^2,
	\end{eqnarray}
	in which $\beta_N=\sqrt{4(d+3)\ln(2N)/N}$. 
\end{lem}
Assumption 3 requires that $h\geq (\ln^2 N/N)^{1/d}$, hence for sufficiently large $N$,
\begin{eqnarray}
	\frac{\beta_N}{\sqrt{p_h(\mathbf{x})}}\leq \frac{\beta_N}{\sqrt{f_m\alpha v_dh^d}}=\sqrt{\frac{4(d+3)\ln(2N)}{f_m\alpha v_dh^dN}}\leq \frac{1}{3},
\end{eqnarray}
and $\beta_N^2/p_h(\mathbf{x})\leq 1/9$. Therefore
\begin{eqnarray}
	\left|p_h(\mathbf{x}) - \frac{n_h(\mathbf{x})}{N}\right|\leq p_h(\mathbf{x})\left(\frac{\beta_N}{\sqrt{p_h(\mathbf{x})}}+\frac{\beta_N^2}{p_h(\mathbf{x})}\right)\leq \frac{4}{9}p_h(\mathbf{x}),
\end{eqnarray}
which yields
\begin{eqnarray}
	n_h(\mathbf{x})\geq \frac{5}{9}Np_h(\mathbf{x})> \frac{1}{2}Nf_m\alpha v_dh^d=\frac{1}{2}n_0.
	\label{eq:nh1}
\end{eqnarray}
For the upper bound, not that $p_h(\mathbf{x})\leq f_Mv_dh^d$, 
\begin{eqnarray}
	n_h(\mathbf{x})\leq \frac{13}{9}Np_h(\mathbf{x})\leq \frac{3}{2}Nf_Mv_dh^d.
	\label{eq:nh2}
\end{eqnarray}
With probability at least $1-1/N$, \eqref{eq:nh1} and \eqref{eq:nh2} hold uniformly for all $\mathbf{x}\in \mathcal{X}$.

\subsection{Proof of Lemma \ref{lem:localw}}\label{sec:localw}
For $i,j\in [N]$, let $A_{ij}$ be a $d-1$ dimensional hyperplane that perpendicularly bisects $\mathbf{X}_i$ and $\mathbf{X}_j$. The number of planes is $N_p=N(N-1)/2$. The number of regions divided by these planes can be bounded by
\begin{eqnarray}
	N_r = \sum_{j=0}^d \binom{N_p}{k}\leq dN_p^d\leq dN^{2d}.
\end{eqnarray}
For all $\mathbf{x}$ within a specific region, its nearest neighbors should be the same. Hence
\begin{eqnarray}
	|\{\mathcal{N}_k(\mathbf{x})|\mathbf{x}\in \mathcal{X}, k\in [N] \}|\leq dN^{2d+1}.
	\label{eq:nsets}
\end{eqnarray}
Similar formulation was used in \cite{jiang2019non}, proof of Lemma 3.

For each $\mathcal{N}_k(\mathbf{x})$, from Assumption 1(d), conditional on positions of $\mathbf{X}_1,\ldots, \mathbf{X}_N$
\begin{eqnarray}
	\mathbb{E}\left[\exp\left(\lambda \sum_{i\in \mathcal{N}_k(\mathbf{x})} W_i\right)|\mathbf{X}^N\right]\leq e^{\frac{k}{2}\lambda^2 \sigma^2}, \forall \lambda \leq \frac{1}{\sigma},
\end{eqnarray}
in which we use $\mathbf{X}^N$ to substitute $\mathbf{X}_1,\ldots, \mathbf{X}_N$ for brevity. Then the following Chernoff bound holds:
\begin{eqnarray}
	\text{P}\left(\sum_{i\in \mathcal{N}_k(\mathbf{x})} W_i > t|\mathbf{X}^N\right)&\leq & \underset{0\leq \lambda\leq 1/\sigma}{\inf} \exp\left(-\lambda t+\frac{k}{2}\lambda^2 \sigma^2\right)\nonumber\\
	&\leq& \left\{
	\begin{array}{ccc}
		e^{-\frac{t^2}{2k\sigma^2}} &\text{if} & t\leq k\sigma\\
		e^{-\frac{t}{2\sigma}} &\text{if} & t>k\sigma.
	\end{array}
	\right.
\end{eqnarray}
If $t\geq \ln^2 N$, let $t=\sqrt{k}\sigma \ln N$. Otherwise, let $t=\sigma\ln^2 N$. Then
\begin{eqnarray}
	\text{P}\left(\sum_{i\in \mathcal{N}_k(\mathbf{x})} W_i > t|\mathbf{X}^N\right) \leq e^{-\frac{1}{2}}\ln N.	
\end{eqnarray}
Combine with \eqref{eq:nsets}, the proof is complete.

\subsection{Proof of Lemma \ref{lem:K}}\label{sec:K}
Recall that the statement of Theorem 2 requires that $K(\mathbf{u})$ monotonic decrease with $\|\mathbf{u}\|$. Define
\begin{eqnarray}
	r_j:=\sup \left\{\|u\| | K(\mathbf{u}) \geq c_K+j\Delta \right\}.
\end{eqnarray}
Cut $K$ into $m$ slices above $c_K$, whose heights are $\Delta=(C_K-c_K)/m$. Define the truncated kernel as
\begin{eqnarray}
	K_T(\mathbf{u}) := \sum_{j=1}^m \Delta \mathbf{1}(\|\mathbf{u}\|\leq r_j)+c_K\mathbf{1}(\|u\|\leq 1),
\end{eqnarray}
then $0\leq K(\mathbf{u})-K_T(\mathbf{u})\leq \Delta$, under $E$, 
\begin{eqnarray}
	\left| \sum_{i=1}^N \left(K\left(\frac{\mathbf{x}-\mathbf{X}_i}{h}\right) - K_T\left(\frac{\mathbf{x}-\mathbf{X}_i}{h}\right)\right) W_i\right|\leq \Delta\max_i|W_i|\leq 4\sigma\Delta \ln N,
	\label{eq:trunc}
\end{eqnarray}
and
\begin{eqnarray}
	\left|\sum_{i=1}^N K_T\left(\frac{\mathbf{x}-\mathbf{X}_i}{h}\right) W_i\right|&\leq & \left|\sum_{i=1}^N c_K\mathbf{1}(\|\mathbf{x}-\mathbf{X}_i\| < h)W_i + \sum_{j=1}^m \sum_{i=1}^N \Delta \mathbf{1}(\|\mathbf{X}_i-\mathbf{x}\|\leq hr_j)W_i\right|\nonumber\\
	&\leq & c_K\left|\sum_{i\in \mathcal{N}_{n_h(\mathbf{x})}(\mathbf{x})} W_i\right|+\delta\sum_{j=1}^m \left|\sum_{i\in \mathcal{N}_{n_{hr_j}(\mathbf{x})}(\mathbf{x})}W_i\right|\nonumber\\
	&\leq & (c_K+m\Delta)\sqrt{n_M}\ln N\nonumber\\
	&=&C_K\sqrt{n_M}\ln N.
	\label{eq:kt}
\end{eqnarray}
in which the last step uses Lemma \ref{lem:localw}. Since $m$ can be arbitrarily large and $\Delta$ can be arbitrarily small, from \eqref{eq:trunc} and \eqref{eq:kt},
\begin{eqnarray}
	\left|\sum_{i=1}^N K\left(\frac{\mathbf{x}-\mathbf{X}_i}{h}\right) W_i\right|\leq C_K\sqrt{n_M}\ln N.
\end{eqnarray}
The proof is complete.
\subsection{Proof of Lemma \ref{lem:monotone}}
Denote $g_1=F[\eta_1]$, $g_2=F[\eta_2]$, $g=F[\eta]$. If $g_1\leq g_2$ is not satisfied somewhere, then define
\begin{eqnarray}
	S=\{\mathbf{x}|g_1(\mathbf{x})>g_2(\mathbf{x}) \}.
\end{eqnarray}
Since $g_1=F(\eta_1)$, due to the uniqueness of optimization solution (Proposition \ref{prop:unique}), $\|\eta_1-g_1\|<\|\eta_1-g\|_1$ for all $L$-Lipschitz function $g$. Hence
\begin{eqnarray}
	\|\eta_1-g_1\|_1<\|\eta_1-\min\{g_1,g_2 \}\|_1,
\end{eqnarray}
thus
\begin{eqnarray}
	\int_S |\eta_1(\mathbf{x})-g_1(\mathbf{x})|d\mathbf{x}<\int_S |\eta_1(\mathbf{x})-g_2(\mathbf{x})|d\mathbf{x}.
\end{eqnarray}
In $S$, $g_2(\mathbf{x})<g_1(\mathbf{x})$, since $\eta_2(\mathbf{x})>\eta_1(\mathbf{x})$,
\begin{eqnarray}
	|\eta_2(\mathbf{x}) - g_2(\mathbf{x})|-|\eta_2(\mathbf{x}) - g_1(\mathbf{x})|\geq |\eta_1(\mathbf{x})-g_2(\mathbf{x})|-|\eta_1(\mathbf{x}) - g_1(\mathbf{x})|.
\end{eqnarray}
Therefore
\begin{eqnarray}
	\int_S |\eta_2(\mathbf{x}) - g_1(\mathbf{x})|d\mathbf{x}<\int_S |\eta_2(\mathbf{x}) - g_2(\mathbf{x})|d\mathbf{x},
\end{eqnarray}
which yields
\begin{eqnarray}
	\|\eta_2-\min\{g_1,g_2\}\|_1<\|\eta_2-g_2\|,
\end{eqnarray}
contradict with that $g_2=F[\eta_2]$ is the solution of the optimization problem \eqref{eq:opt}. Therefore $g_1\leq g_2$ everywhere.

\subsection{Proof of Lemma \ref{lem:largek}}
According to Assumption 1(d), which requires that $W_i$ is sub-exponential with parameter $\sigma$, we have
\begin{eqnarray}
	\mathbb{E}\left[\exp\left[\left.\lambda \sum_{i\in I_h(\mathbf{x})}K\left(\frac{\mathbf{x}-\mathbf{X}_i}{h}\right) W_i\right]\right|\mathbf{X}_1,\ldots, \mathbf{X}_N\right]\leq \exp\left[\frac{1}{2}\lambda^2C_K^2\sigma^2n_h(\mathbf{x})\right], \forall |\lambda|\leq \frac{1}{C_K\sigma},
\end{eqnarray}
and
\begin{eqnarray}
	\text{P}\left(\sum_{i\in I_h(\mathbf{x})}K\left(\frac{\mathbf{x}-\mathbf{X}_i}{h}\right) W_i> t|\mathbf{X}_1,\ldots, \mathbf{X}_N\right)\leq \underset{|\lambda|\leq 1/(C_K\sigma)}{\inf}e^{-\lambda t + \lambda ^2C_K^2\sigma^2n_h(\mathbf{x})}/2.
\end{eqnarray}
Therefore for sufficiently large $N$, let $t=2C_K\sigma \sqrt{n_h(\mathbf{x})\ln N}$,
\begin{eqnarray}
	\text{P}\left(\left.\sum_{i\in I_h(\mathbf{x})}K\left(\frac{\mathbf{x}-\mathbf{X}_i}{h}\right) W_i> 2C_K\sigma\sqrt{n_h(\mathbf{x})\ln N}\right|\mathbf{X}_1,\ldots, \mathbf{X}_N\right)\leq e^{-2\ln N} = \frac{1}{N^2}.
\end{eqnarray}
The opposite side can be proved similarly. Recall that under $E$, $n_h(\mathbf{x})\leq 3Nf_Mv_dh^d/2$. The proof is complete. 
\subsection{Proof of Lemma \ref{lem:l2}}
$\eta$ is Lipschitz, thus $\eta-\Delta$ is Lipschitz. Since $F[\eta_1]$ is the solution of optimization problem \eqref{eq:opt} with $\eta_1$, we have
\begin{eqnarray}
	\|\eta_1-F[\eta_1]\|_1\leq \|\eta_1-(\eta-\Delta)\|_1.
\end{eqnarray} 
Hence
\begin{eqnarray}
	\|F[\eta_1]-(\eta-\Delta)\|_1\leq \|F[\eta_1]-\eta_1\|_1 + \|\eta_1-(\eta-\Delta)\|_1\leq 2\|\eta_1-(\eta-\Delta)\|_1.
\end{eqnarray}
It remains to bound $\|\eta_1-(\eta-\Delta)\|_1$:
\begin{eqnarray}
	\|\eta_1-(\eta-\Delta)\|_1 &\leq & \|Lh+3C_k\sigma \sqrt{Nf_Mv_dh^d\ln N}-\Delta\|_1\nonumber\\
	&\leq & \frac{2TC_K}{c_kn_0}\int q_h(\mathbf{x})d\mathbf{x}+2M\int \mathbf{1}\left(q_h(\mathbf{x})>\frac{c_Kn_0}{3C_K+c_K}\right)d\mathbf{x}+2M\int_Sd\mathbf{x}\nonumber\\
	&\leq & \frac{2TC_Kqv_dh^d}{c_Kn_0}+\frac{2M(3C_K+c_K)qv_dh^d}{c_Kn_0}+2M\int_S d\mathbf{x}.
\end{eqnarray}
$\|\eta_2-F[\eta_2]\|_1$ can be bounded in the same way. Therefore
\begin{eqnarray}
	\max\left\{\|F[\eta_1]-(\eta-\Delta)\|_1,\|F[\eta_2]-(\eta+\Delta)\|_1 \right\}\leq (C_1T+C_2)\frac{qh^d}{n_0}+2M\int_S d\mathbf{x},
\end{eqnarray}
in which $C_1=4C_Kv_d/c_K$, $C_2=4M(3C_K+c_K)v_d/c_K$. We then show the following lemma.
\begin{lem}
	If a function $g$ is $L_d$-Lipschitz with bounded $\|g\|_1$, then 
	\begin{eqnarray}
		\|g\|_2^2\lesssim \left(\frac{(d+1)d^\frac{d}{2}L^d}{v_d}\right)^\frac{1}{d+1}\|g\|_1^\frac{d+2}{d+1}.
	\end{eqnarray}
\end{lem}
\begin{proof}
	$g$ is continuous with bounded $\|g\|_1$, thus it reaches its maximum at some $\mathbf{x}_0$. Then
	\begin{eqnarray}
		|g(\mathbf{x})|\geq \|g\|_\infty-L_d\|\mathbf{x}-\mathbf{x}_0\|, \forall \|\mathbf{x}-\mathbf{x}_0\|\leq \|g\|_\infty/L_d.
	\end{eqnarray}
	Denote $s_d$ as the surface area of $d$-dimension unit ball, $s_d=dv_d$, then
	\begin{eqnarray}
		\|g\|_1&=&\int |g(\mathbf{x})|d\mathbf{x}\nonumber\\
		&\geq & \int_{\|\mathbf{x}-\mathbf{x}_0\|\leq \|g\|_\infty/L_d} (\|g\|_\infty - L_d\|\mathbf{x}-\mathbf{x}_0\|))d\mathbf{x}\nonumber\\
		&=&\int_0^{\|g\|_\infty/L_d}(\|g\|_\infty-L_dr)s_dr^{d-1}dr\nonumber\\
		&=&\frac{s_d}{d}\|g\|_\infty\left(\frac{\|g\|_\infty}{L_d}\right)^d - \frac{L_ds_d}{d+1}\left(\frac{\|g\|_\infty}{L_d}\right)^{d+1}\nonumber\\
		&=&\frac{v_d}{(d+1)L_d^d}\|g\|_\infty^{d+1}.
	\end{eqnarray}
	Hence
	\begin{eqnarray}
		\|g\|_\infty \leq \left(\frac{(d+1)L_d^d}{v_d}\|g\|_1\right)^\frac{1}{d+1},
	\end{eqnarray}
	\begin{eqnarray}
		\|g\|_2^2\leq \|g\|_1\|g\|_\infty\leq \left(\frac{(d+1)L_d^d}{v_d}\right)^\frac{1}{d+1}\|g\|_1^\frac{d+2}{d+1}.
	\end{eqnarray}
\end{proof}
Moreover, in \eqref{eq:opt}, the derivative of $g$ is bounded by $L$ in each dimension, hence $F[\eta]$ is $\sqrt{d}L$-Lipschitz. Let $L_d=\sqrt{d}L$. Since $\eta-\Delta$, $F[\eta_1]$ are both $L_d$-Lipschitz, $\eta-\Delta-F[\eta_1]$ is $2L_d$-Lipschitz, hence
\begin{eqnarray}
	\|F[\eta_1]-(\eta-\Delta)\|_2^2\lesssim \left(\frac{Tqh^d}{n_0}\right)^\frac{d+2}{d+1} + \left(\int_S d\mathbf{x}\right)^\frac{d+2}{d+1},
\end{eqnarray}
and 
\begin{eqnarray}
	\|F[\eta_1]-\eta\|_2^2&\leq& 2\|F[\eta_1]-(\eta-\Delta)\|_2^2+2\Delta^2\int_\mathcal{X}d\mathbf{x}\nonumber\\
	&\lesssim & \left(\frac{Tqh^d}{n_0}\right)^\frac{d+2}{d+1} + \left(\int_S d\mathbf{x}\right)^\frac{d+2}{d+1}+h^2+\frac{\ln N}{Nh^d}.
	\label{eq:etabound1}
\end{eqnarray}
From Lemma \ref{lem:largek},
\begin{eqnarray}
	\mathbb{E}\left[\int_S d\mathbf{x}\right] = \int_\mathcal{X} \text{P}(\mathbf{x}\in S)d\mathbf{x}\leq\int_\mathcal{X}\frac{2}{N^2}d\mathbf{x} \leq \frac{2}{f_mN^2}, 
\end{eqnarray}
\begin{eqnarray}
	\mathbb{E}\left[\left(\int_S d\mathbf{x}\right)^\frac{d+2}{d+1}\right]\leq \mathbb{E}[\int_S d\mathbf{x}]\left(\int_\mathcal{X}d\mathbf{x}\right)^\frac{1}{d+1}\leq \frac{1}{f_m^\frac{1}{d+1}}\mathbb{E}\left[\int_Sd\mathbf{x}\right]\lesssim \frac{1}{N^2}.
\end{eqnarray}
Therefore the second term in \eqref{eq:etabound1} decays faster than other terms, and
\begin{eqnarray}
	\|F[\eta_1]-\eta\|_2^2&\leq& 2\|F[\eta_1]-(\eta-\Delta)\|_2^2+2\Delta^2\int_\mathcal{X}d\mathbf{x}\nonumber\\
	&\lesssim & \left(\frac{Tqh^d}{n_0}\right)^\frac{d+2}{d+1} +h^2+\frac{\ln N}{Nh^d}.
	\label{eq:etabound}
\end{eqnarray}
The bound also holds for $\|F[\eta_2]-\eta\|_2^2$. Substitute $n_0$ in \eqref{eq:etabound} with \eqref{eq:n0}. The proof is complete.

\section{Additional Numerical Experiments}\label{sec:numadd}
In this section, we show some additional numerical experiments. In particular, section \ref{sec:uci} shows some experiments on real data from UCI repository. Section \ref{sec:profile} shows the comparison of the profile of estimated functions using synthesized data.

\subsection{Experiments on Real Data}\label{sec:uci}
We use the following datasets from UCI repository:

(1) Auto MPG, in which the 'displacement' column is the target;

(2) Abalone, in which the 'rings' column is the target, which represents the age;

(3) Wine quality. It includes two datasets, corresponding to red and white wine, respectively. For both two datasets, the 'quality' column is the target;

(4) Liver disorders. The 'drinks' column is the target.

For each dataset, we run experiments using simple kernel regression, median-of-means, trimmed means, initial estimator and the corrected one respectively. For the convenience of computation and comparison, the data are all scaled to $[0,1]$. Each dataset is splitted into train and test datasets, in which the train dataset has $90\%$ samples, while the test dataset has $10\%$. Among the training dataset, about $10\%$ of samples are attacked, i.e. $q=[0.1N_{train}]$, in which $[\cdot]$ rounds the value to integer, $q$ and $N_{train}$ denotes the number of attacked samples and all training samples, respectively. 

We use two attack strategies separately. The first one is random attack, which means that the attacked samples are randomly selected from the whole training dataset, and the label value after attack just follows a simple normal distribution with $\mu=0, \sigma = 10$. The other one is concentrated attack, which me The performances of these methods are evaluated using root mean squared error (RMSE). Results are shown in the following table.

\begin{table}[h!]
	\begin{center}
		\begin{tabular}{c|c|c|c|c|c|c|c} 
			\hline
			\textbf{Dataset} & \textbf{Attack} & \textbf{Clean } &
			\textbf{Kernel} &
			\textbf{Median} &
			\textbf{Trimmed} &
			\textbf{Initial} &
			\textbf{Corrected} \\
			& \textbf{strategy} & \textbf{data} &\textbf{regression}&\textbf{of means}&
			\textbf{means}&\textbf{estimator}&\textbf{estimator}\\
			\hline
			AutoMPG & random & 0.132 & 0.909 & 0.227 & 0.182 & 0.180 & 0.162\\
			AutoMPG & concentrated & 0.132 & 1.401 & 1.079 & 0.927 & 0.664 & 0.579\\
			Abalone & random & 0.101 & 0.115 & 0.141 & 0.103 & 0.104 & 0.104\\
			Abalone & concentrated & 0.101 & 1.306 & 1.391 & 0.605 & 0.278 & 0.270\\
			Wine(Red) & random & 0.132 & 0.198 & 0.209 & 0.136 & 0.134 & 0.120\\
			Wine(Red) & concentrated & 0.132 & 1.434 & 1.515 & 0.433 & 0.222 & 0.205 \\
			Wine(White) & random & 0.141 & 0.179 & 0.151 & 0.166 & 0.141 & 0.141\\
			Wine(White) & concentrated & 0.141 & 1.244 & 1.210 & 0.240 & 0.210 & 0.196\\
			Liver Disorders & random & 0.202 & 0.472 & 0.444 & 0.200 & 0.199 & 0.199\\
			Liver Disorders & concentrated & 0.202 & 0.283 & 0.234 & 0.203 & 0.203 & 0.205\\
			\hline
		\end{tabular}
		\caption{Results on numerical experiments from UCI repository. The 'clean data' column indicates the RMSE value of simple kernel regression with unattacked training dataset.}
		\label{tab:table1}
	\end{center}
\end{table}

From table \ref{tab:table1}, it can be observed that our initial estimator performs much better than simple kernel regression, as well as traditional robust methods like median-of-means and trimmed mean. The correction technique makes the performance can further improve the performance. Under random attack, the improvement of our methods over trimmed mean is not obvious. Our improvement is more significant under concentrated attack. This finding agrees with our theory. As has been discussed in the main paper, the drawback of trimmed mean is that the trim fraction can not be adjusted appropriately among the whole support, if the attack is concentrated in a small region. Therefore, such drawback is crucial only if the attacked samples are concentrated.

\subsection{Profile Comparison}\label{sec:profile}
For each case, we generate $N=10000$ training samples, with each sample follows uniform distribution in $[0,1]$. The ground truth is
\begin{eqnarray}
	\eta(x) = \sin(2\pi x).
\end{eqnarray}
The noise follows standard Gaussian distribution $\mathcal{N}(0,1)$. The kernel function is
\begin{eqnarray}
	K(u) = 2-|u|, \forall |u|\leq 1.
\end{eqnarray}
The hyperparameters are $h=0.03$, $T=1$ and $M=3$.

For all cases, let $q=1000$. The estimated function with kernel estimator $\hat{\eta}_{NW}$, initial estimator $\hat{\eta}_0$ and corrected estimator $\hat{\eta}_c$ under these three types of attack are shown in Figure \ref{fig:random}, \ref{fig:onedir} and \ref{fig:centralize}, respectively, in which the blue solid curve is the estimated function, while the orange dashed curve is the ground truth.
\begin{figure}[h!]
	\centering
	\begin{subfigure}{0.32\textwidth}
		\includegraphics[width=\textwidth,height=0.8\textwidth]{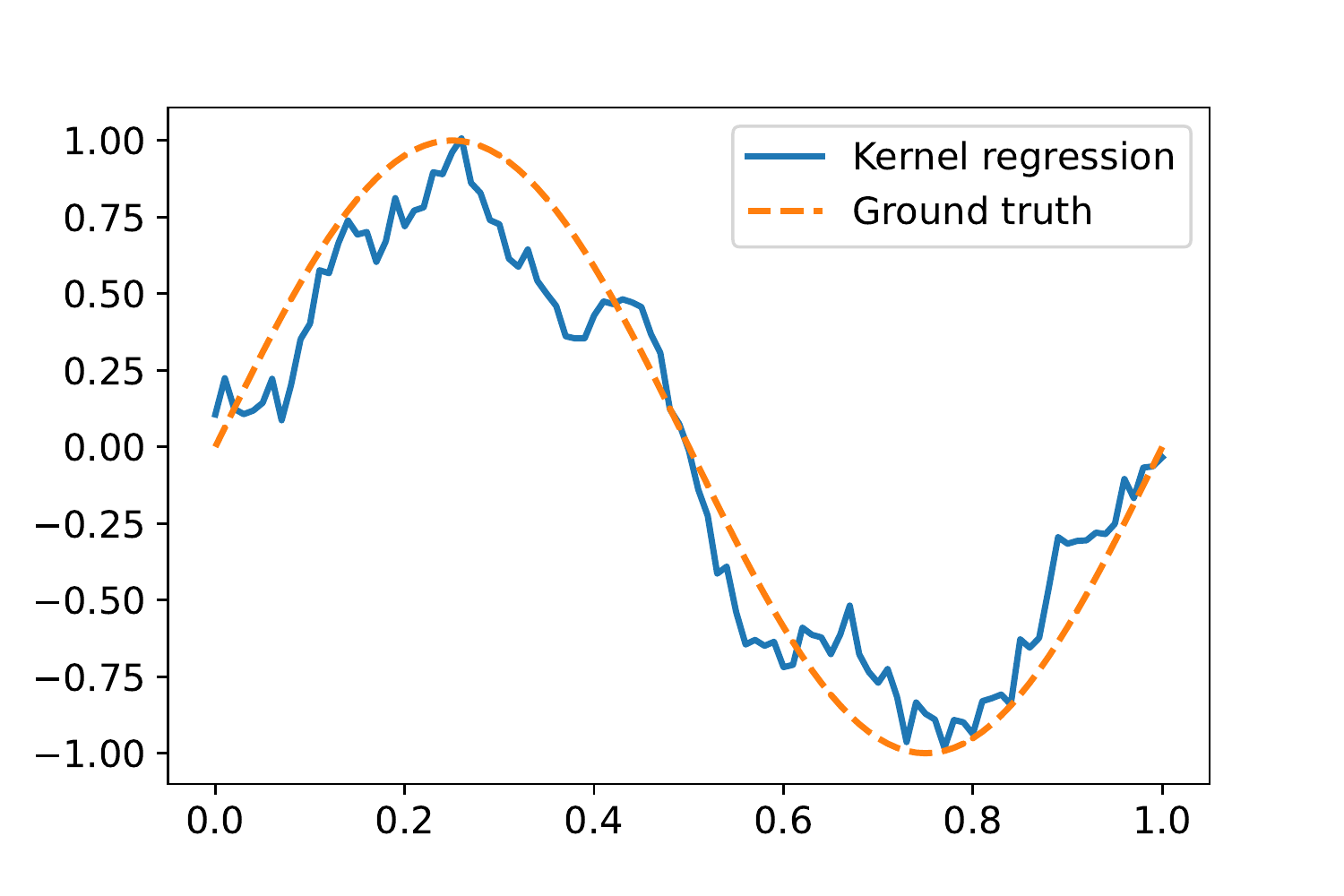}
		\caption{Kernel regression.}
	\end{subfigure}
	\hfill
	\begin{subfigure}{0.32\textwidth}
		\includegraphics[width=\textwidth,height=0.8\textwidth]{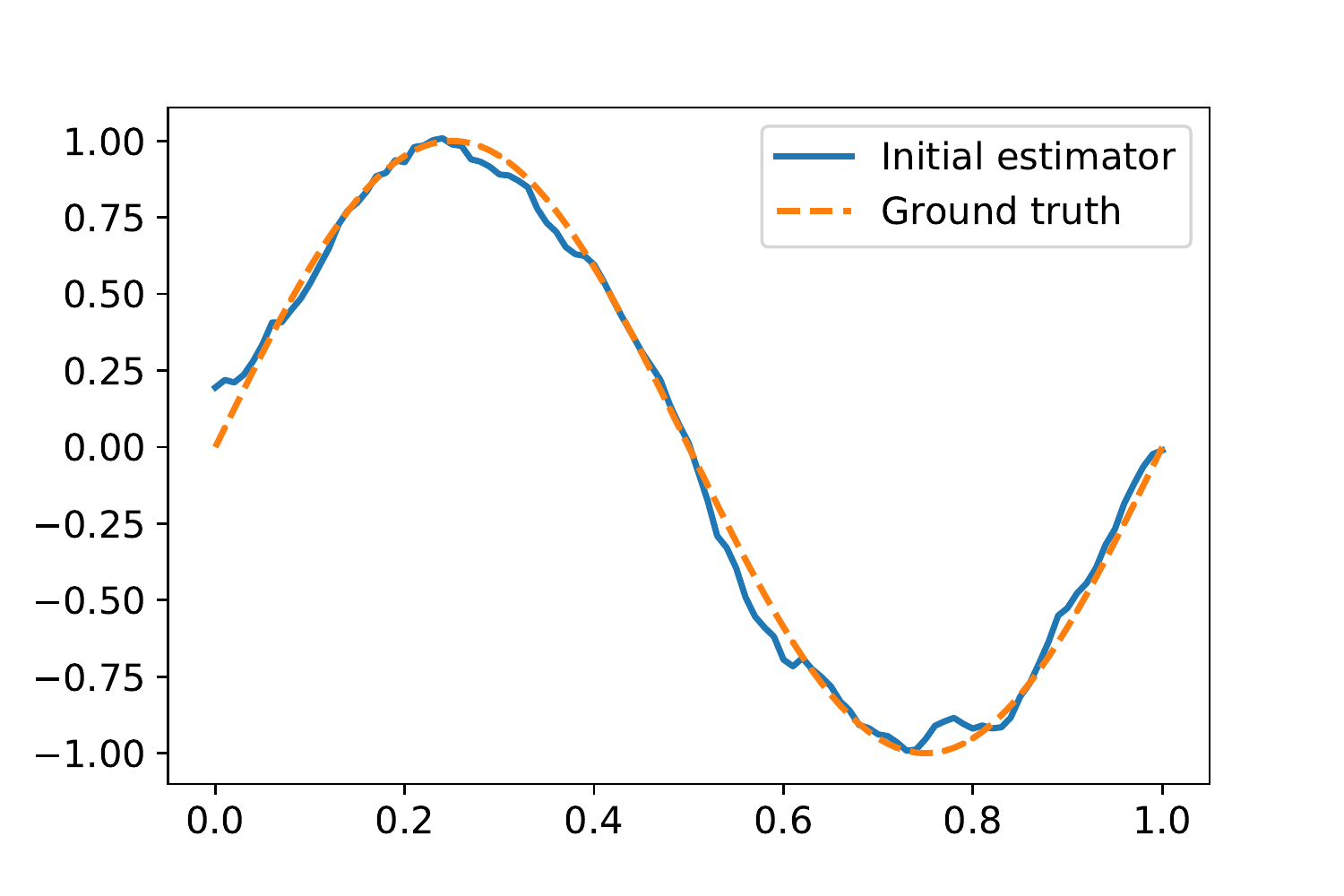}
		\caption{Initial estimator.}
	\end{subfigure}
	\hfill
	\begin{subfigure}{0.32\textwidth}
		\includegraphics[width=\textwidth,height=0.8\textwidth]{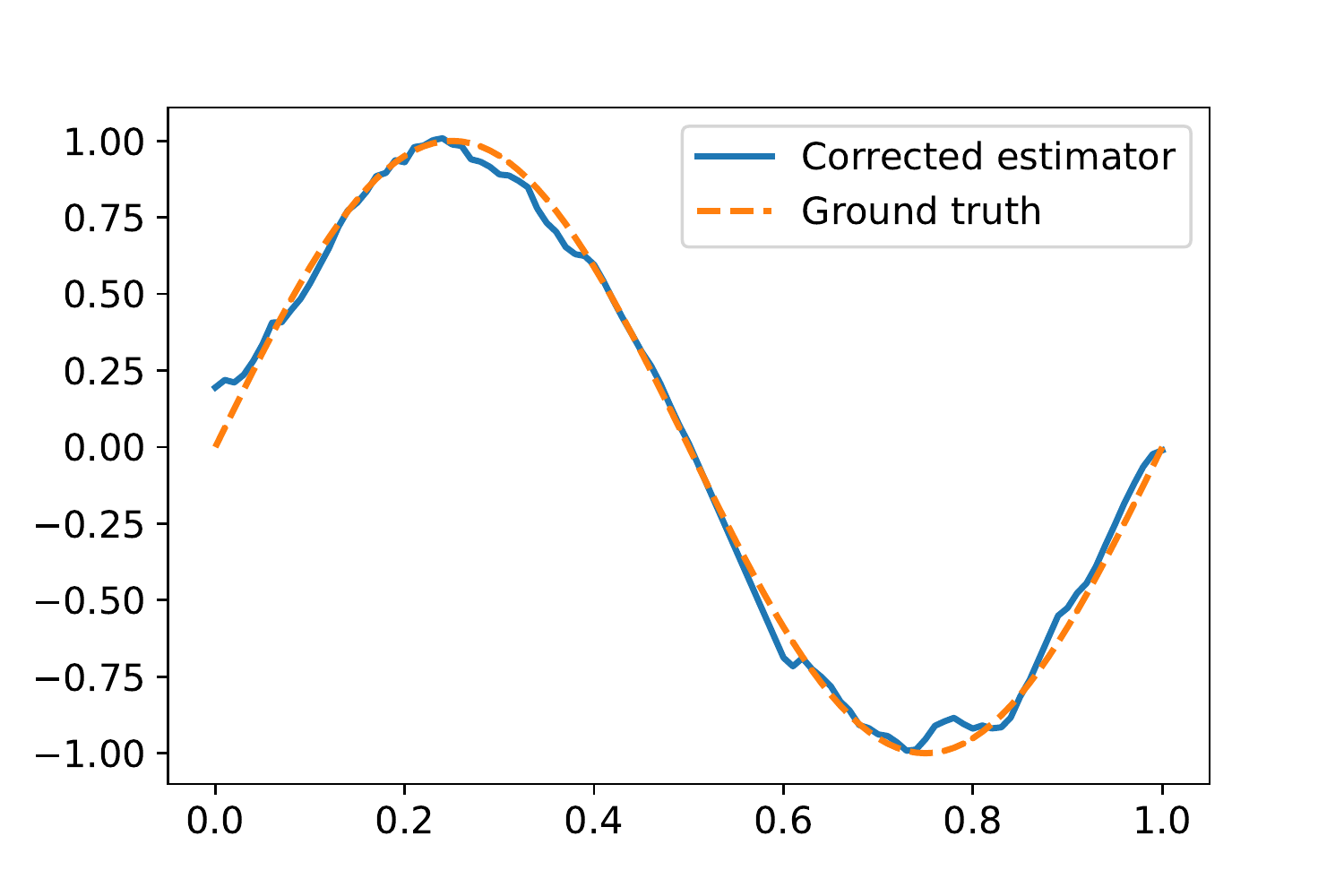}
		\caption{Corrected estimator.}
	\end{subfigure}
	
	\caption{Performance under random attack.}
	\label{fig:random}
\end{figure}

\begin{figure}[h!]
	\centering
	\begin{subfigure}{0.32\textwidth}
		\includegraphics[width=\textwidth,height=0.8\textwidth]{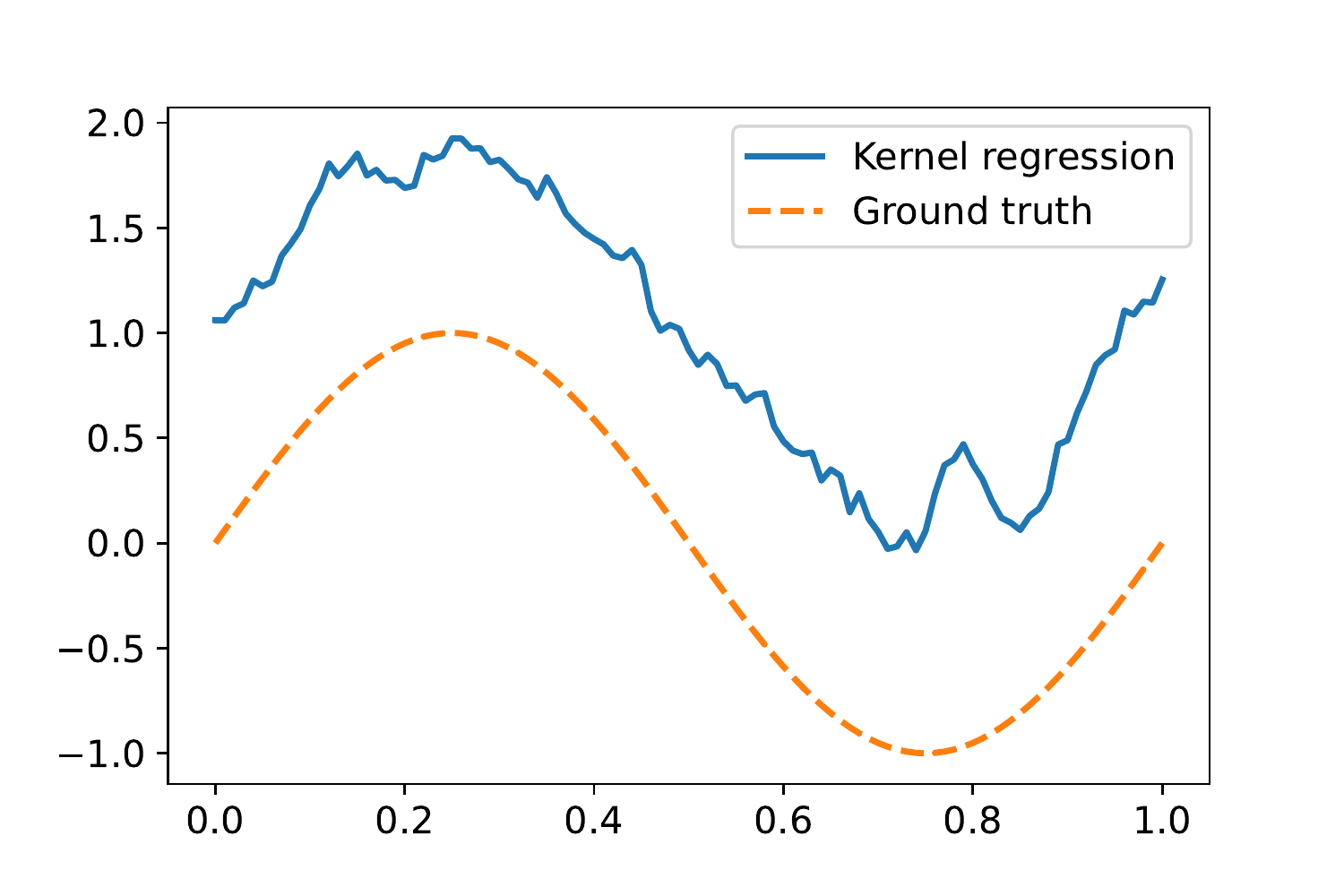}
		\caption{Kernel regression.}
	\end{subfigure}
	\hfill
	\begin{subfigure}{0.32\textwidth}
		\includegraphics[width=\textwidth,height=0.8\textwidth]{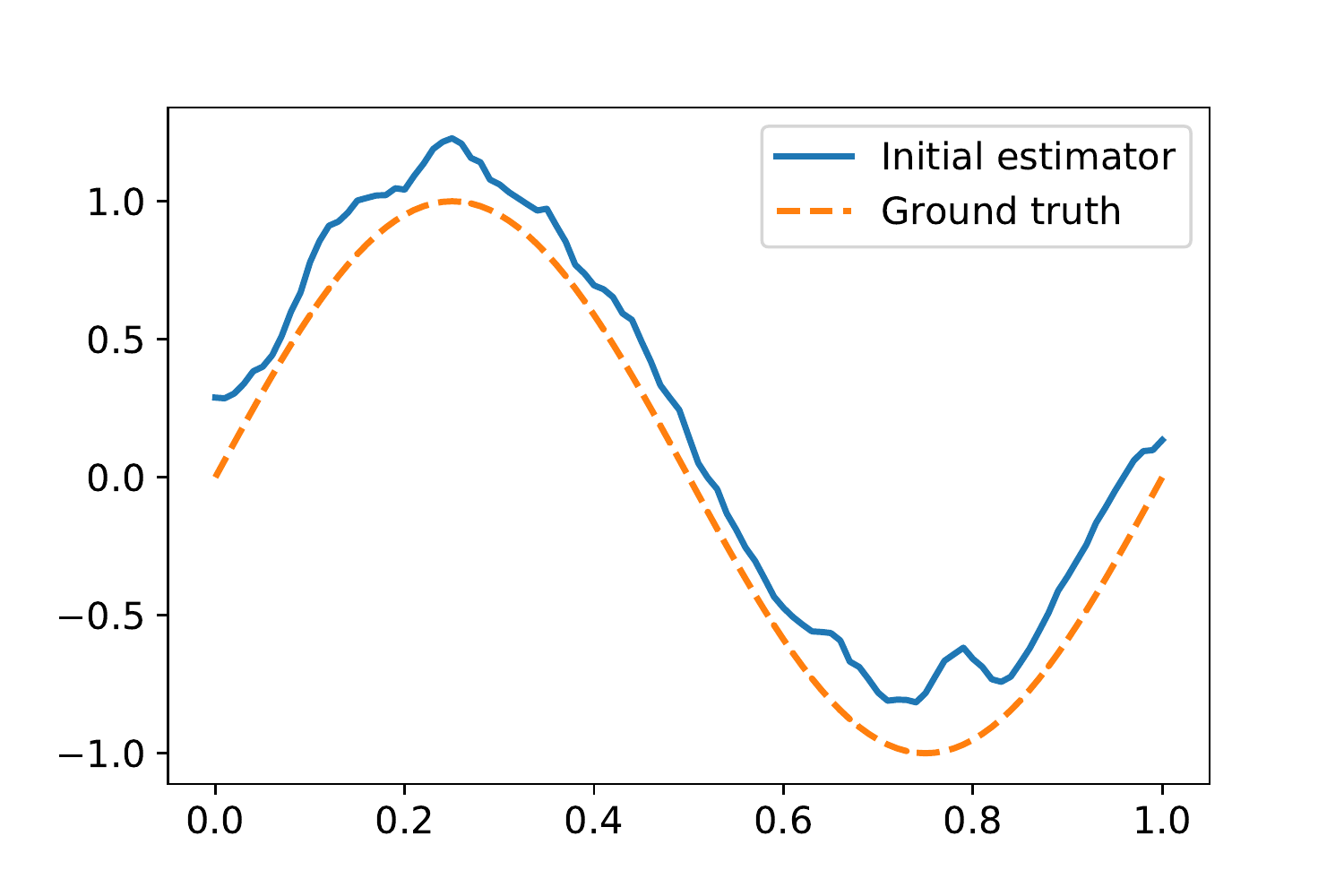}
		\caption{Initial estimator.}
	\end{subfigure}
	\hfill
	\begin{subfigure}{0.32\textwidth}
		\includegraphics[width=\textwidth,height=0.8\textwidth]{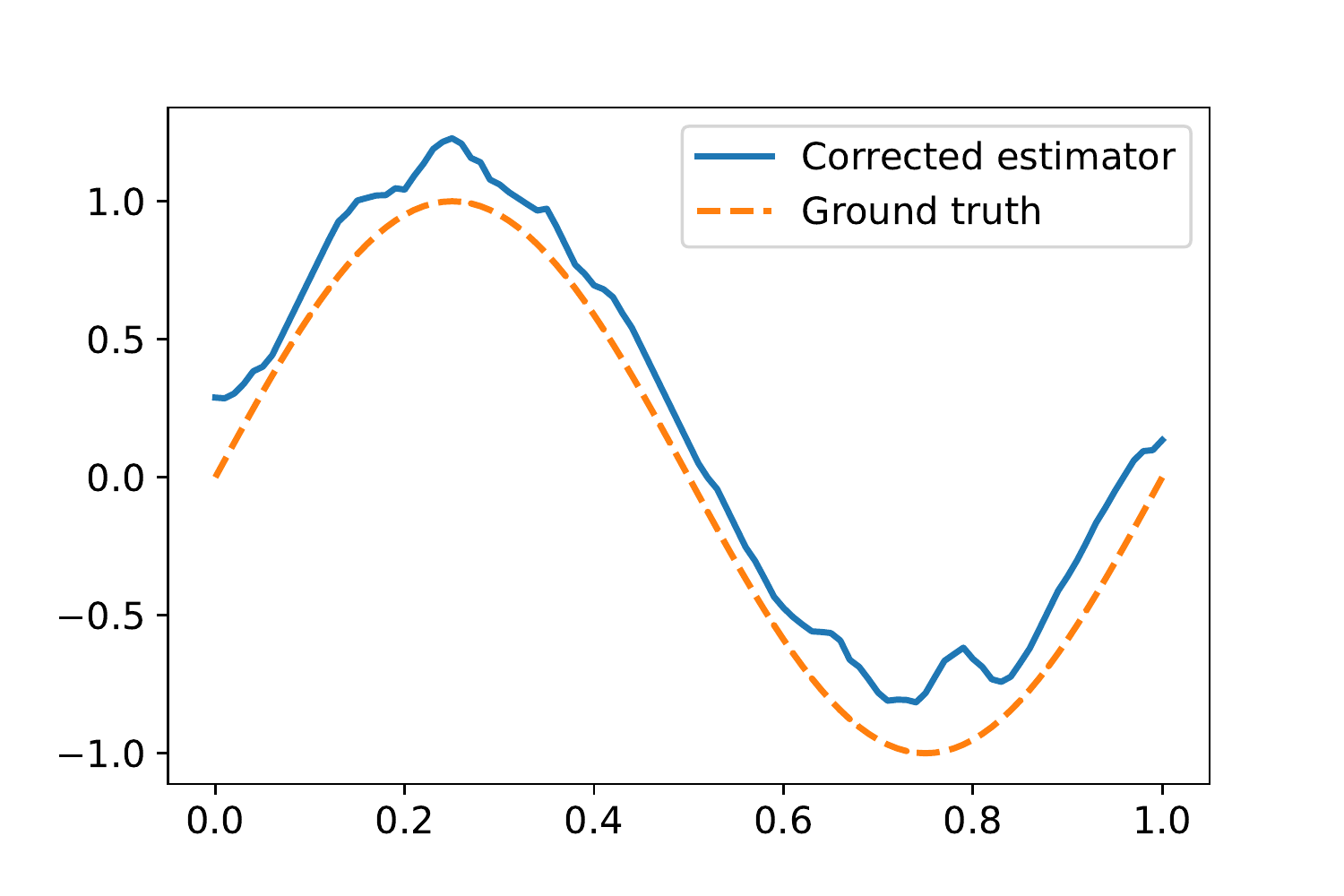}
		\caption{Corrected estimator.}
	\end{subfigure}
	
	\caption{Performance under one direction attack.}
	\label{fig:onedir}
\end{figure}
\begin{figure}[h!]
	\centering
	\begin{subfigure}{0.32\textwidth}
		\includegraphics[width=\textwidth,height=0.8\textwidth]{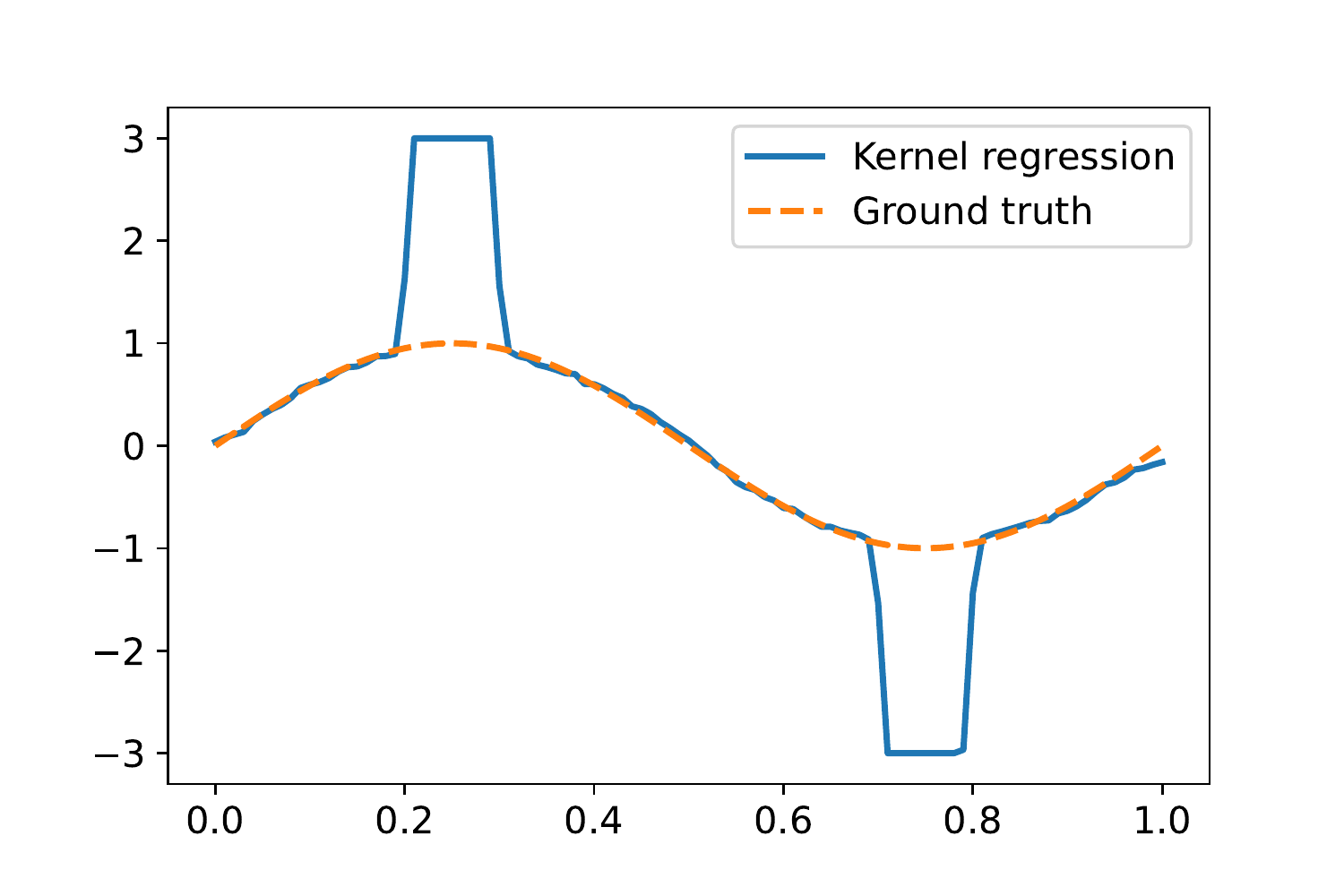}
		\caption{Kernel regression.}
		\label{fig:cent_ker}
	\end{subfigure}
	\hfill
	\begin{subfigure}{0.32\textwidth}
		\includegraphics[width=\textwidth,height=0.8\textwidth]{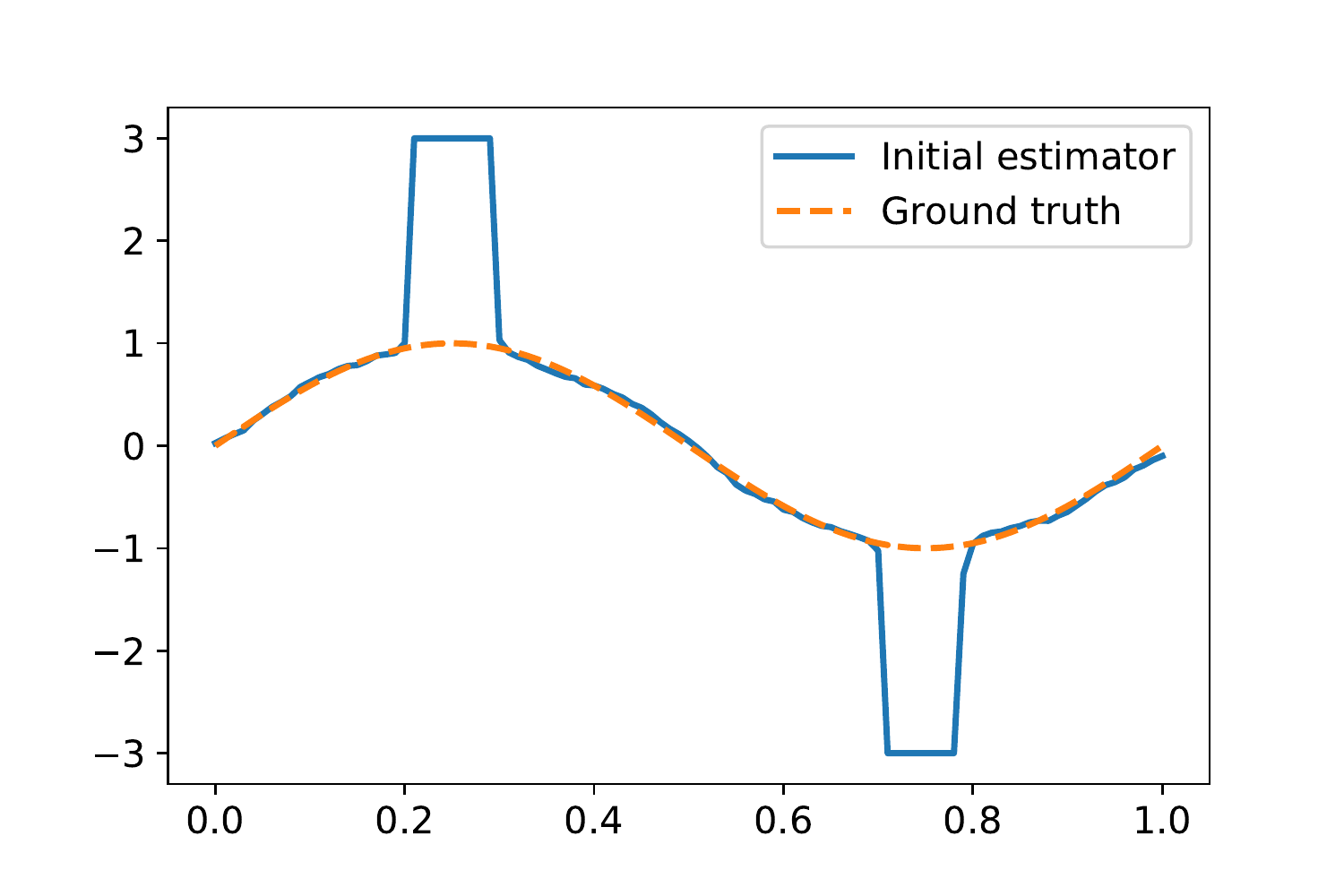}
		\caption{Initial estimator.}
		\label{fig:cent_init}
	\end{subfigure}
	\hfill
	\begin{subfigure}{0.32\textwidth}
		\includegraphics[width=\textwidth,height=0.8\textwidth]{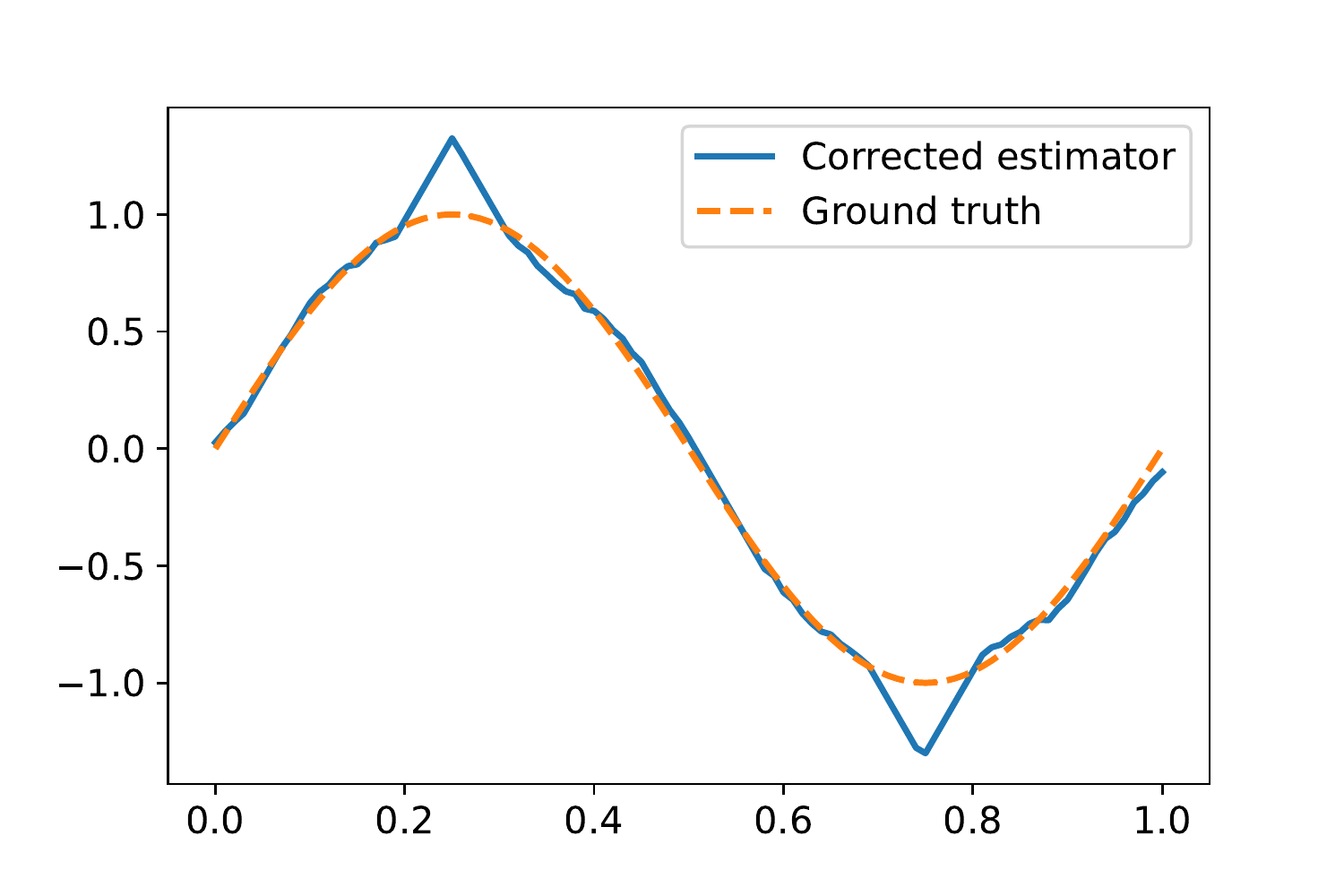}
		\caption{Corrected estimator.}
		\label{fig:cent_corr}
	\end{subfigure}
	
	\caption{Performance under concentrated attack.}
	\label{fig:centralize}
\end{figure}

From Figure \ref{fig:random}, it can be observed that the original kernel regression is badly affected by random attack, while the estimator \eqref{eq:eta_app} successfully withstand the malicious samples and fits the ground truth well. Figure \ref{fig:onedir} shows that with attacks in one direction, the kernel regression has large bias, while \eqref{eq:eta_app} handles it well. However, the initial estimator \eqref{eq:eta_app} no longer performs well under concentrated attack. As is already discussed in the main paper, although $q/N$ is small, the proportion of attacked samples around $0.25$ and $0.75$ are high. As a result, clean samples are not enough to beat attacked samples, and the estimated function value can deviate far away from its ground truth. Figure \ref{fig:centralize} shows the profile of estimated function with attacks concentrate around $0.25$ and $0.75$. In this case, the initial estimator gives a large spike at these two points. These two spikes are successfully removed by the corrected estimator \eqref{eq:optim}.

\end{document}